\documentclass[11pt]{amsart}

\usepackage{fullpage}

\usepackage{amsmath, amscd, amsthm, amssymb}

\def\C{{\mathbf C}}
\def\R{{\mathbf R}}
\def\Z{{\mathbf Z}}
\def\Q{{\mathbf Q}}
\def\A{{\mathbf A}}
\def\Qbar{{\overline{\mathbf{Q}}}}

\newtheorem{theorem}{Theorem}[section]
\newtheorem{thm}[theorem]{Theorem}
\newtheorem{mthm}[theorem]{Main Theorem}
\newtheorem{lemma}[theorem]{Lemma}
\newtheorem{lem}[theorem]{Lemma}
\newtheorem{proposition}[theorem]{Proposition}
\newtheorem{prop}[theorem]{Proposition}
\newtheorem{corollary}[theorem]{Corollary}
\newtheorem{coro}[theorem]{Corollary}

\theoremstyle{definition}
\newtheorem{definition}[theorem]{Definition}
\newtheorem{defin}[theorem]{Definition}
\newtheorem{hypothesis}[theorem]{Hypothesis}

\theoremstyle{remark}
\newtheorem{remark}[theorem]{Remark}

\newcommand{\mm}[4]{\(\begin{smallmatrix} #1 & #2\\ #3 & #4\end{smallmatrix}\)}
\newcommand{\mb}[4]{\(\begin{array}{cc} #1 & #2\\ #3 & #4\end{array}\)}

\newcommand{\Hom}{\operatorname{Hom}}

\DeclareMathOperator{\gal}{Gal}
\DeclareMathOperator{\nm}{Nm}
\DeclareMathOperator{\ord}{ord}
\DeclareMathOperator{\tr}{Tr}
\DeclareMathOperator{\ind}{Ind}

\DeclareMathOperator{\GU}{GU}
\DeclareMathOperator{\U}{U}
\DeclareMathOperator{\SU}{SU}
\DeclareMathOperator{\SL}{SL}
\DeclareMathOperator{\GL}{GL}
\DeclareMathOperator{\PGL}{PGL}

\DeclareMathOperator{\charf}{char}
\DeclareMathOperator{\res}{Res}

\DeclareMathOperator{\diag}{diag}

\renewcommand{\(}{\left(}
\renewcommand{\)}{\right)}
\newcommand{\set}[1]{\left\lbrace#1\right\rbrace}
\newcommand{\gen}[1]{\left\langle #1\right\rangle}
\newcommand{\ol}[1]{\overline{#1}}
\newcommand{\inj}{\hookrightarrow}

\newcommand{\iso}{\stackrel{\sim}{\rightarrow}}

\begin{document}

\title{A class number formula for Picard modular surfaces}

\author{Aaron Pollack}
\address{Department of Mathematics, Institute for Advanced Study, Princeton, NJ 08540, USA}\email{aaronjp@math.ias.edu}
\author{Shrenik Shah}
\address{Department of Mathematics, Columbia University, New York, NY 10027, USA}\email{snshah@math.columbia.edu}

\subjclass[2010]{11F67 (primary); 11G18, 11G40, 14G35, 19E15 (secondary)}

\thanks{A.P.\ has been partially supported by NSF grant DMS-1401858 and by the Schmidt fund at the Institute for Advanced Study.  S.S.\ has been supported by NSF grant DMS-1401967.}

\begin{abstract}
We investigate arithmetic aspects of the middle degree cohomology of compactified Picard modular surfaces $X$ attached to the unitary similitude group $\mathrm{GU}(2,1)$ for an imaginary quadratic extension $E/\mathbf{Q}$.  We construct new Beilinson--Flach classes on $X$ and compute their Archimedean regulator.  We obtain a special value formula involving a non-critical $L$-value of the degree six standard $L$-function, a Whittaker period, and the regulator.  This provides evidence for Beilinson's conjecture in this setting.
\end{abstract}
\maketitle

\tableofcontents

\section{Introduction}

Motivated by Beilinson's Archimedean regulator conjecture, we study the arithmetic of the middle degree cohomology of the canonically compactified Picard modular surfaces attached to quasi-split unitary similitude groups of signature $(2,1)$.  The $L$-function associated to this cohomology has ``trivial zeroes'' at integer points to the left of the line of symmetry that can be computed in terms of the $\Gamma$-factors of the completed $L$-function.  Beilinson conjectured that the order of vanishing at these points measures the size of an integral space inside of motivic cohomology, and gave a conjectural formula for the leading term of the Taylor expansion at these values.  His conjecture represents a far-reaching generalization of the class number formula for the residue of the Dedekind $\zeta$-function at $s=1$.

To state his conjecture precisely, Beilinson \cite{beilinson} constructed a regulator map from motivic cohomology -- represented here concretely using Bloch's higher Chow groups -- to Deligne cohomology, which in our setting is closely related to the dual of the space of $d$-closed smooth differential $(1,1)$-forms on the surface modulo coboundaries.  He predicted that classes in motivic cohomology generate a $\mathbf{Q}$-structure in Deligne cohomology via the regulator.

Deligne cohomology carries a natural $\mathbf{R}^\times/\mathbf{Q}^\times$-valued volume form coming from comparison to Betti and de Rham cohomology.  Beilinson further conjectured that the covolume of the motivic $\mathbf{Q}$-structure with respect to this form is equal to the leading term of the $L$-function of the motive at a corresponding integer point up to $\mathbf{Q}^\times$.  In our case, the relevant motivic $L$-function has been decomposed into a product of automorphic $L$-functions in \cite[p.\ 291, Theorem A]{pms}.  The main contribution comes from degree 6 standard $L$-functions of certain cuspidal automorphic representations on the group $\GU(2,1)$.

We study this conjecture in our case by constructing suitable classes in the higher Chow group using the theory of Siegel units and explicitly computing the pairing of their regulator with certain algebraic differential $(1,1)$-forms.  A weak form of our main theorem may be briefly described as follows.  More precise definitions and statements may be found later in this introduction.
\begin{mthm} \label{mthm:main}
Let $E/\mathbf{Q}$ be a quadratic extension and let $\GU(2,1)$ be the unitary similitude group of a Hermitian space of signature $(2,1)$ over $E$.  Let $\pi$ be a generic cuspidal automorphic representation on $\GU(2,1)$ that contributes to the middle-degree cohomology of the compactified Picard modular surface $X_{/E}$ and whose central character $\omega_\pi$ is trivial at infinity.

Assume that the twisted base change of $\pi$ to $\mathrm{Res}_\mathbf{Q}^E (\GL_3 \times \mathbf{G}_m)$ is cuspidal.  Let $\nu$ be an even Dirichlet character (viewed as an automorphic character of $\mathbf{G}_{m/\mathbf{Q}}$) such that $\nu^2\cdot(\omega_\pi|_{Z(\mathbf{A})})$ is non-trivial, where $Z \subseteq \GU(2,1)$ is a diagonally embedded $\mathbf{G}_{m/\mathbf{Q}}$.  Let $E(\nu)$ denote the extension generated by the values of $\nu$.  Then after possibly increasing the level of $X$, there exists a class $\xi$ in the motivic cohomology group $H_\mathcal{M}^3(X_{/E(\nu)},\ol{\mathbf{Q}}(2))$ and an algebraic differential $(1,1)$-form $\omega$ associated to $\pi$ such that
\begin{enumerate}
	\item the Archimedean regulator pairing $\langle \xi,\omega \rangle \ne 0$ and
	\item moreover, there exists an explicit constant $W(\pi) \in \mathbf{C}^\times$ such that
	\[\langle \xi,\omega \rangle =_{\ol{\mathbf{Q}}^\times} W(\pi) L'(\pi \times \nu, \mathrm{Std},0),\]
	where $L'(\pi \times \nu, \mathrm{Std},0)$ is the derivative at 0 of the standard $\nu$-twisted degree 6 $L$-function of $\pi$ and $=_{\ol{\mathbf{Q}}^\times}$ indicates equality up to an algebraic multiple.
\end{enumerate}
\end{mthm}

\subsection{Motivation}

One primary motivation is to obtain evidence in support of Beilinson's conjecture.  Beilinson predicts a formula for the leading term of the Taylor expansion of a motivic $L$-function at an integer point to the left of the center of symmetry with respect to the functional equation.  When this point is exactly $\frac{1}{2}$ to the left of the center, one calls the value near-central.  The conjecture takes a slightly different shape at such values -- it includes a contribution from cycles.  In particular, if $X$ is a smooth proper algebraic surface defined over $\mathbf{Q}$, Beilinson defines a morphism
\[\mathrm{Reg}_X: H_\mathcal{M}^3(X_{/\mathbf{Q}},\mathbf{Q}(2)) \oplus \mathrm{N}^1(X) \rightarrow H_\mathcal{D}^3(X_{/\mathbf{R}},\mathbf{R}(2)),\]
where $\mathrm{N}^1(X)$ is the group of $\mathbf{Q}$-linear combinations of codimension 1 algebraic cycles on $X$ modulo homological equivalence \cite[\S 5]{schneider}.

The group $\mathrm{N}^1(X)$ accounts for the portion of the middle-degree cohomology whose $L$-function has poles $\frac{1}{2}$ to the right of the center of symmetry.  In the case of the Picard modular surface, this contribution comes exactly from endoscopic generic cuspidal automorphic representations contributing to middle-degree cohomology.  Blasius and Rogawski \cite{br} have addressed this part of Beilinson's conjecture, so we focus here on the contribution from non-endoscopic forms.  However, in a sense described in Section \ref{subsubsec:intrononvanish}, endoscopic forms give rise to a local obstruction to non-vanishing of the regulator pairing.

Another motivation is to obtain new Beilinson-Flach classes in the sense of the recent line of works of Kings, Lei, Loeffler, Skinner, and Zerbes constructing Euler systems to study motivic $L$-functions and Iwasawa theory \cite{llz,llz2,klz,lz,lsz}.  Our classes -- the $\xi$ in Main Theorem \ref{mthm:main} -- should be the bases of Euler systems that may be used to study the motives appearing in the middle degree cohomology of Picard modular surfaces.  Part (1) of this theorem ensures that the classes forming these Euler systems will be non-zero in motivic cohomology.

\subsection{Relation to previous work}

Beilinson, following a computation of Bloch, proved a class number formula for the products of modular curves \cite{beilinson}.  He also used a cup product of two Siegel units to prove a similar formula for the modular curve itself \cite{beilinson2}.  These works have been refined and extended by many authors, such as Schappacher--Scholl \cite{ss}, Brunault \cite{brunault}, and Brunault--Chida \cite{brunaultchida}.

Ramakrishnan \cite{ramakrishnan,ramakrishnan2} and Kings \cite{kings} studied the case of Hilbert modular surfaces associated to $\mathrm{Res}_\mathbf{Q}^F \GL_2$ for a quadratic extension $F/\mathbf{Q}$.  More recently, Lemma \cite{lemma,lemma2} considered the case of Siegel modular threefolds.

Beilinson \cite{beilinson,beilinson2} and Schappacher--Scholl \cite{ss} studied the question of extending classes to the compactified variety, as did Ramakrishnan \cite{ramakrishnan2}.  This has typically been achieved only when working with the cohomology of the variety itself rather than of an automorphic vector bundle.  We are in the same setting, and have developed an extension to the boundary in our case as well.  One part of our strategy -- the application of a suitable Hecke operator to eliminate the contribution of the boundary -- is based on an old idea of Manin and Drinfel'd.  The other part, which addresses a difficulty not present in previous cases (explained in Section \ref{subsubsec:introgeom}), uses a computation in coordinates on the Picard modular surface.  We were guided in this by work of Cogdell \cite{cogdell}.

One entirely new feature in the present work is that the non-vanishing of the regulator becomes more subtle for $\GU(2,1)$ than for $\GL_2 \times \GL_2$ or $\mathrm{Res}_\mathbf{Q}^F \GL_2$.  This issue is connected to a distinction question, and forces us to add a local hypothesis in Theorem \ref{coarseThm} below.

Our construction is based on an integral of Rankin-Selberg type that was first studied by Gelbart and Piatetski-Shapiro \cite{gelbartPS}.  Jacquet \cite{jacquet2} fully established all the local properties of the Rankin-Selberg constructions on $\GL_2 \times \GL_2$ and $\mathrm{Res}_\mathbf{Q}^F \GL_2$.  By contrast, less is known about the Gelbart--Piatetski-Shapiro construction \cite{gelbartPS}.  However, many essential properties have been established by Baruch \cite{baruch} and Miyauchi \cite{my1,my2,my3,my4}, which we leverage in Sections \ref{subsec:locrationality} and \ref{subsec:finer} below.  It turns out that although the Bump--Friedberg integral \cite{bf} yields a different local $L$-function for $\GL_3$ than ours at split places, the detailed work of Matringe \cite{matringe,matringe2} and a computation of Miyauchi--Yamauchi \cite{my} may be used to deduce facts about our construction as well.

\subsection{Algebraic differential forms} \label{subsec:algdiff}

Fix a quadratic extension $E$ of $\mathbf{Q}$, and let $-D \in \mathbf{Z}_{<0}$ denote the corresponding fundamental discriminant.  Let $G = \GU(2,1)$ be the quasi-split unitary similitude group over $\mathbf{Q}$ stabilizing the Hermitian form
\[J = \(\begin{array}{ccc}  & & \delta^{-1} \\ & 1 & \\ -\delta^{-1} & &  \end{array}\)\]
of signature $(2,1)$ on an $E$-vector space $V$ of dimension 3.  Here, $\delta = \sqrt{-D}$.  We write $\mu: G \rightarrow \mathbf{G}_m$ for the similitude map.  We let $\mathbf{A}$ denote the adeles of $\mathbf{Q}$ and write $\mathbf{A}_f$ for the finite adeles.  Let $X$ be the space of negative lines in $V \otimes_E \mathbf{C}$ and write $K_f$ for an open compact subgroup of $G(\mathbf{A}_f)$.  The main geometric object of interest is the Shimura variety $S_{K_f} = G(\mathbf{Q}) \backslash X \times G(\mathbf{A}_f) / K_f$, which is an algebraic surface defined over $E$.  (See Section \ref{sec:higherchow} for more details.)

The cohomology of $S_{K_f}$ in degree 2 is a motive.  Its \'etale realization has been determined in \cite[pp.\ 291-3, Theorems A and B]{pms}.  The main contribution to the middle-degree cohomology -- that of Theorem B.(1) of {\it loc.\ cit.}\ -- is parametrized by $L$-packets that contain a unique generic cuspidal automorphic representation $\pi$ of $G(\mathbf{A})$.  If one factors $\pi = \pi_\infty \otimes \pi_f$, the relevant automorphic representations have $\pi_\infty$ a discrete series representation of Blattner parameter $(1,-1)$.  Blasius, Harris, and Ramakrishnan \cite{bhr} show that $\pi_f$ may be concretely realized inside the direct limit of interior coherent cohomology groups of the Shimura varieties $S_{K_f}$.  This yields a representation of $G(\mathbf{A}_f)$ on a $\ol{\mathbf{Q}}$-vector space $V_{\ol{\mathbf{Q}}}$ with the property that $V_{\ol{\mathbf{Q}}} \otimes_{\ol{\mathbf{Q}}} \mathbf{C} \cong \pi_f$.

One can naturally associate harmonic differential $(1,1)$-forms on $S_{K_f}$ to such a $\pi$.  Once a certain choice at infinity has been made (in Definition \ref{defin:v0}), and using our concrete realization of these forms in coherent cohomology, these differential forms are naturally associated to vectors $\varphi_f$ in the space of $\pi_f$.  We denote them by $\omega_{\varphi_f}$, and we say that such a form is algebraic if $\varphi_f \in V_{\ol{\mathbf{Q}}}$.

\subsection{Higher Chow classes} \label{subsec:introhigherchow}

The stabilizer of the vector ${}^t(0,1,0) \in V$ is the group $\GU(1,1) \boxtimes \GU(1)$, where $\boxtimes$ denotes the subgroup where the similitudes are equal, and $\GU(1,1)$ is the unitary similitude group for the form $J_2 = \mm{}{-\delta^{-1}}{\delta^{-1}}{}$.  Write $H \subseteq \GU(1,1) \boxtimes \GU(1)$ for the subgroup whose first factor additionally satisfies $\det=\mu$.  Then $H$ is identified with $\GL_2 \boxtimes \mathrm{Res}_\mathbf{Q}^E \mathbf{G}_m$, where $\boxtimes$ now means that $\det(h_1)=\mathrm{Nm}_\mathbf{Q}^E(h_2)$.  (Here, $h_i$ denotes the $i^\mathrm{th}$ projection.)

Let $K_{H,f} = K_f \cap H(\mathbf{A}_f)$.  We have a morphism from the Shimura variety $C_{K_{H,f}}$ of $H$ to $S_{K_f}$.  There is a natural projection $H \rightarrow \GL_2$ to its first factor, which yields a morphism of the Shimura variety $C_{K_{H,f}}$ to a classical modular curve $X_{K_{2,f}}$, where $K_{2,f} = K_{H,f} \cap \GL_2(\mathbf{A}_f)$.  We have a large source of invertible holomorphic functions on $X_{K_{2,f}}$ -- the Siegel units.  We obtain functions on $C_{K_{H,f}}$ by pullback, which we again call Siegel units.

Beilinson's conjecture is stated for the motives associated to a projective smooth variety.  For this reason, we consider the canonical smooth compactification $\ol{S}_{K_f}$ of $S_{K_f}$.  The variety $S_{K_f}$ has a minimal compactification $S_{K_f}^\mathrm{BB}$ -- the Baily-Borel compactification -- obtained by adjoining a finite set of cusps to $S_{K_f}$.  The variety $\ol{S}_{K_f}$ maps to $S_{K_f}^\mathrm{BB}$, and the preimages of the cusps are CM elliptic curves.

Using the Siegel units on $C_{K_{H,f}}$ and the theory of meromorphic functions on CM elliptic curves, we establish in Section \ref{sec:higherchow} a way to produce elements of higher Chow groups.  Roughly speaking, these are represented by formal sums $\sum_i (C_i,u_i)$ of pairs of curves $C_i \subseteq \ol{S}_{K_f}$ and rational functions $u_i$ on $C_i$ such that $\sum_i \mathrm{div}(u_i) = 0$.  This latter condition is a cancellation of poles and zeroes, so it can be thought of as saying that $\sum_i (C_i,u_i)$ defines a higher ``unit'' on $\ol{S}_{K_f}$.

If $\omega$ is a $d$-closed $(1,1)$-form on $\ol{S}_{K_f}$, we define the regulator pairing
\[\langle \sum_i (C_i,u_i) , \omega \rangle = \sum_i \int_{C^\mathrm{reg}_i} \log|u_i|\omega,\]
where $C_i^\mathrm{reg}$ is the curve $C_i$ with singular points removed.  (This differs by a factor of $2\pi i$ from some conventions in the literature.)  By letting $\omega$ be an algebraic differential form as defined above (extended by 0 to the boundary), it makes sense to ask for algebraicity properties of this pairing.

We remark that the volume (valued in $\mathbf{R}^\times/\ol{\mathbf{Q}}^\times$) induced by the algebraic structure $V_{\ol{\mathbf{Q}}}$ is not the same as the aforementioned canonical volume on Deligne cohomology.  This introduces an extra period, which should conjecturally compare between these volumes.  In this paper, it will take the form of a Whittaker period, but we will not prove the relationship between these volumes.

\subsection{The global integral} \label{subsec:globalintro}

Write $B_2 \subseteq \GL_2$ for the upper-triangular Borel subgroup, write $T_2$ for the diagonal torus of $\GL_2$, and write $V_2$ for the natural 2-dimensional representation of $\GL_2$ on row vectors.  Let $E(g,\Phi,\nu,s) = \sum_{\gamma \in B(\mathbf{Q}) \backslash \GL_2(\mathbf{Q})} f(\gamma g,\Phi,\nu,s)$ denote the usual real-analytic Eisenstein series on $\GL_2$, where
\[f(g,\Phi,\nu,s) = \nu_1(\det(g))|\det(g)|^{s} \int_{\GL_1(\A)}{\Phi(t(0,1)g)(\nu_1\nu_2^{-1})(t)|t|^{2s}\,dt},\]
$\Phi$ is a Schwartz-Bruhat function on $V_2(\mathbf{A}^2)$, $\nu:T_2(\mathbf{Q}) \backslash T_2(\mathbf{A}) \rightarrow \mathbf{C}^\times$ is an automorphic character trivial at infinity, and $s$ is a complex parameter.  We more concretely define $\nu_1$ and $\nu_2$ as the Hecke characters such that $\nu(\diag(t_1,t_2))=\nu_1(t_1)\nu_2(t_2)$.  We regard $E(g,\Phi,\nu,s)$ as a function of $H$ using the first projection.  For our purposes here, we will always assume that $\Phi = \Phi_\infty \otimes \Phi_f$ for $\Phi_\infty(x,y)=e^{-\pi(x^2+y^2)}$.  The Kronecker limit formula (see Section \ref{sec:klf}) then shows that for suitable $\Phi_f$, $E(g,\Phi,\nu,0)$ is the logarithm of a Siegel unit.

This reduces our computation of the regulator pairing to understanding an integral of Rankin-Selberg type: for $\varphi$ an automorphic form in the space of $\pi$, we consider
\[I(\Phi,\varphi,\nu,s) = \int_{H(\mathbf{Q})Z(\mathbf{A}) \backslash H(\mathbf{A})} \varphi(g)E(g,\Phi,\nu,s)dg.\]
Here, $Z$ is the group over $\mathbf{Q}$ of diagonal matrices $Z(R) = \set{\diag(z,z,z): z \in R^\times} \cong \mathbf{G}_{m/\mathbf{Q}}$, which is the intersection of the centers of $G$ and $H$.  This is a modified version of the integral of Gelbart and Piatetski-Shapiro \cite{gelbartPS}.

In the results below, we will only allow $\varphi_f$ (rather than $\varphi$) to vary and select $\varphi$ such that $I(\Phi,\varphi,\nu,s)$ is connected to the integral of an automorphic differential $(1,1)$-form as described above.  (See Section \ref{subsec:diff} for more details.)  For this introduction, we emphasize this by writing $I(\Phi,\varphi_f,\nu,s)$ for the resulting integral.

\subsection{Local integrals and $L$-factors} \label{subsec:localintro}

An unfolding computation in \cite{gelbartPS} shows that for factorizable data, $I(\Phi,\varphi,\nu,s)$ factorizes into local integrals $I_p(\Phi_p,\varphi_p,\nu_p,s)$ involving Whittaker functions.  The unramified computation in {\it loc.\ cit.}\ shows that when all the data is unramified, $I_p(\Phi_p,\varphi_p,\nu_p,s)$ computes $L_p^\mathrm{L}(\pi_p\times\nu_{1,p},\mathrm{Std},s)$, the degree 6 standard $\nu_1$-twisted Langlands $L$-function of $\pi_p$.  We extend this statement in Section \ref{subsec:ramifiedlocal} to the case where $p$ is ramified in $E$ but $\pi_p$ has a vector fixed by $G(\mathbf{Z}_p)$, and write $L_p^\mathrm{L}(\pi_p\times\nu_{1,p},\mathrm{Std},s)$ again for the $L$-factor in this case.  Here, the exponent $\mathrm{L}$ is written to emphasize that this $L$-function is attached to the local Langlands parameter of $\pi_p \times \nu_{1,p}$ composed with the standard representation of its $L$-group, which is given in Definition \ref{defin:stdl} below.

Although the local Langlands parameter has not been constructed for all groups, one can in principle use the work of Rogawski \cite{rogawski} to make a definition of such a parameter for any $\pi_p$ and $\nu_{1,p}$ when $p$ is inert or ramified, but we will not do this or use such $L$-factors below.  On the other hand, at places where $p$ is split, the group $\GU(2,1)$ may be identified with $\GL_3 \times \mathbf{G}_m$.  This group has well-defined Langlands parameters for any $\pi_p$ and $\nu_{1,p}$ and we again write $L_p^\mathrm{L}(\pi_p\times\nu_{1,p},\mathrm{Std},s)$ for the degree 6 $L$-factor corresponding to the aforementioned standard representation of the $L$-group.  (Note that this differs from the usual standard representation of ${}^L\!\GL_3$, which has degree 3.)

We take the following alternative approach, which allows us to concretely define $L$-factors at all places.  By a result of Baruch \cite{baruch}, the local integrals $I_p(\Phi_p,\varphi_p,\nu_p,s)$ have the inverse of a unique polynomial $P(p^{-s})$ with constant term 1 as their GCD (in a suitable sense) as $\Phi_p$ and $\varphi_p$ vary.  We write $L_p^\mathrm{G}(\pi_p\times\nu_{1,p},\mathrm{Std},s)$ for this GCD whether or not it is equal to the Langlands $L$-factor, and write $L(\pi\times\nu_1,\mathrm{Std},s)=\prod_{p \textrm{ prime}}L_p^\mathrm{G}(\pi_p\times\nu_{1,p},\mathrm{Std},s)$.  (See Section \ref{subsec:overviewproof} later in this introduction for further discussion.)  For every result other than the fine regulator formula (Theorem \ref{thm:finecnf}) we will use this definition for the global $L$-function.

Note that if $p$ is inert or ramified and $\pi_p$ has a vector fixed by $G(\mathbf{Z}_p)$, or if $p$ is split, we have given two possibly different definitions of the local $L$-factor.  We apply results of Matringe \cite{matringe} and Miyauchi \cite{my4} to show in Theorem \ref{thm:unramgcd} that if $p$ is inert or split and $\pi_p$ is unramified, then $L_p^\mathrm{L}(\pi_p\times\nu_{1,p},\mathrm{Std},s)=L_p^\mathrm{G}(\pi_p\times\nu_{1,p},\mathrm{Std},s)$.  The $L$-factors are also the same (and equal to $1$) if $p$ is split and $\pi_p$ is supercuspidal.  In particular, the two definitions differ only at finitely many places.

As mentioned above, this $L$-function may also be identified (at almost all places) with that of the motive appearing in the middle-degree cohomology of $\ol{S}_{K_f}$.

\subsection{Main theorems} \label{subsec:maintheorems}

We now state the main theorems of this paper.  We make the following hypothesis for all the results in this section.  Write $[H] = H(\mathbf{Q})Z(\mathbf{A}) \backslash H(\mathbf{A})$.  Then either
\begin{enumerate}
	\item we have
	\[\int_{[H]} \nu_1(\mu(g))\varphi(g)dg = 0\]
	for every $\varphi$ in the space of $\pi$, or
	\item $\nu_1^2\omega_\pi|_{Z(\mathbf{A})}$ is not trivial.
\end{enumerate}

Our first main theorem, Theorem \ref{thm:geom} below, interprets the value of the integral $I(\Phi,\varphi_f,\nu,s)$ at $s=0$ in terms of a regulator pairing.  It may be stated roughly as follows.
\begin{mthm}[Geometrization] \label{mthm:geom}
If $\nu_1^2\omega_\pi|_{Z(\mathbf{A})}$ is trivial, assume that $\Phi(0)=0$.  The value at $s=0$ of the integral $I(\Phi,\varphi_f,\nu,s)$ may be interpreted as the regulator pairing $\langle [\eta],\omega_{\varphi_f} \rangle$ for some higher Chow class $[\eta]$.
\end{mthm}
Due to this result, we may write subsequent theorems purely in terms of the special value $I(\Phi,\varphi_f,\nu,0)$.

Our next theorem, which corresponds to Theorem \ref{thm:rationality} below, concerns algebraicity.  In the following, $W(\pi)$ denotes the equivalence class in $\mathbf{C}^\times/\ol{\mathbf{Q}}^\times$ of the value
\[W(\pi,\varphi_f) = \int_{U(\mathbf{Q}) \backslash U(\mathbf{A})} \chi^{-1}(u)\varphi(u)\,du\]
for any $\varphi_f$ such that the right-hand side is non-zero.  Here, $U$ is the group of unipotent upper-triangular matrices in $G$, $\chi$ is an explicit automorphic character of $U$ given in Definition \ref{defin:whit}, and $\varphi$ is uniquely associated to $\varphi_f$ as discussed earlier.

We also write $W_{0,0}: \mathbf{R}^\times \rightarrow \mathbf{R}$ for the classical Whittaker function with parameters $(0,0)$.  Let $L'(\pi\times\nu_1,\mathrm{Std},s)$ denote the derivative of the standard $L$-function.
\begin{mthm}[Algebraicity] \label{mthm:algebraicity}
For all $\varphi_f \in V_{\ol{\mathbf{Q}}}$, and all $\ol{\mathbf{Q}}$-valued functions $\Phi$, we have
\[I(\Phi,\varphi,\nu,0) \in \Qbar W(\pi)W_{0,0}(8\sqrt{2}\pi D^{-\frac{3}{4}})^{-1} \pi^4 L'(\pi \times \nu_1,\mathrm{Std},0).\]
\end{mthm}

Our third main theorem, which corresponds to Theorem \ref{coarseThm} below, gives a nonvanishing statement for $I(\Phi,\varphi,\nu,0)$.  We will need an auxiliary local hypothesis, which uses the following definition.  A smooth representation $\pi_p$ of the group $G(\mathbf{Q}_p)$ with underlying space $W$ is said to be $(H,\nu_{1,p}^{-1})$ distinguished if there is a nonzero functional $\Lambda:W \rightarrow \mathbf{C}$ such that $\Lambda(\pi(h)w)=\nu_{1,p}^{-1}(\mu(h))\Lambda(w)$ for all $h \in H(\mathbf{Q}_p)$.  (Recall that $\mu$ denotes the similitude.)

The condition of being $(H,\nu_{1,p}^{-1})$-distinguished has been proven to be equivalent to the twisted representation $\pi_p \times {\nu_{1,p}}$ belonging to an endoscopic $L$-packet by Gelbart, Rogawski, and Soudry \cite[Theorem D]{grs4}.

\begin{mthm}[Nonvanishing] \label{mthm:nonvanishing}
Assume that the twisted base change of $\pi$ to $\mathrm{Res}_\mathbf{Q}^E \GL_3 \times \mathbf{G}_m$ remains cuspidal.  We break into two cases.
\begin{enumerate}
	\item If $\nu_1^2 \cdot \omega_\pi|_{Z(\mathbf{A})}$ is trivial, assume that $\pi_\ell$ is not $(H,\nu_{1,\ell}^{-1})$-distinguished for at least one finite place $\ell$. Then there exists $\varphi_f \in V_{\ol{\mathbf{Q}}}$ and a factorizable $\mathbf{Q}$-valued Schwartz-Bruhat function $\Phi$ such that $\Phi(0)=0$ and $I(\Phi,\varphi,\nu,0) \neq 0$.
	\item If $\nu_1^2 \cdot \omega_\pi|_{Z(\mathbf{A})}$ is non-trivial, then there exists $\varphi_f \in V_{\ol{\mathbf{Q}}}$ and a factorizable $\mathbf{Q}$-valued Schwartz-Bruhat function $\Phi$ such that $I(\Phi,\varphi,\nu,0) \neq 0$.
\end{enumerate}
\end{mthm}

We may combine Theorems \ref{mthm:geom}, \ref{mthm:algebraicity}, and \ref{mthm:nonvanishing} to obtain the following corollary, which corresponds to Corollary \ref{coro:chowreg} below.
\begin{coro}[Class number formula, coarse form] \label{coro:cnf}
Maintain the hypotheses in Main Theorem \ref{mthm:nonvanishing}.  Then there exists a higher Chow class $[\eta]$ and $\varphi_f \in V_{\ol{\mathbf{Q}}}$ such that
\[ \langle [\eta] , \omega_{\varphi_f} \rangle =_{\ol{\mathbf{Q}}^\times} W(\pi)W_{0,0}(8\sqrt{2}\pi D^{-\frac{3}{4}})^{-1} \pi^4 L'(\pi \times \nu_1,\mathrm{Std},0),\]
where $=_{\ol{\mathbf{Q}}^\times}$ means that the two sides are nonzero and equal up to algebraic multiple.
\end{coro}

We finally give a refined version of the special value formula that removes the unknown algebraic factor.  We introduce a number of additional hypotheses in order to do this -- for instance, we require $\pi$ to have fixed vectors under $G(\mathbf{Z}_p)$ at all primes $p$ that ramify in $E$.  See Section \ref{subsec:finer} for a detailed statement of the hypotheses and notation as well as the main result, which is Theorem \ref{thm:finecnf} there.  We provide the following weaker statement for this introduction.
\begin{mthm}[Class number formula, fine form] \label{mthm:cnffine}
We make the following hypotheses.
\begin{itemize}
	\item $\pi$ has trivial central character.
	\item If $p$ is inert or ramified in $E$, $\pi_p$ has a fixed vector under $G(\mathbf{Z}_p)$
	\item If $p$ splits in $E$, $\pi_p$ is either unramified or supercuspidal, and the latter occurs for at least one $p$.
	\item The vector $\varphi_f$ is always fixed by $G(\mathbf{Z}_p)$ and $\Phi_p = \mathrm{char}(\mathbf{Z}_p^2)$ unless $p$ is split and $\pi_p$ is supercuspidal of conductor $n=n(\pi_p)$ in the sense of \cite{jpss2}, in which case $\varphi_f$ is a new vector in the sense of {\it loc.\ cit.}\ and $\Phi_p=(p^{2n}-p^{2n-2})\cdot \mathrm{char}(p^n\mathbf{Z}_p \oplus (1+p^n\mathbf{Z}_p))$.  Here, $\mathrm{char}(\cdot)$ denotes the characteristic function.
\end{itemize}
This gives a well-defined $\ol{\mathbf{Q}}$-line of choices of $\varphi_f$.  Then $L(\pi, \mathrm{Std},s)$ as defined using the GCD is also the global Langlands $L$-function and we have the formula
\[I(\Phi,\varphi_f,\nu,0) = W(\varphi_f,\pi)8iW_{0,0}(8\sqrt{2}\pi D^{-\frac{3}{4}})^{-1} \pi^4 D^{-\frac{3}{2}}L'(\pi, \mathrm{Std},0).\]
Moreover, $I(\Phi,\varphi_f,\nu,0)$ may be interpreted as a regulator pairing.
\end{mthm}

\subsection{Overview of proof} \label{subsec:overviewproof}

We explain the ideas that go into each of the four main results.

\subsubsection{Geometrization} \label{subsubsec:introgeom}

We can use the Kronecker limit formula to find a level $K_f$ and unit $u(\Phi,\nu_1)$ on the open curve $C_{K_{H,f}}$ (where $K_{H,f} = K_f \cap H(\mathbf{A}_f)$) such that
\[I(\Phi,\varphi_f,\nu,0) = \int_{C_{K_{H,f}}} \log|u(\Phi,\nu_1)|\omega_{\varphi_f}.\]
(See Section \ref{sec:klf} below.)  However, we need to construct a class $[\eta]$ in the higher Chow group that also computes $I(\Phi,\varphi_f,\nu,0)$.  This is achieved in two stages.
\begin{enumerate}
  \item Apply a Hecke operator $T$ to the pair $(C_{K_{H,f}},u(\Phi,\nu_1))$ that annihilates functions of the cusps of $S_{K_f}^\mathrm{BB}$, but acts on $\pi$ by a nonzero scalar.
	\item Assume that a formal sum $\sum_i (C_i,u_i)$ of embedded modular curves and units on them is given with the following property: for each boundary elliptic curve $E$ of $\ol{S}_{K_f}$, the restriction of the 0-cycle $\sum_i \mathrm{div}(u_i)$ to $E$ is of degree 0.  Prove that after multiplying by a nonzero integer, this 0-cycle on $E$ is the divisor of a rational function.
\end{enumerate}
The Hecke operator described (1) has the property that $(C_{K_{H,f}},u(\Phi,\nu_1))\cdot T$ satisfies the properties assumed in (2).  Adding rational functions on the boundary curves yields the representative $\eta$ of the needed higher Chow class.

To achieve (1), we formulate in Section \ref{sec:hecke} a precise Hecke action on (formal sums of) pairs $(C_i,u_i)$ that is adjoint via the pairing to the Hecke action on differential forms.

For (2), one considers the set $\Sigma$ of all points on the boundary $\partial S_{K_f} = \ol{S}_{K_f} \setminus S_{K_f}$ that can appear as the cusp of an embedded curve arising from a Shimura subvariety attached to $H$.  By explicit calculations in coordinates, we prove that any degree 0 divisor on any boundary elliptic curve fully supported on $\Sigma$ is always torsion in the Picard group, which implies the needed result.  These arguments are carried out in Section \ref{sec:higherchow}.

\subsubsection{Algebraicity}

As above, write $I_p(\Phi_p,\varphi_p,\nu_p,s)$ for the local integral.  Baruch \cite{baruch} showed that $I_p(\Phi_p,\varphi_p,\nu_p,s)$ is always a rational function in $p^{-s}$ and admits a GCD that is the inverse of a polynomial in $p^{-s}$.  We show that in general, $I_p(\Phi_p,\varphi_p,\nu_p,s)$ is always a ratio of polynomials with coefficients in $\ol{\mathbf{Q}}$ under suitable assumptions on $\Phi_p$ and $\varphi_p$.  (See Proposition \ref{prop:locrationality}.)

From the definition of the GCD, for factorizable data, $I(\Phi,\varphi,\nu,s)$ is always a multiple of $L(\pi\times\nu_1, \mathrm{Std},s)$ by an Archimedean factor together with quotients $\frac{I_p(\Phi_p,\varphi_p,\nu_p,s)}{L_p^\mathrm{G}(\pi_p\times \nu_{1,p},\mathrm{Std},s)}$ at a finite set of bad places.  By the result above, the quotients are always polynomials $P_p(p^{-s})\in\ol{\mathbf{Q}}[p^{-s}]$.  The Archimedean factor has been computed by Koseki and Oda \cite{ko}.  The algebraicity statement follows.

\subsubsection{Nonvanishing} \label{subsubsec:intrononvanish}

Our proof of nonvanishing of the regulator requires understanding $I(\Phi,\varphi,\nu,s)$ near $s=0$ by the Kronecker limit formula.  To do this, we (1) study a partial $L$-function involving only places where the data is unramified and (2) understand the behavior of local integrals at the remaining places.

For (1), we first apply the functional equation of the Eisenstein series to relate the Taylor expansions at $s=0$ and $s=1$, compare the partial $L$-function with that of the base change to $\mathrm{Res}^E_\mathbf{Q}(\GL_3 \times \mathbf{G}_m)$, and apply the Jacquet--Shalika prime number theorem \cite{js} and a bound of Jacquet--Shalika and Rudnick--Sarnak \cite{js2,rs} towards the Generalized Ramanujan Conjecture.

For (2), it is simple to choose data to ensure the nonvanishing of local integrals when we do not require that $\Phi(0)=0$ -- in fact, we can make the integral constant.  However, to interpret $I(\Phi,\varphi,\nu,s)$ as a regulator pairing at $s=0$, we need $\Phi_\ell(0)=0$ for one place $\ell$, which can sometimes \emph{force} the local integral to vanish.  We prove the nonvanishing at $\ell$ under a non-distinction hypothesis in Section \ref{subsec:nonvanishing}.

The basic idea of this result, which is fully applicable to other situations, is to translate the question into representation theory.  The condition $\Phi_\ell(0)=0$ corresponds to looking at the Steinberg representation inside the degenerate principal series $\tau$ in which the local section lives.  If we regard the local integral as a functional on $\pi_\ell \times \tau$, then vanishing whenever $\Phi_\ell(0)=0$ means that this functional factors through the quotient of $\tau$ on which $H(\mathbf{Q}_\ell)$ acts by $\nu_{1,\ell} \circ \mu$, which gives the needed result.

\subsubsection{Finer computation}

We apply the works of Miyauchi \cite{my1,my2,my3,my4} at inert places to make explicit computations of local integrals for suitable data.  We apply a computation of Miyauchi and Yamauchi \cite{my} in the context of the Bump--Friedberg integral \cite{bf} in order to make explicit computations at split places as well, though in this case we do not always obtain the exact local Langlands $L$-factor for our choice of data.  This leads to the factors $\alpha_p$ appearing in the statement of Theorem \ref{thm:finecnf} below.

Finally, at ramified places, we make an exact computation of the local integral when $\varphi_p$ is fixed by $K_p=G(\mathbf{Z}_p)$ in Section \ref{subsec:ramifiedlocal}.  We first show that such a vector is a test function for the Whittaker model.  We then apply the technique of Piatetski-Shapiro and Rallis, which is usually associated with \emph{non-unique} models.  More precisely, we compute the inverse Satake transform of the standard local $L$-function, which may be viewed as an element of the power series ring $\mathfrak{H}_p(T(\mathbf{Q}_p),T(\mathbf{Z}_p))^{W_G}[[p^{-s}]]$ over the Weyl group invariants of the spherical Hecke algebra of the maximal torus.  The transform yields a bi-$K$-invariant function $\Delta(g,s) \in \mathfrak{H}_p(G(\mathbf{Q}_p),K_p)[[p^{-s}]]$ -- a power series in $p^{-s}$ with coefficients in the $K_p$-spherical Hecke algebra of $G(\mathbf{Q}_p)$.  We then check that the convolution of $\Delta(g,s)$ with the ramified $K_p$-spherical Whittaker function is our local integral, as needed.

\subsection{Acknowledgements}

It is a pleasure to thank George Boxer, Wee Teck Gan, Michael Harris, Kai-Wen Lan, David Loeffler, Dinakar Ramakrishnan, Giovanni Rosso, Anthony Scholl, Christopher Skinner, David Soudry, Jack Thorne, Eric Urban, Akshay Venkatesh, Sarah Zerbes, and Wei Zhang for helpful discussions about this work.

Both authors would like to thank the Bernoulli Center at EPFL for invitations to participate in part of the semester on ``Euler systems and special values of $L$-functions'', where some aspects of this work were discussed.

The second named author would also like to thank Stanford University for their incredible hospitality during his visits there, as well as to thank Jack Thorne for his invitation to visit Cambridge for discussions regarding this project.

\section{Higher Chow classes on Picard modular surfaces} \label{sec:higherchow}

The goal of this section is to prove Proposition \ref{prop:cycles}, which gives a way to construct classes in higher Chow groups.  We give relevant definitions and references for these groups in Section \ref{subsec:beilinson}.  As discussed in Section \ref{subsec:introhigherchow}, a class in the higher Chow group of interest may be represented by a formal sum $\sum_i (C_i,u_i)$ of pairs of divisors $C_i$ on the compactified Picard modular surface and rational functions $u_i$ on $C_i$, which must satisfy certain conditions.

We will limit ourselves to two types of divisors $C_i$.  Namely, these are
\begin{enumerate}
	\item the closures of embedded modular curves on compactified Picard modular surfaces and
	\item CM elliptic curves along the boundary of the compactification.
\end{enumerate}
The key will be to compute the intersections of these divisors.  Works of Cogdell \cite{cogdell} and Kudla \cite{kudla} carry out similar computations in a more classical setting.  Due to the usage of certain special coordinate systems in \cite{cogdell} that may not exist on every component, we vary the method there slightly to handle the Shimura varieties considered below.

In Section \ref{subsec:maps}, we define subgroups of $\GU(2,1)$ that give rise to embedded modular curves.  We define the relevant open varieties in Section \ref{subsec:shimura}.  We then study the compactified varieties in Section \ref{subsec:embeddings}.  We perform precise computatations of intersections on a limited class of divisors in Section \ref{subsec:imagecusp} and extend this to Hecke translates in Section \ref{subsec:translates}.  We check a smoothness property in Section \ref{subsec:smoothness}.  Finally, we combine our calculations to explain how to extend units on the open variety to elements of the relevant higher Chow group in Section \ref{subsec:higherchow}.  Note that our classes as defined here could in principle be zero in motivic cohomology.  We will only see that certain of these classes are nontrivial once we compute their regulator in Section \ref{sec:regulatorcomp}.

\subsection{Maps between unitary similitude groups} \label{subsec:maps}

We define the reductive group $G=\GU(2,1)$ over $\mathbf{Z}$ as follows.  Let $E/\mathbf{Q}$ be an imaginary quadratic extension generated by $\sqrt{-D}$, where $-D \in \mathbf{Z}_{< 0}$ is a fundamental discriminant, i.e.\ either $-D \equiv 1 \pmod{4}$ and is square-free or $-D$ is 4 times a square-free integer $D'$ with $D' \equiv 2\textrm{ or }3 \pmod{4}$.  In either case, write $\delta = \sqrt{-D}$. We define
\begin{equation}\label{eqn:jdef} J  = \(\begin{array}{ccc}  & & \delta^{-1} \\ & 1 & \\ -\delta^{-1} & &  \end{array}\).\end{equation}
Let $V \cong E^3$ be the underlying Hermitian space, so that the Hermitian form attached to $J$ is calculated by ${}^t\ol{v}Jv$ for $v \in V$, viewed as a column vector, where ${}^t$ denotes transpose and $\ol{\cdot}$ is the conjugation in $E$.  For any pair $(V,J)$ of a Hermitian space of dimension $n$ with form $J$, we define $\GU(J)$ by setting
\begin{equation}\label{eqn:gudef} \GU(J)(R) = \set{g \in \GL_n(R\otimes_\mathbf{Z} \mathfrak{O}_E)| {}^*gJg = \mu(g)J, \mu(g)\in R^\times}\end{equation}
for a ring $R$, where we write ${}^*g = {}^t\ol{g}$.  The similitude is a map $\mu:\GU(J) \rightarrow \mathbf{G}_m$.  For $(V,J)$ as in (\ref{eqn:jdef}), we write $G=\GU(2,1)=\GU(J)$.  We think of $g \in \GU(2,1)$ as acting on the left of $V$.  We also write $\gen{\cdot,\cdot}$ for the pairing $J$.

We write $\mathbf{P}(V)^+$, $\mathbf{P}(V)^0$, and $\mathbf{P}(V)^-$ for the positive, isotropic, and negative lines in $\mathbf{P}(V)$, where these words refer to the sign of the restriction of the Hermitian pairing to the line.  For $W \subseteq V$ a subspace, we define $W^\perp = \set{w \in V: \langle w, W \rangle = 0}$.  Then for $W\in \mathbf{P}(V)^+$, $W^\perp$ has signature $(1,1)$ at the Archimedean place with respect to the restriction of $J$.  Define $J_2 = \mb{}{\delta^{-1}}{-\delta^{-1}}{}$, and write $\GU(1,1) = \GU(J_2)$.  We have the following proposition.
\begin{prop} \label{prop:w}
For $W \in \mathbf{P}(V)^+$, the following are equivalent.
\begin{enumerate}
	\item The group $\GU(J|_{W^\perp}) \cong \GU(J_2)$ as algebraic groups over $\mathbf{Q}$.
	\item The space $W^\perp$ contains an isotropic line.
	\item The norm of any $w \in W$ is in the image of the norm homomorphism $\mathrm{Nm}_\mathbf{Q}^E: E^\times \rightarrow \mathbf{Q}^\times$.
\end{enumerate}
\end{prop}
\begin{proof}
If $\GU(J|_{W^\perp}) \cong \GU(J_2)$, the rational upper-triangular Borel subgroup of $\GU(J_2)$ maps to a rational Borel subgroup of $\GU(J|_{W^\perp})$, whose existence is tantamount to $J|_{W^\perp}$ containing an isotropic vector.  So (1) implies (2).

Fix $w \in W$.  If $W^\perp$ contains an isotropic line $L$, every vector in $L^\perp$ can be written as $a_1v+a_2w$ for a fixed $v \in L$ and varying $a_1,a_2\in E$.  We have $\langle a_1v+a_2w,a_1v+a_2w \rangle = \mathrm{Nm}_\mathbf{Q}^E(a_2) \langle w,w \rangle$.  In particular, the norm of every vector in $L^\perp \setminus L$ is the same in $\mathbf{Q}^\times/\mathrm{Nm}_\mathbf{Q}^E(E^\times)$.  Since the rational Borel subgroups of $G(\mathbf{Q})$ are all conjugate, the isotropic lines are in a single orbit.  In particular, we may choose $g \in G(\mathbf{Q})$ taking $L$ to the line $L_0=E\cdot {}^t(1,0,0)$.  Multiplying a vector in $L^\perp$ by $g$ does not change the norm, so by combining this fact with the preceding calculation, every element of $L^\perp \setminus L$ must have the same norm, considered as an element of $\mathbf{Q}^\times/\mathrm{Nm}_\mathbf{Q}^E(E^\times)$, as a vector in $L_0^\perp \setminus L_0$.  But ${}^t(0,1,0)$ is an element in $L_0^\perp \setminus L_0$ of norm 1, so (2) implies (3).

The determinant of a Hermitian form, viewed as an element of $\mathbf{Q}^\times/\mathrm{Nm}_\mathbf{Q}^E(E^\times)$, is independent of the basis.  We have $-1=\det(J)=\det(J|_{W})\det(J|_{W^\perp}) \in \mathbf{Q}^\times/\mathrm{Nm}_\mathbf{Q}^E(E^\times)$.  If $w \in W$ has $\langle w,w \rangle \in \mathrm{Nm}_\mathbf{Q}^E(E^\times)$ for some $w$, it follows that $\det(J|_{W^\perp}) \in -1 \cdot \mathrm{Nm}_\mathbf{Q}^E(E^\times)$.  It is a standard fact that Hermitian spaces are classified by their signature together with this invariant, so $\GU(J|_{W^\perp}) \cong \GU(J_2)$.
\end{proof}
We define $\Xi$ to be the set of $W \in \mathbf{P}(V)^+$ satisfying the equivalent conditions of Proposition \ref{prop:w}.  There is a natural action of $G(\mathbf{Q})$ on $\Xi$ via its left action on $\mathbf{P}(V)^+$.
\begin{prop} \label{prop:transxi}
The group $G(\mathbf{Q})$ acts transitively on $\Xi$.
\end{prop}
\begin{proof}
Let $W, W'\in \Xi$.  We may pick $w \in W$ and $w' \in W'$ of the same norm using the center of $G(\mathbf{Q})$.  Choose bases $(v_1,v_2)$ and $(v_1',v_2')$ of $W^\perp$ and $W^{\prime \perp}$, respectively, that correspond to the standard basis of $\GU(J_2)$ under the isomorphisms provided by Proposition \ref{prop:w}.(1).  Then the element of $\GL_3(E)$ mapping the ordered basis $(v_1,w,v_2)$ to $(v_1',w',v_2')$ preserves $J$, so it lies in $G(\mathbf{Q})$ and maps $W$ to $W'$ as needed.
\end{proof}

For $W \in \Xi$, we define $G_W$ to be the subgroup of $\GU(2,1)$ defined by
\[G_W = \set{(g_1,g_2) \in \GU(J|_{W}) \times \GU(J_{W^\perp}): \mu(g_1)=\mu(g_2)=\det(g_2)},\]
which is a subgroup of the full stabilizer of $W$ in $\GU(2,1)$.  Let $\GU(1,1)' \subseteq \GU(1,1)$ denote the kernel of $\mu/\det$.  By Proposition \ref{prop:w}.(1), $G_W$ is isomorphic to the subgroup $\GU(1) \boxtimes \GU(1,1)' \subseteq \GU(1) \times \GU(1,1)'$ of pairs with the same similitude.  Here, $\GU(1) = \GU(J_1)$, where $J_1 = (1)$; this group is isomorphic to $\res^E_\mathbf{Q} \mathbf{G}_m$.  In this way, we obtain different embeddings of unitary groups over $\mathbf{Q}$ of signature $(1,1)$ into $\GU(2,1)$.  It will be useful later to fix one particular embedding; let $W = E\cdot {}^t(0,1,0)$.

\subsection{Shimura data} \label{subsec:shimura}

To describe the maximal compact subgroup of $\GU(2,1)(\mathbf{R})$, we will work instead with the form
\begin{equation}\label{eqn:j2} J'  = \(\begin{array}{ccc} 1 & & \\ & 1 & \\  & & -1 \end{array}\),\end{equation}
which induces an isomorphic Hermitian space over $\mathbf{R}$ to the one defined by $J$.  More precisely, let
\begin{equation}\label{eqn:chgjtoj2} C=\(\begin{array}{ccc} \frac{D^{\frac{1}{4}}}{\sqrt{2}} & & \frac{D^{\frac{1}{4}}}{\sqrt{2}}\\ & 1 & \\ -\frac{D^{\frac{1}{4}}}{\sqrt{-2}} & & \frac{D^{\frac{1}{4}}}{\sqrt{-2}} \end{array}\),\end{equation}
where $(\cdot)^{\frac{1}{4}}$ denotes the positive real fourth root.  Then ${}^t\ol{C}JC  = J'$ and $g \mapsto C^{-1}gC$ gives a morphism $\GU(J) \rightarrow \GU(J')$.  (This particular choice of $C$ will be important in the Archimedean calculation below.)  Using the Hermitian form $J'$, the maximal compact subgroup $K_\infty$ and the maximal split torus $Z_\infty$ in the center of $\GU(2,1)(\mathbf{R})$ are
\[K_\infty = \set{g \in \GU(2,1)(\mathbf{R}): g=\diag(A,b), A \in U(2)(\mathbf{R}), b\in U(1)(\mathbf{R})}\]
and $Z_\infty=\set{\diag(t,t,t) \in \GU(2,1)(\mathbf{R}): t \in \mathbf{R}^\times}$, respectively.

The Lie algebra $\mathfrak{g} = \mathrm{Lie}(\GU(2,1)(\mathbf{R}))$ has a decomposition $\mathfrak{g} = \mathfrak{z} \oplus \mathfrak{k} \oplus \mathfrak{p}$ into the direct sum of $\mathfrak{z} = \mathrm{Lie}(Z_\infty)$ and the eigenspaces of the Cartan involution $X \mapsto -X^*$, where $\mathfrak{k} = \mathrm{Lie}(K_\infty)$.  There is a natural adjoint action of $K_\infty$ on $\mathfrak{p}$.  With respect to the form $J'$, we define a map
\[h: \mathbf{S} \rightarrow \GU(2,1)_{/\mathbf{R}}\]
by $h(z) = \diag(z,z,\overline{z})$, where $\mathbf{S}$ is the Weil restriction of scalars of $\mathbf{G}_m$ from $\mathbf{C}$ to $\mathbf{R}$, and $\GU(2,1)_{/\mathbf{R}}$ is the extension of scalars of $\GU(2,1)$ to $\mathbf{R}$.  The associated Hodge structure on $\mathfrak{g}$ gives a splitting $\mathfrak{p}_\mathbf{C} = \mathfrak{p}^+ \oplus \mathfrak{p}^-$, where $\mathfrak{p}_\mathbf{C}$ is the complexification and $\mathfrak{p}^+$ and $\mathfrak{p}^-$ are the spaces on which $z\in \mathbf{S}(\mathbf{R})$ acts by $z\ol{z}^{-1}$ and $z^{-1}\ol{z}$, respectively.  We define $X$ to be the $G(\mathbf{R})$-conjugacy class of $h$.

Given $W \in \Xi$ as above, we can also associate a Shimura datum to $G_W$.  Recall that $G_W \cong \GU(1) \boxtimes \GU(1,1)'$, so we need only define a Shimura datum for $H=\GU(1) \boxtimes \GU(1,1)'$.  We define $J_2' = \diag(1,-1)$; this defines the same group as $J_2$.  The morphism $h: \mathbf{S} \rightarrow (\GU(1) \boxtimes \GU(J_2')')_{/\mathbf{R}}$ is defined by $z \mapsto (z,\diag(z,\ol{z}))$, and we define $X_H$ to be its $H(\mathbf{R})$-conjugacy class.  The embedding $G_W \rightarrow G$ induces a morphism of Shimura data.  We also note that the maximal compact $K_{2,\infty}$ of $H(\mathbf{R})$ is
\[K_{2,\infty} = \set{g=(g_1,g_2) \in H(\mathbf{R}): g_1=a, g_2=\diag(b,b^{-1}), a,b \in U(1)(\mathbf{R})}.\]
We write $X_W$ for the symmetric space associated to $G_W$.

Let $\mathbf{A}$ denote the adeles of $\mathbf{Q}$, and let $\mathbf{A}_f$ denote the finite adeles.  Define $K_\mathrm{max}$ to be the open compact subgroup $K_\mathrm{max} = \prod_p G(\mathbf{Z}_p) \subseteq G(\mathbf{A}_f)$.  (The group $G(\mathbf{Z}_p)$ consists of integral matrices with respect to the form $J$ defined in (\ref{eqn:jdef}).)  Let $K_f = \prod_p K_p \subseteq K_\mathrm{max}$ be a factorizable open subgroup contained in $K_\mathrm{max}$.  Define $\Gamma_{K_f} = K_f \cap G(\mathbf{Q})$.  For $W \in \Xi$, define $K_{W,f} = K_f \cap G_W$ and $\Gamma_{W,K_f} = K_f \cap G_W(\mathbf{Q})$.  Then the double quotients $S_{K_f} = G(\mathbf{Q}) \backslash X \times G(\mathbf{A}_f) / K_f$ and $C_{W,K_f} = G_W(\mathbf{Q}) \backslash X_W\times G_W(\mathbf{A}_f) / K_{W,f}$ naturally possess the structure of complex varieties.  Here, $\gamma \in G(\mathbf{Q})$ acts by conjugating the map $h_x \in X$ and acts by left multiplication on $G(\mathbf{A}_f)$, and $K_f$ acts trivially on $X$ and acts on $G(\mathbf{A}_f)$ by right multiplication.  The double quotient for $G_W$ is defined similarly.

The following lemma, which identifies the derived subgroup $\SU(1,1)$ of $H$ with the kernel of a single character, will be useful for studying the connected components of the Shimura varieties.
\begin{lem} \label{lem:obstruction}
We have the following two facts.
\begin{enumerate}
	\item The group $\SU(2,1)$ is the kernel of the surjective map $\det \cdot \nu^{-1}: \GU(2,1) \rightarrow \res^E_\mathbf{Q} \mathbf{G}_m$.
	\item The kernel of $\det|_H \cdot \nu|_H^{-1}: H \rightarrow \res^E_\mathbf{Q} \mathbf{G}_m$ is $\SU(1,1) \subseteq \GU(1,1)'$.
\end{enumerate}
Both of these maps are surjective on points.  Moreover, when written using the form $J_2$, $\GU(1,1)'$ is naturally identified with $\GL_{2/\mathbf{Q}}$.
\end{lem}
\begin{proof}
If $\det(g) = \nu(g)$, taking norms of both sides yields $\det(g)\ol{\det(g)} = \nu(g)^2$.  Meanwhile, on $\GU(2,1)$, we always have the identity $\det(g)\ol{\det(g)} = \nu(g)^3$.   It follows that $\nu(g)=\det(g)=1$.

The kernel of the restriction of $\det \cdot \nu^{-1}$ to $H$ is the same as the intersection of $H$ with the kernel of $\det \cdot \nu^{-1}$, so $\nu$ and $\det$ must both be trivial by the above.  The kernel of $\nu$ is $\U(1) \times \SU(1,1)$ due to the $\det=\nu$ condition on $\GU(1,1)'$, and the kernel of $\det$ on this subgroup is just $\SU(1,1)$.

To see that the map is surjective, observe that $\diag(a\ol{a},a,1) \in \GU(2,1)(R)$ and $(a,\diag(a\ol{a},1)) \in H(R)$ map to any given $a \in \res^E_\mathbf{Q} \mathbf{G}_m(R)$.

For the final claim, observe that by definition, ${}^*g J_2 g = \mu(g) J_2 = \det(g) J_2$ for $g \in \GU(1,1)'$.  On the other hand, ${}^tg J_2 g = \det(g)J_2$ for any $g \in \mathrm{Res}^E_\mathbf{Q}\GL_2$.  Therefore $g=\ol{g}$ as needed.
\end{proof}

Note that $G(\mathbf{R})$ and $G_W(\mathbf{R})$ are both connected.  (The connectivity of $G_W(\mathbf{R})$ is due to the condition denoted $\boxtimes$; $\GU(1,1)'(\mathbf{R})$ is disconnected.)  Let $U=\res_\mathbf{Q}^E \mathbf{G}_m$ and $\eta=\det \cdot \nu^{-1}: \GU(2,1) \rightarrow U$. Using Lemma \ref{lem:obstruction}, it is then a standard fact that since $G(\mathbf{R})$ and $G_W(\mathbf{R})$ are connected and the derived groups of $G$ and $G_W$ are simply connected, the connected components of the Shimura varieties $S_{K_f}$ and $C_{W,K_{W,f}}$ are parametrized by
\[U(\mathbf{Q}) \backslash U(\mathbf{A}_f) / \eta(K_f) \textrm{ and }U(\mathbf{Q}) \backslash U(\mathbf{A}_f) / \eta(K_{W,f}),\]
respectively.  Moreover, the map $\eta$ on the factor $G(\mathbf{A}_f)$ or $G_W(\mathbf{A}_f)$ defines the image of a point of $S_{K_f}$ or $C_{W,K_f}$, respectively.

To simplify notation, we fix a neat $K_f = \prod_{v \ne \infty} K_v$ for the remainder of the entire section, and drop $K_f$ from the subscripts throughout.  The embedding $G_W \rightarrow G$ induces natural maps $\iota_W^\circ:C_W \rightarrow S$ defined on double cosets in the obvious way, and the corresponding cycles are defined over $E$.  (See, for instance, \cite{harris}.)  We would like to determine whether the image of this map is the same for two choices $W,W' \in \Xi$.
\begin{prop} \label{prop:cwequiv}
Let $\gamma \in G(\mathbf{Q})$ take $W$ to $W'$.  Then the following are equivalent.
\begin{enumerate}
	\item The images of $C_W$ and $C_{W'}$ in $S$ are the same.
	\item The intersection of the images of $C_W$ and $C_{W'}$ in $S$ is of dimension 1.
	\item The intersection $K_f \cap \gamma G_W(\mathbf{A}_f)$ is nonempty.
\end{enumerate}
\end{prop}
\begin{proof}
Item (1) clearly implies (2).  To see that (2) implies (3), assume that the dimension of $\iota_W(C_W) \cap \iota_{W'}(C_{W'})$ is 1.  Consider, as subsets of $X \times G(\mathbf{A}_f)/K_f$,
\[G(\mathbf{Q})X_W\times G_W(\mathbf{A}_f)K_f/K_f \textrm{ and } G(\mathbf{Q})X_{W'}\times G_{W'}(\mathbf{A}_f)K_f/K_f.\]
Since $K_f$ acts only on the $G_W$ (or $G_{W'}$) factor, and $G_W(\mathbf{A}_f)/(K_f \cap G_W(\mathbf{A}_f))$ is countable, both sides are a countable union of sets of the form $X_{W''} \times gK_f$ for various $W''$.  For $W_1,W_2 \in \Xi$, the possibilities for $X_{W_1} \times gK_f \cap X_{W_2} \times gK_f$ are either the empty set, a single point, or $X_{W_1} \times gK_f = X_{W_2} \times gK_f$, as is easy to see from the realization of $X_{W_i}$ as the set of negative lines in $V \otimes_E \mathbf{C}$ perpendicular to $W_i$.  In particular, the intersection $\iota_W(C_W) \cap \iota_{W'}(C_{W'})$ can only have countable size unless there exist $\alpha,\alpha' \in G(\mathbf{Q})$, $g\in G_W(\mathbf{A}_f)$, and $g' \in G_{W'}(\mathbf{A}_f)$ such that $X_{\alpha^{-1} W}\times \alpha^{-1}gK_f = X_{\alpha^{\prime -1} W'} \times \alpha^{\prime -1}g'K_f$.  Therefore, $\alpha'\alpha^{-1} W= W'$ and $\alpha'\alpha^{-1} gk = g'k'$ for some $k,k' \in K_f$.  Since $\alpha\alpha^{\prime -1}$ differs from the given element $\gamma$ by an element of $G_W(\mathbf{Q})$ on the right, we obtain $g^{\prime -1}\gamma h \in K_f$, where $h \in G_W(\mathbf{A}_f)$.  We can rewrite this as $\gamma (\gamma^{-1} g^{\prime -1}\gamma) h$, which is an element of $\gamma G_W(\mathbf{A}_f)$ as needed.

To see that (3) implies (1), if $K_f \cap \gamma G_W(\mathbf{A}_f)$ is nonempty, we can write $k =\gamma g$ for $k \in K_f$, $g \in G_W(\mathbf{A}_f)$, and $\gamma \in G(\mathbf{Q})$ with $\gamma W = W'$.  Then $\gamma X_W\times G_W(\mathbf{A}_f) k^{-1} = X_{W'} \times \gamma G_W(\mathbf{A}_f)g^{-1}\gamma^{-1} = X_{W'} \times G_{W'}(\mathbf{A}_f)$, so the images of $C_W$ and $C_{W'}$ are the same.
\end{proof}

We can reinterpret the condition of Proposition \ref{prop:cwequiv} as follows.  Let $V(\mathbf{A}_f) = V \otimes_\mathbf{Q} \mathbf{A}_f$.  There is a natural diagonal embedding $V \inj V(\mathbf{A}_f)$.  This induces an embedding of $\Xi \subseteq \mathbf{P}(V)^+ \inj \mathbf{P}(V(\mathbf{A}_f))$.  The space $\mathbf{P}(V(\mathbf{A}_f))$ admits an action by left multiplication by $G(\mathbf{A}_f)$ and thus $K_f$.  Then the result of the proposition is that two elements of $\Xi$ are in the same $K_f$-orbit in $\mathbf{P}(V(\mathbf{A}_f))$ if and only if the images of $C_W$ and $C_{W'}$ in $S$ are the same.

We can associate an invariant to $W \in \Xi$.  For any $w \in W$ and prime $p$ of $\mathbf{Z}$, we define a fractional $\mathbf{Z}_p$-ideal
\[\mathfrak{I}_{W,p} = \frac{(\gen{w,\mathfrak{L}_p})_p(\ol{\gen{w,\mathfrak{L}_p}})_p}{(\gen{w,w})_p} \cap \mathbf{Q}_p,\]
where $\mathfrak{L}_p$ is the $\mathbf{Z}_p$-lattice generated by the fixed standard basis of $V$ and $(\gen{w,\mathfrak{L}_p})_p$ is the fractional ideal in $\mathbf{Q}_p \otimes_\mathbf{Q} E$ generated by $\set{\gen{w,v}: v \in \mathfrak{L}_p}$.  It is easy to see that $\mathfrak{I}_{W,p}$ is independent of the choice of $w \in W$ and preserved by the equivalence in Proposition \ref{prop:cwequiv}.  We write $\mathfrak{I}_W = \prod_p \mathfrak{I}_{W,p} \subseteq \mathbf{A}_f$.

\begin{prop}
There are infinitely many distinct curves $C_W$ in $S$.
\end{prop}

\begin{proof}
We need only check that $\mathfrak{I}_W$ can take on infinitely many values.  Fix a large odd prime $p$.  Let $w=(-\frac{\delta}{p^k},0,\frac{p^k}{2})$ for $k \in \mathbf{Z}$.  Then $\langle w,w \rangle = 1$, so the line $W$ generated by $w$ is in $\Xi$.  Moreover, $(\gen{w,\mathfrak{L}_p})_p(\ol{\gen{w,\mathfrak{L}_p}})_p$ is the fractional ideal generated by $p^{-k}$.  Letting $k$ vary gives distinct possibilities for $\mathfrak{I}_{W,p}$.
\end{proof}

\subsection{Cycles on compactified Picard modular surfaces} \label{subsec:embeddings}

We extend the discussion of Section \ref{subsec:shimura} to the Picard modular surface obtained by compactifying $S$.  In particular, we describe the embedded $\GU(1,1)$-Shimura subvarieties on $S$.  The cusps of $S=S_{K_f}$ are in natural bijection with the set $\sigma=\sigma_{K_f} = B(\mathbf{Q}) \backslash G(\mathbf{A}_f) / K_f$, where $B \subseteq \GU(2,1)$ is the upper-triangular Borel subgroup with respect to the form $J$ and the map $B(\mathbf{Q}) \inj G(\mathbf{A}_f)$ is the diagonal embedding.

An alternative characterization of the cusps uses isotropic lines in $V$.  In particular, note that $\mathbf{P}(V)^0 \cong B(\mathbf{Q}) \backslash G(\mathbf{Q})$, where the isomorphism is induced by letting $\gamma \in G(\mathbf{Q})$ act on $V$ by multiplying by $\gamma^{-1}$, which identifies $B(\mathbf{Q})$ as the stabilizer of the isotropic line $E\cdot{}^t(1,0,0) = L_0 \in \mathbf{P}(V)^0$.  We obtain an isomorphism
\[f: G(\mathbf{Q}) \backslash \mathbf{P}(V)^0 \times G(\mathbf{A}_f) / K_f \rightarrow \sigma_{K_f},\]
where $\gamma \in G(\mathbf{Q})$ acts by $(Bg,h)\mapsto (Bg\gamma^{-1},\gamma h)$, $k \in K_f$ acts trivially on $\mathbf{P}(V)^0$ and by right multiplication on $G(\mathbf{A}_f)$, and the map $f$ is given by $(Bg,h)\mapsto gh$.  We write $\sigma_W$ for the set of cusps of $C_W$.  These cusps are in bijection with $G_W(\mathbf{Q}) \backslash \mathbf{P}(W^\perp)^0 \times G_W(\mathbf{A}_f) / K_{W,f}$.

The surface $S$ has a canonical toroidal compactification $\ol{S}$.  See Larsen \cite{larsen} for a detailed discussion of the construction of this compactification and its geometry.  We also have compactifications $\ol{C}_W$, and $\sigma_W$ is naturally the boundary $\ol{C}_W \setminus C_W$.  By considering the closure of the image of the open curve $C_W$ in $\ol{S}$, we obtain a map $\iota_W: \ol{C}_W \rightarrow \ol{S}_W$ extending $\iota_W^\circ$.  Write $K_\mathrm{max} = G(\widehat{\mathbf{Z}})$.  The exceptional divisors of $\ol{S}$ over each point of the boundary of the Baily-Borel compactification $\ol{S}^{\mathrm{min}}$ of $S$ are smooth genus one curves $E_j$ for $j \in \sigma$, and at level $K_f = K_\mathrm{max}$, they have a distinguished identity point.  At higher level, this is no longer the case.

In Section \ref{subsec:imagecusp}, we will write $E_j$ in coordinates.  We explicitly calculate the image under $\iota_W$ of a cusp in a given $\sigma_W$.

\subsection{Beilinson's regulator} \label{subsec:beilinson}

We introduce some definitions related to Beilinson's conjecture for the Archimedean regulator.  We continue to write $\ol{S}$ for the compactified Picard modular surface, and for concreteness, make definitions for this case only.

\begin{defin}
By restriction of scalars from $E$ to $\mathbf{Q}$, we always regard $\ol{S}$ as being defined over $\mathbf{Q}$.  (So $\ol{S}(\mathbf{C})$ now consists of two copies of the original surface.)
\end{defin}

It follows from results of Bloch \cite{bloch} and Levine \cite{levine} that Beilinson's motivic cohomology group $H^3_\mathcal{M}(\ol{S},\mathbf{Q}(2))$ can be realized as a higher Chow group $\mathrm{Ch}^2(\ol{S},1)$.  In Bloch's definition, elements of $\mathrm{Ch}^2(\ol{S},1)$ may be represented by cycles inside the product of $\ol{S}$ with a 1-simplex.  We will use a slightly different definition here, which can be identified with Bloch's definition -- see \cite[\S 5]{ms} for details on this identification.

\begin{definition} \label{def:higherchow}
Let $F$ be a number field.  The group $\mathrm{Ch}^2(\ol{S}_{/F},1)$ is the quotient of $Z^2(\ol{S}_{/F},1)$ by $B^2(\ol{S}_{/F},1)$, where these $\mathbf{Q}$-vector spaces have the following definitions.
\begin{itemize}
	\item An element of $Z^2(\ol{S},1)$ is a $G_F$-invariant formal sum $\sum_i a_i(D_i,f_i)$ of codimension 1 irreducible $\ol{\mathbf{Q}}$-rational cycles $D_i$ on $\ol{S}$ and rational functions $f_i$ on $D_i$ with coefficients $a_i \in \mathbf{Q}$ so that
	\begin{equation} \label{eqn:vancond} \sum_i a_i\mathrm{div}(f_i)=0\end{equation}
	as a cycle of codimension 2.  We also regard $n(D,f)$ to be the same as $(D,f^n)$ for $n \in \mathbf{Z}$.  By $G_F$-equivariant, we require that the entire sum is carried to itself by any $\sigma \in G_F$.
	\item The subspace $B^2(\ol{S},1) \subseteq Z^2(\ol{S},1)$ is the $\mathbf{Q}$-linear span of the ``tame symbols'', defined as follows.  Let $S'$ be an irreducible component of $\ol{S}$ over $F$ and let $(g,h)$ be a pair of non-zero rational functions on $S'$ defined over $F$.  To this pair we attach the sum $\sum_i (D_i,f_{D_i(g,h)})$, where the $D_i$'s are all the poles and roots of $g$ and $h$, and
	\[f_{D_i(g,h)} = (-1)^{\ord_{D_i}(g)\ord_{D_i}(h)}\left.\(\frac{g^{\ord_{D_i}(h)}}{h^{\ord_{D_i}(g)}}\)\right|_{D_i}.\]
\end{itemize}
We write $Z^2(\ol{S},1) = Z^2(\ol{S}_{/\mathbf{Q}},1)$, $B^2(\ol{S},1) = B^2(\ol{S}_{/\mathbf{Q}},1)$, and $\mathrm{Ch}^2(\ol{S},1)=\mathrm{Ch}^2(\ol{S}_{/\mathbf{Q}},1)$.
\end{definition}

As mentioned in the introduction, we will not be considering the integrality of the classes we construct, nor will we consider the question of whether these classes exhaust the higher Chow group.  However, our calculations below are closely related to Beilinson's conjecture for the value at $s=0$ of the $L$-function of $\pi$.  Note that we are using automorphic normalizations, so $s=0$ is the left near-central point.

Beilinson's conjecture concerns a regulator morphism
\[\mathrm{Reg}_\mathrm{Beil}: H^3_\mathcal{M}(\ol{S},\mathbf{Q}(2)) \rightarrow H^3_\mathcal{D}(\ol{S}_{/\mathbf{R}},\mathbf{R}(2)).\]
The group $H^3_\mathcal{D}(\ol{S}_{/\mathbf{R}},\mathbf{R}(2))$ is a Deligne cohomology group carrying a canonical volume modulo $\mathbf{Q}^\times$, which we will not define.  (See Schneider \cite{schneider} for a survey of this conjecture and relevant definitions.)

Using the explicit realization of elements of the higher Chow group above, it is possible to give an alternative definition of a regulator morphism via integration of differential forms -- see Lewis \cite[Example 8.11]{lewis} for more details.  In our setting, our regulator morphism group will be valued in the space $H^{1,1}(\ol{S}(\mathbf{C}),\mathbf{C})^\vee$ dual to the group of $d$-closed smooth differential $(1,1)$-forms on the $\mathbf{C}$-points of the restriction of scalars of $\ol{S}$ modulo coboundaries (as in \cite[p.\ 32]{lewis}); this target is closely related to that of Beilinson.  We then define the morphism
\[\mathrm{Reg}: \mathrm{Ch}^2(\ol{S}_{/F},1) \rightarrow H^{1,1}(\ol{S}(\mathbf{C}),\mathbf{C})^\vee\]
as follows.  Let $\xi = [\sum_j a_i(D_j,f_j)]\in \mathrm{Ch}^2(\ol{S}_{/F},1)$ and let $\omega$ be a $d$-closed smooth differential $(1,1)$-form on $\ol{S}$.  Then
\[\mathrm{Reg}(\xi)(\omega) = \langle \mathrm{Reg}(\xi),\omega \rangle = \sum_j \int_{D_j} \log |f_j| \omega.\]

\subsection{The image of a cusp of an embedded modular curve} \label{subsec:imagecusp}

Each $\iota_W$ sends the points of $\sigma_W$ to points on the various curves $E_j$, and we will need to understand these maps explicitly.  

We first prove Proposition \ref{prop:elliptic} below, which merely computes the effect of changing the coordinate in the discussion of \cite[\S 2]{kr}.  (Also see the very similar computations in \cite[\S 2]{cogdell}.)  Let $(L,g) \in \mathbf{P}(V)^0 \times G(\mathbf{A}_f)$ be a representative of a cusp in $\sigma$.  Define $B_L$ to be the stabilizer of $L$ and write $B_L = M_LN_L$ for its Levi decomposition.  Recall that $K_f$ has been assumed neat, and so $B_L(\mathbf{Q}) \cap gK_fg^{-1} = N_L(\mathbf{Q}) \cap gK_fg^{-1}$; define $\Gamma_{L,g}=N_L(\mathbf{Q}) \cap gK_fg^{-1}$.  Write $E_{L,g}$ for the genus one curve in $\ol{S}$ over the cusp $(L,g) \in \sigma$.  Finally, let $\mathbf{v}=\set{v_1,v_2,v_3}$ be a basis of the vector space $V$ over $E$ such that $v_1 \in L$ and our fixed Hermitian form has the same form (\ref{eqn:jdef}) when rewritten in this basis.

We give a description of $E_{L,g}$ by using a coordinate system associated to the choice of $\mathbf{v}=\set{v_1,v_2,v_3}$.
\begin{prop} \label{prop:elliptic}
Define the $\mathbf{Z}$-lattice $\mathfrak{L}_{\mathbf{v},g} \subseteq E$ to be generated by the elements $s(\gamma) \in \mathbf{C}$ for $\gamma \in \Gamma_{L,g}$, where $s(\gamma)$ is defined by writing
\begin{equation}\label{eqn:gammaeq}\gamma = \(\begin{array}{ccc} 1 & \ol{s(\gamma)}\delta & r(\gamma)+s(\gamma)\ol{s(\gamma)}\frac{\delta}{2}\\  & 1 & s(\gamma) \\ & & 1 \end{array}\)\textrm{ for }r(\gamma) \in \mathbf{Q}\textrm{ and } s(\gamma) \in E\end{equation}
(The matrix is written using the basis $\mathbf{v}$.)  Then there is an isomorphism $\eta_{\mathbf{v},g}: \mathbf{C}/\mathfrak{L}_{\mathbf{v},g} \rightarrow E_{L,g}$.

Given a choice of a second $v_1' \in L$ and a basis $\mathbf{v}'$ extending $v_1'$, one can explicitly identify the map $\eta_{\mathbf{v}',g}^{-1} \circ \eta_{\mathbf{v},g}$.

Suppose that $(L,g)$ and $\mathbf{v}$ are given.  Given another $(L',g')$ representing the same cusp, one can choose a coordinate system $\mathbf{v}'$ and explicitly identify $E_{L,g}$ with $E_{L',g'}$ via their respective identifications with $\mathbf{C}/\mathfrak{L}_{\mathbf{v},g}$ and $\mathbf{C}/\mathfrak{L}_{\mathbf{v}',g'}$.
\end{prop}

\begin{proof}
Choose inhomogenous coordinates $(z,u)$ associated to the basis $\mathbf{v}=\set{v_1,v_2,v_3}$ as follows.  The negative definite line $x=\mathbf{C}\cdot (x_1,x_2,x_3) \subseteq V \otimes_E \mathbf{C}$, written in the coordinate system $\mathbf{v}$, is given the coordinates $z(x) = \frac{x_1}{x_3}$ and $u(x) = \frac{x_2}{x_3}$.  In these coordinates, $L$ is naturally the ``cusp at infinity'', and $X \times gK_f \subseteq X \times (G(\mathbf{A}_f)/K_f)$ is identified with $\set{(z,u):\mathrm{Tr}^E_\mathbf{Q}(\delta^{-1}z)-|u|^2 > 0}$.  Define $X(\epsilon) = \set{(z,u):\mathrm{Tr}^E_\mathbf{Q}(\delta^{-1}z)-|u|^2 > \epsilon^{-1}}$.  For $\epsilon >0$ sufficiently small, the image of $X(\epsilon)$ in $S$ is $X(\epsilon)/\Gamma_{L,g}$.

Write $\Gamma_{L,g}'$ for the commutator $[\Gamma_{L,g},\Gamma_{L,g}]$.  As explained in \cite[\S 2]{kr} or \cite[\S 2]{cogdell}, there exists an isomorphism between $X(\epsilon)/\Gamma_{L,g}'$ and a punctured disk bundle $F$ over $\mathbf{C}$ via a map $(z,u) \mapsto (w,u)=(\exp(2\pi \delta^{-1}zq_g),u)$ for some value $q_g \in \mathbf{Q}_{>0}$.  Moreover, there is a natural action of $\Gamma_{L,g}/\Gamma_{L,g}'$ on $F$ compatible with the one on $X(\epsilon)/\Gamma_{L,g}'$, and the isomorphism of quotient spaces identifies $X(\epsilon)/\Gamma_{L,g}$ with a punctured disk bundle over the elliptic curve $\mathbf{C}/\mathfrak{L}_{\mathbf{v},g}$.  There is a unique way to extend this bundle to a disk bundle over the same curve; the adjoined points can be thought of as points $(w,u)=(0,u)$, or, alternatively, the images of the points $(z,u)=(\infty,u)$ for $u \in \mathbf{C}$.  (That is, we let the imaginary part of $z$ tend to $+\infty$.)  This identification gives the map $\eta_{\mathbf{v},g}$.

Let $v_1' \in L$ and let $\mathbf{v}' = \set{v_1',v_2',v_3'}$ be a basis extending $v_1'$.  If $v_1' = av_1$ for $a \in E$, we can factor the change of basis between $\mathbf{v}$ and $\mathbf{v}'$ through an intermediate basis $\mathbf{v}'' = \set{av_1,a^{-1}\ol{a}v_2,\ol{a}^{-1}v_3}$.  By composing the changes of basis from $\mathbf{v}$ to $\mathbf{v}''$ and from $\mathbf{v}''$ to $\mathbf{v}'$, this reduces us to considering two cases.

{\bf Case 1} ($\mathbf{v}' = \set{av_1,a^{-1}\ol{a}v_2,\ol{a}^{-1}v_3}$): We write $(z,u)_\mathbf{v} \in V$ for the line parametrized by $(z,u) \in X$ written using the coordinate system $\mathbf{v}$.  From the definition of $\mathbf{v}'$, we have $(z,u)_{\mathbf{v}} = (a^{-1}\ol{a}^{-1}z,a\ol{a}^{-2}u)_{\mathbf{v}'}$, so the induced map $\mathbf{C}/\mathfrak{L}_{\mathbf{v},g} \rightarrow \mathbf{C}/\mathfrak{L}_{\mathbf{v}',g}$ is given by $\zeta \mapsto a\ol{a}^{-2}\zeta$.  Correspondingly, $\mathfrak{L}_{\mathbf{v}',g} = a\ol{a}^{-2}\mathfrak{L}_{\mathbf{v},g}$.

{\bf Case 2} ($\mathbf{v}' = \set{v_1,v_2',v_3'}$):  We have
\begin{equation}\label{eqn:chgbase} v_2 = \ol{s}'\delta v_1+v_2'\textrm{ and }v_3 = (r'+s'\ol{s}'\frac{\delta}{2})v_1+s'v_2'+v_3'\end{equation}
for some $r'\in \mathbf{Q}, s' \in E$.  We compute
\begin{align*}
  (z,u)_\mathbf{v} &= E (zv_1+uv_2+v_3) = E (zv_1+u(\ol{s}'\delta v_1+v_2')+(r'+s'\ol{s}'\frac{\delta}{2})v_1+s'v_2'+v_3')\\
  &= E ((z+u\ol{s}'\delta+r'+s'\ol{s}'\frac{\delta}{2})v_1+(u+s')v_2'+v_3')\\
  &= (z+u\ol{s}'\delta+r'+s'\ol{s}'\frac{\delta}{2},u+s')_{\mathbf{v}'}.
\end{align*}
The induced map $\eta_{\mathbf{v}',g}^{-1} \circ \eta_{\mathbf{v},g}$ ignores the translation in the $z$ coordinate; it is given exactly by translation by $s'$.

Finally, suppose that $(L,g)$ and $\mathbf{v}$ are fixed, and a second representative $(L',g')$ of the cusp is given.  We choose the coordinate system $\mathbf{v}' = \gamma \mathbf{v}$.  Multiplication on the right by an element of $K_f$ clearly has no effect on the preceding computations, so we may assume that $(L',g') = \gamma(L,g)$.  Observe that $\Gamma_{L,g}$ and $\Gamma_{L',\gamma g}$ are related by
\[\Gamma_{L',\gamma g} = N_{L'}(\mathbf{Q}) \cap \gamma gK_fg^{-1}\gamma^{-1} = \gamma ( N_L(\mathbf{Q}) \cap gK_fg^{-1})\gamma^{-1} = \gamma \Gamma_{L,g} \gamma^{-1}.\]
The definition of $\mathfrak{L}_{\mathbf{v},g}$ (resp.\ $\mathfrak{L}_{\mathbf{v}',\gamma g}$) uses matrices with respect to $\mathbf{v}$ (resp.\ $\mathbf{v}'$), so as ideals, $\mathfrak{L}_{\mathbf{v},g} = \mathfrak{L}_{\mathbf{v}',\gamma g'}$.  Then the map $\mathbf{C}/\mathfrak{L}_{\mathbf{v},g} \rightarrow \mathbf{C}/\mathfrak{L}_{\mathbf{v}',\gamma g'}$ given by the identity on $\mathbf{C}$ naturally identifies these curves.  We may compose with a change of basis in $\mathbf{v}'$ to obtain the most general formula for a change of coordinates or representative.
\end{proof}

\begin{defin}
Let $E$ be an elliptic curve defined over a subfield $F$ of $\mathbf{C}$.  We say that two points $p,q \in E(\mathbf{C})$ are $\mathbf{Q}$-equivalent if $p-q$ is torsion in the Picard group $\mathrm{Pic}(E)$.  It is easy to see that this forms an equivalence relation on points that is stable by $\mathrm{Gal}(\mathbf{C}/F)$.
\end{defin}

Then we have the following immediate corollary of Proposition \ref{prop:elliptic}.

\begin{coro} \label{coro:oneclass}
Fix $j \in \sigma$.  As one varies over the pairs $(L,g) \in \mathbf{P}(V)^0 \times G(\mathbf{A}_f)$ that represent the cusp $j$ and the coordinate systems $\mathbf{v} = \set{v_1,v_2,v_3}$ that have $v_1 \in L$, the values $\eta_{\mathbf{v},g}(0)$ all lie in a single $\mathbf{Q}$-equivalence class of the elliptic curve $E_j$.
\end{coro}

\begin{defin}
We write $\mathbf{e}_{L,g} \subseteq E_{L,g}$ for the $\mathbf{Q}$-equivalence class of Corollary \ref{coro:oneclass}.
\end{defin}

We think of $\eta_{\mathbf{v},g}$ as specifying a coordinate on $E_{L,g}$ that depends not only on the particular representative $(L,g)$, but on the choice of a basis $\mathbf{v}$ of $V$ that places $L$ at the ``point at infinity''.  Using this coordinate, we may calculate the image of each cusp of $C_W$ in the toroidal compactification of $S$ explicitly.
\begin{prop} \label{prop:cuspmap}
Let $(L,g)\in \mathbf{P}(W^\perp)^0 \times G_W(\mathbf{A}_f)$ represent a cusp of $\sigma_W$.  We also use $(L,g)$ to represent the boundary curve $E_{L,g}$ of $\ol{\sigma}$.  Fix $v_1 \in L$ and $w \in W$, and extend $v_1$ to a basis $\set{v_1,v_2,v_3}$.  Then define $u_{\mathbf{v},w,g} =-\frac{\gen{w,v_3}}{\gen{w,v_2}} \in \mathbf{C}/\mathfrak{L}_{\mathbf{v},g}$.  Then $\eta_{\mathbf{v},g}(u_{\mathbf{v},w,g})$ is independent of the choices of $w$ and $\mathbf{v}$ and defines the image of the cusp $(L,g) \in \sigma_W$.
\end{prop}

\begin{proof}
The definition of $\eta_{\mathbf{v},g}(u_{\mathbf{v},w,g})$ is clearly independent of $w \in W$.  For coordinate independence, suppose that we are given another $\mathbf{v}'=\set{v_1',v_2',v_3'}$.  Comparing the change of basis formula in Case 1 of the proof of Proposition \ref{prop:elliptic} to the definition of $\mathfrak{L}_{\mathbf{v},g}$ shows that the effect of rescaling $\mathbf{v}$ to $\set{av_1,a^{-1}\ol{a}v_2,\ol{a}^{-1}v_3}$ for $a \in E^\times$ is compatible with the definition $u_{\mathbf{v},w,g}=-\frac{\gen{w,v_3}}{\gen{w,v_2}}$.  So we can assume $v_1 = v_1'$.

Suppose that $v_2',v_3'$ are defined as in (\ref{eqn:chgbase}).  Write $w = av_1+bv_2$; it has no $v_3$ component by definition of $\Sigma$.  Then $-\frac{\gen{w,v_3}}{\gen{w,v_2}} = -\delta^{-1}\frac{\ol{a}}{\ol{b}}$.  We also have $w = av_1 + b(\ol{s}'\delta v_1+v_2') = (a+\ol{s}'\delta)v_1 + bv_2'$.  It follows that
\[-\frac{\gen{w,v_3'}}{\gen{w,v_2'}}= -\frac{\gen{(a+\ol{s}'\delta)v_1 + bv_2',v_3'}}{\gen{(a+\ol{s}'\delta)v_1 + bv_2',v_2'}} = -\delta^{-1}\frac{\ol{a}}{\ol{b}}+s',\]
which matches the effect of changing coordinates from Case 2 of Proposition \ref{prop:elliptic}.

Finally, the derivation of the formula $-\frac{\gen{w,v_3}}{\gen{w,v_2}}$ is identical to the computation of Cogdell \cite[Lemma 3.2]{cogdell} once his $\mathfrak{N}_0\mathfrak{a}\ol{\mathfrak{a}}^2$ is replaced by $\mathfrak{L}_{\mathbf{v},g}$.
\end{proof}

The following corollary is immediate.
\begin{coro} \label{coro:torsion}
The image of $(L,g)\in \sigma_W$ is in $\mathbf{e}_{L,g}$.
\end{coro}

\begin{proof}
Fix any coordinate system $\mathbf{v}$ with $v_1 \in L$.  We have $-\frac{\gen{w,v_3}}{\gen{w,v_2}} \in E \subseteq \mathbf{C}$, so some positive integer multiple of this value is in $\mathfrak{L}_{\mathbf{v},g}$.  In other words, writing $p = \iota_W((L,g))$, we have that $n(p-\eta_{\mathbf{v},g}(0))=0 \in \mathrm{Pic}^0(E_j)(\mathbf{C})$ for some $n \in \mathbf{Z}_{\ge 1}$.  Applying Corollary \ref{coro:oneclass}, we have $p \in \mathbf{e}_{L,g}$.
\end{proof}

\subsection{Translation operators} \label{subsec:translates}

We define translation operators on the Shimura variety and discuss how they interact with the preceding results.
\begin{defin} \label{defin:translation}
If $g \in G(\A_f)$, denote by $T(g): S_{K_f} \rightarrow S_{g^{-1}K_fg}$ the usual translation operator defined on complex points by
\[G(\Q)(x,h)K_f \mapsto G(\Q)(x,h)K_fg = G(\Q)(x,hg)(g^{-1}K_fg).\]
Recall that the map $T(g)$ is an isomorphism defined over the reflex field.  Observe that $T(g) \circ T(h) = T(hg)$ and that $T(k)$ acts by the identity for $k \in K_f$.
\end{defin}

A special fact in this setting is that for $K_f$ sufficiently small, $T(g)$ extends canonically to an isomorphism $\ol{T}(g): \ol{S}_{K_f} \rightarrow \ol{S}_{g^{-1}K_fg}$.  This is due to the existence of a canonical toroidal compactification.  The paper \cite[\S2]{kr} gives an explicit description of this map, which we will describe using the proof of Proposition \ref{prop:elliptic} above.

\begin{defin} \label{defin:heckecompact}
On the boundary of the Baily-Borel compactification, the map $T(g)$ has the natural extension
\[G(\mathbf{Q})(L,h)K_f \mapsto G(\mathbf{Q})(L,hg)g^{-1}K_fg.\]
We may explicitly identify the map $\ol{T}(g)$ in a neighborhood of a point of $E_{L,h}$ as follows.  We will use the exponent $g$ when considering the Shimura variety $\ol{S}_{g^{-1}K_fg}$.  So the image of $\ol{T}(g)|_{E_{L,h}}$ will be $E_{L,hg}^g$.  In the notation of Section \ref{subsec:imagecusp}, we have $\Gamma_{L,h} = N_L(\mathbf{Q}) \cap gK_fg^{-1}$ and $\Gamma_{L,hg}^g = N_L(\mathbf{Q}) \cap hg(g^{-1}K_fg)g^{-1}h^{-1} = \Gamma_{L,g}$ as subsets of $N_L(\mathbf{Q})$.  We may use the same coordinate system $\mathbf{v}$ for both curves.  Then we have $\mathfrak{L}_{\mathbf{v},h} = \mathfrak{L}_{\mathbf{v},hg}^g$ and may define the map $\ol{T}(g): E_{L,h} \rightarrow E_{L,hg}^g$ as the identity map $\mathbf{C}/\mathfrak{L}_{\mathbf{v},h} \rightarrow \mathbf{C}/\mathfrak{L}_{\mathbf{v},hg}^g$.
\end{defin}

We now check that Corollary \ref{coro:torsion} extends to translates.

\begin{prop} \label{prop:torsion}
For any $g \in G(\mathbf{A}_f)$, every cusp of $\ol{T}(g)(\ol{C}_W)$ belongs to some $\mathbf{e}_{L,hg}^g$.
\end{prop}

\begin{proof}
It suffices to check that any translation operator sends any $\mathbf{e}_{L,g} \subseteq \ol{S}_{K_f}$ to $\mathbf{e}_{L,hg}^g \subseteq \ol{S}_{g^{-1}K_fg}$.  This follows immediately from the explicit description in Definition \ref{defin:heckecompact} of the maps $\ol{T}(g)$ and Corollary \ref{coro:oneclass}.
\end{proof}

\subsection{Smoothness} \label{subsec:smoothness}

The following result is an extension of a result of Cogdell \cite[Lemma 3.2]{cogdell} to the similitude case.  Deligne \cite[Proposition 1.15]{deligne} proves the embedding below for the open variety, so in our proof, we focus on the boundary.
\begin{prop} \label{prop:cwembed}
Fix $W \in \Xi$.  There exists an open compact $K_f(W) \subseteq G(\mathbf{A}_f)$ such that for any $K_f \subseteq K_f(W)$, we have a smooth embedding $\ol{C}_{W,K_{f,W}} \rightarrow \ol{S}_{K_f}$.  (Here, $K_{f,W} = K_f \cap G_W(\mathbf{A})$.)
\end{prop}

\begin{proof}
We may assume that $K_f(W)$ is always taken small enough that $K_f$ is neat.  Cogdell's proof of Lemma 3.2 in \cite{cogdell} gives a chart for $\ol{S}$ around a point $P\in \iota_W(\sigma_W)$ that implies that the intersection between $\ol{C}_{W,K_{f,W}}$ and the boundary elliptic curve containing $P$ is transverse as long as $\ol{C}_W$ has no self-intersection at $P$.  Thus we need only verify this latter fact.

First consider the map on minimal compactifications, which is given by
\[G_W(\mathbf{Q}) \backslash \mathbf{P}(W^\perp)^0 \times G_W(\mathbf{A}_f) / K_{W,f} \rightarrow G(\mathbf{Q}) \backslash \mathbf{P}(V)^0 \times G(\mathbf{A}_f) / K_f.\]
Suppose that $(L,g)$ and $(L',g')$ have the same image.  Since $G_W(\mathbf{Q})$ acts transitively on the isotropic vectors, we may assume $L = L'$ by changing the representative.  Then there exists $\gamma \in B_L(\mathbf{Q})$ and $k \in K_f$ so that $\gamma (L,g) k = (\gamma^{-1} L,\gamma^{-1} gk) = (L,g')$.  In particular, $\gamma \in g K_f g^{'-1}$ for $\gamma \in B_L(\mathbf{Q})$ and $g,g' \in G_W(\mathbf{A}_f)$.  By shrinking $K_f(W)$ if needed, we may assume that we have $\gamma \in N_L(\mathbf{Q})$ instead.

Now consider the curve $E_{L,g} = E_{L,g'} \subseteq \ol{S}$.  Fix $v_1 \in L$, $w \in W$ of norm 1, and let $\mathbf{v} = \set{v_1,w,v_3}$ so that the Hermitian form is given by the matrix $J$.  The image of $(L,g)$ and $(L,g') \in \sigma_W$ are respectively $\eta_{\mathbf{v},g}(0)$ and $\eta_{\mathbf{v},g'}(0)$ by Proposition \ref{prop:cuspmap}.  Since we have assumed $(L,g)$ and $(L,g')$ have the same image, $\eta_{\mathbf{v},g}(0)=\eta_{\mathbf{v},g'}(0)$.  The change of coordinate map $\mathbf{C} / \mathfrak{L}_{\gamma \mathbf{v},g} \rightarrow \mathbf{C} / \mathfrak{L}_{\mathbf{v},g'}$ is the identity by Proposition \ref{prop:elliptic}, so $\eta_{\mathbf{v},g}(0)=\eta_{\gamma \mathbf{v},g}(0)$ as well.  In particular, the coordinate change induced by $\gamma$ fixes the identity.  It follows from this and $\gamma \in N_L(\mathbf{Q})$ that the matrix of $\gamma$, written in the basis $\mathbf{v}$, is of the form (\ref{eqn:gammaeq}) with $s=0$.  In particular, $\gamma \in G_W(\mathbf{Q})$ as needed.
\end{proof}

We can use this to control pullbacks of cycles.

\begin{defin} For $K_f' \subseteq K_f$, write $\pi_{K_f'}^{K_f}: \ol{S}_{K_f'} \rightarrow \ol{S}_{K_f}$ for the covering.  We call the full preimage $\pi_{K_f'}^{K_f-1}(\ol{C}_{W,K_f})$ of the cycle $\ol{C}_{K_f}$ in $\ol{S}_{W,K_f'}$ the pullback of $\ol{C}_{K_f}$ to $\ol{S}_{K_f'}$.  Since the covering map is defined over the reflex field, the pullback preserves rationality of cycles and divisors.\end{defin}

\begin{prop} \label{prop:pullbacksmooth}
Fix $W \in \Xi$ and assume $K_f'(W) \subseteq K_f(W)$ is normal in $K_\mathrm{max} = G(\widehat{\mathbf{Z}})$, where $K_f(W)$ is as in Proposition \ref{prop:cwembed}.  Then for any open compact $K_f \subseteq K_\mathrm{max}$ and $K_f' \subseteq K_f'(W) \cap K_f$, the pullback of $\ol{C}_{K_f}$ to $\ol{S}_{K_f'}$ is a union of smooth irreducible components.
\end{prop}

\begin{proof}
It suffices to verify this for $K_f' = K_f'(W) \cap K_f$.  By our assumption that $K_f'(W)$ is normal in $K_\mathrm{max}$, $K_f'$ is normal in $K_f$, so the translations $T(k)$ for $k \in K_f$ are automorphisms of $S_{K_f'}$.  On $S_{K_f}$, the pullback is the union of the double cosets represented by the set of pairs $\{(x,gk): x \in X_W, g \in G_W(\mathbf{A}_f),k \in K_f/K_f'\}$.  In particular, $C_{W,K_{W,f'}}$ is covered by the translates $T(k)(C_{W,K_{W,f'}})$ as $k$ varies over the finite set of representatives of $K_f/K_f'$.  Since $\ol{C}_{W,K_{W,f'}}$ is smoothly embedded in $\ol{S}_{K_f'}$ by Proposition \ref{prop:cwembed}, so is each translate by the existence of the extended translation operators $\ol{T}(k)$.  (See Definition \ref{defin:translation}.)  It follows that the full pullback is a union of smooth irreducible components.
\end{proof}

\subsection{Construction of elements of $\mathrm{Ch}^2(\ol{S},1)$} \label{subsec:higherchow}

We are interested in constructing elements inside the higher Chow group $\mathrm{Ch}^2(\ol{S},1)$.  (See Definition \ref{def:higherchow} above.)  In this section, we focus on how to satisfy (\ref{eqn:vancond}), and leave addressing the field of definition for future sections.  Our main goal is to prove Proposition \ref{prop:cycles}, which will allow us to produce classes in $\mathrm{Ch}^2(\ol{S},1)$ in a flexible way.

We will be interested only in higher Chow classes that have a particular form.
\begin{enumerate}
	\item We will consider two types of cycles.
	\begin{enumerate}
		\item The first is an embedded curve $\ol{C}_i \subseteq \ol{S}_{K_f}$ given by an irreducible component of $\ol{T}(g)(\ol{C}_W)$, where $\ol{C}_W \subseteq \ol{S}_{gK_fg^{-1}}$ is an embedded curve and $\ol{T}(g)$ is the translation operator of Definition \ref{defin:heckecompact}.
		\item The second is one of the boundary elliptic curves $E_j$ for $j \in \sigma$.
	\end{enumerate}
  \item We will only consider rational functions $u_i$ on $\ol{C}_i$ whose divisor is supported on the set $\sigma_i=\partial \ol{C}_i$ of cusps.  On $E_j$, we will allow an arbitrary rational function.
\end{enumerate}
Amongst the formal sums $\sum_i (\ol{C}_i,u_i)$ involving only the first type of cycle, we would like to classify the ones for which there exist choices of rational functions $v_j$ on each $E_j$ such that
\begin{equation} \label{eqn:maineq} \sum_i \mathrm{div}(u_i) + \sum_{j \in \sigma} \mathrm{div}(v_j) = 0\end{equation}
as formal sums of points on $S$.

Our goal of finding cycles satisfying (\ref{eqn:maineq}) is simplified by the following result of Manin and Drinfel'd; it allows us to work with divisors (which are easy to construct) rather than rational functions.
\begin{prop} \label{prop:md}
For every divisor $\Xi_i$ on $\ol{C}_i$ that is supported on $\sigma_i$ and has degree 0 on $\ol{C}_i$, there exists a positive integer $n \in \mathbf{Z}$ so that $n\Xi_i$ is the divisor of a rational function $u_i$.
\end{prop}

Using Proposition \ref{prop:pullbacksmooth}, we may assume every $\ol{C}_i$ is smoothly embedded.  Due to this and Proposition \ref{prop:md}, to study (\ref{eqn:maineq}), one can consider a collection of degree 0 divisors $\Xi_i$ on $\sigma_i$ itself.

Recall that as described in Section \ref{subsec:beilinson}, we regard $\ol{S}$ as being defined over $\mathbf{Q}$ using restriction of scalars.  A formal sum $\sum_i (\ol{C}_i,\Xi_i)$ of curves and divisors $\Xi_i \in \mathbf{Z}[\sigma_i]$ is said to be defined over $\mathbf{Q}$ if it is invariant by the action of $G_\mathbf{Q}$.

\begin{prop} \label{prop:cycles}
Let $\Xi$ denote the formal sum $\sum_i (\ol{C}_i,\Xi_i)$, where $\ol{C}_i$ are irreducible components of some $\ol{T}(g)(\ol{C}_W)$, $\Xi_i$ is supported on the cusps $\sigma_i = \partial \ol{C}_i$, and $\Xi$ is defined over $\mathbf{Q}$.  Using Proposition \ref{prop:pullbacksmooth}, we pull back $\Xi$ to a level $K_f$ that is sufficiently small that each $\ol{C}_i$ is smoothly embedded.  (Note that pullback preserves the field of definition.)  Then the following are equivalent.
\begin{enumerate}
	\item There exist rational functions $u_i$ on $\ol{C}_i$, rational functions $f_j$ on $E_j$ for $j\in \sigma$, and a positive integer $n$ such that
	\begin{equation}\label{eqn:cycles} \mathrm{div}(u_i) = n\Xi_i, \quad \sum_i (\ol{C}_i,u_i) + \sum_{j\in \sigma} (E_j,f_j) \in Z^2(\ol{S},1),\end{equation}
	and $\sum_i (\ol{C}_i,u_i) + \sum_{j\in \sigma} (E_j,f_j)$ is stable by $G_\mathbf{Q}$.
	\item We have the following properties of $\Xi$.
	\begin{enumerate}
		\item Each $\Xi_i$ has degree 0.
		\item Let $\iota_i^\mathrm{min}: \sigma_i \rightarrow \sigma$ denote the map on boundaries induced by $\ol{C}_i \rightarrow \ol{S}_{K_f}^\mathrm{min}$, where $\ol{S}_{K_f}^\mathrm{min}$ is the minimal compactification.  Then $\sum_i \iota_{i*}^\mathrm{min} \Xi_i=0$ as a divisor on $\sigma$.
	\end{enumerate}
\end{enumerate}
\end{prop}

\begin{proof}
It is obvious that (1) implies (2); we now show the converse.  Condition (a) together with Proposition \ref{prop:md} implies that we may replace $\Xi$ by an integer multiple so that each $\Xi_i$ is a rational divisor on $\ol{C}_i$.  This leaves us to construct the functions $f_j$.

Let $\iota_i: \sigma_i \rightarrow \cup_j E_j$ be the map on boundaries induced by $\ol{C}_i \rightarrow \ol{S}_{K_f}$, and write $\Psi_j$ for the divisor $\left.\sum_i \iota_{i*}\right|_{E_j}$.  By condition (b), $\Psi_j$ has degree 0.  By Proposition \ref{prop:torsion}, if we fix any coordinate system $\eta_{\mathbf{v},g}$ for $E_j$, $\Psi_j$ is supported on torsion points with respect to the identity $\eta_{\mathbf{v},g}(0)$.  It is an easy fact that such a divisor defines a torsion point of $\mathrm{Pic}^0(E_j)$.  Therefore, we may multiply by an additional integer so that each $\Psi_j$ is a rational divisor.  Take $f_j$ to be the inverse of the corresponding rational function.  Then the existence of $n,u_i,$ and $f_j$ satisfying (\ref{eqn:cycles}) is now verified, save for the field of definition.

Recall that we have assumed that the formal sum $\sum_i (\ol{C}_i,u_i)$ is carried to itself by any element of $G_\mathbf{Q}$.  Due to our restriction of scalars, the boundary $\partial S_{K_f}$ is defined over $\mathbf{Q}$ \cite{harris}.  So the intersection of $\partial S_{K_f}$ with the sum of the divisors of the $u_i$ is also defined over $\mathbf{Q}$.  Then it follows from the equality $\sum_{j\in \sigma} \mathrm{div}(f_j)=-\sum_i \mathrm{div}(u_i)$ that the formal sum $\sum_{j\in \sigma} (E_j,f_j)$ is also defined over $\mathbf{Q}$.
\end{proof}

\section{Rankin-Selberg method}  \label{sec:diff}

We will study the regulator pairing defined in Section \ref{subsec:beilinson} by interpreting it as a special value of an integral of the type considered by Rankin and Selberg.  There are several ingredients that are required to translate between these notions, some of which will be treated later in Sections \ref{sec:hecke} and \ref{sec:klf}.

In this section, starting with a suitable automorphic representation $\pi$ on the unitary group $G = \GU(2,1)$, we explain how to associate an automorphic differential $(1,1)$-form $\omega$ on the Shimura variety of Section \ref{sec:higherchow} to a vector in the finite part $\pi_f$.  We will integrate these forms over a cycle of the type considered in Section \ref{sec:higherchow}; this construction is closely related to the Rankin-Selberg integral of Gelbart and Piatetski-Shapiro \cite{gelbartPS}.  We factorize the integral and compute the local factors at unramified, Archimedean, and certain ramified places.

\subsection{Automorphic differential $(1,1)$-forms} \label{subsec:diff}

We define a class of automorphic representations $\pi$ on $\GU(2,1)$ and construct associated differential $(1,1)$-forms on the Shimura varieties considered in Section \ref{sec:higherchow}.  We also introduce some group-theoretic preliminaries to be used later.

Suppose that we have a generic cuspidal automorphic representation $\pi$ on $\GU(2,1)$ with trivial central character at $\infty$ such that $\pi \cong \pi_\infty \otimes \pi_f$ with $\pi_\infty$ a discrete series representation.  Write $V_\pi$ for the underlying space of automorphic forms and write $V_\infty$ for the underlying space of $\pi_\infty$.  Denote the minimal $K_\infty$-type of $\pi_\infty$ by $V_\tau \subseteq V_\infty$, where $\tau$ is the corresponding representation of $K_\infty$.  Then we have the following result, which combines a special case of a theorem of Blasius, Harris, and Ramakrishnan with a multiplicity formula of Rogawski.
\begin{thm}[{\cite[Theorem 3.2.2\textrm{ and }(2.3.3)]{bhr},{\cite[Theorem 13.3.1]{rogawski}}}]
The representation $\pi_f$ is defined over a number field.  Moreover, $\pi_f$ has a canonical realization over $\ol{\mathbf{Q}}$ inside the interior coherent cohomology group of the tower of Shimura varieties $\ol{S}_{K_f}$ corresponding to $\pi_\infty$, where it appears with multiplicity one.
\end{thm}
We write $V_{\ol{\mathbf{Q}}}$ for the canonical model of $\pi_f$ over $\ol{\mathbf{Q}}$ given by this subspace of interior coherent cohomology.  Also write $V_f$ for $V_{\ol{\mathbf{Q}}} \otimes_{\ol{\mathbf{Q}}} \mathbf{C}$.

We now define $\pi_\infty$ by requiring that its minimal $K$-type is the three-dimensional representation appearing in $\mathfrak{p}^+ \otimes \mathfrak{p}^-$.  With this choice, there is a canonical isomorphism
\begin{equation}\label{eqn:alpha} \alpha: V_{\ol{\mathbf{Q}}} \otimes_{\ol{\mathbf{Q}}} \mathbf{C} \rightarrow \Hom_{K_\infty}(\mathfrak{p}^+ \otimes \mathfrak{p}^-,V_\pi).\end{equation}
(See, for instance, \cite[(3.3.8)]{grobseb}.)

We deduce that for any $\varphi_f \in V_f$, there is a canonically defined map $F_{\varphi_f}: \GU(2,1)(\mathbf{A}) \rightarrow \mathfrak{p}^{+\vee} \otimes \mathfrak{p}^{-\vee}$ given by $F_{\varphi_f}(g)(v)=(\alpha(\varphi_f)(v))(g)$ for $v \in \mathfrak{p}^{+\vee} \otimes \mathfrak{p}^{-\vee}$ and $g\in \GU(2,1)(\mathbf{A})$.  (Here we have used the realization of vectors in $\pi$ as complex-valued functions on $\GU(2,1)(\mathbf{A})$.)  This function satisfies
\begin{equation} \label{eqn:kaction} F_{\varphi_f}(gk) = k^{-1}F_{\varphi_f}(g).\end{equation}

We may now name a differential $(1,1)$-form $\omega_{\varphi_f}$ attached to $\varphi_f$.  Since the map $F_{\varphi_f}$ is invariant on the left by $\GU(2,1)(\mathbf{Q})$ and we are assuming trivial central character at $\infty$, we obtain a map $\omega_{\varphi_f}: \GU(2,1)(\mathbf{Q}) \backslash \GU(2,1)(\mathbf{A}) / Z_\infty \rightarrow \mathfrak{p}^{+\vee} \otimes \mathfrak{p}^{-\vee}$ satisfying (\ref{eqn:kaction}).  Since $\mathfrak{p}^+$ and $\mathfrak{p}^-$ are respectively the holomorphic and anti-holomorphic tangent spaces at the identity, a function $\omega_{\varphi_f,\alpha}$ satisfying (\ref{eqn:kaction}) is tantamount to a differential $(1,1)$-form on the $\GU(2,1)$ symmetric space.

\subsection{Global integrals} \label{subsec:globalints}

We consider integration over a single cycle corresponding to the line $W_0 = \gen{e_2}$, where $\set{e_1,e_2,e_3}$ is the fixed basis used to define $J$ in (\ref{eqn:jdef}).  This will lead directly to the integral of Gelbart and Piatetski-Shapiro.

Recall the group $G_{W_0} \cong \GU(1,1)' \boxtimes \GU(1)$ defined in Section \ref{subsec:shimura} and the matrix $J_2 = \mm{}{\delta^{-1}}{-\delta^{-1}}{}$.  For convenience, we write $H = G_{W_0}$ in this section as we will only be considering this particular embedded group.  Recall that the final claim of Lemma \ref{lem:obstruction} provides an identification $\GU(1,1)' \cong \GL_2$.  Thus the group $H$ is identified with the pairs $(g,\lambda) \in \GL_2 \times \res_\mathbf{Q}^E\mathbf{G}_m$ that satisfy $\det(g) = \nm_\mathbf{Q}^E\lambda$.  The embedding $H \rightarrow \GU(J)$ is then given by
\[ (\left(\begin{array}{cc} a & b \\ c& d\end{array}\right), \lambda) \mapsto \left(\begin{array}{ccc} a& & b \\& \lambda & \\ c & & d \end{array}\right).\]
We also define $Z \subseteq H$ to be the intersection $Z_G \cap Z_H$ of the centers of $G$ and $H$.  It is given by the subgroup
\begin{equation} \label{eqn:zdef} Z(R) = \set{(t \mathbf{1}_2,t):t \in R^\times}.\end{equation}

We write $T \subseteq \GL_2$ for the diagonal torus.  Fix a character $\nu: T(\mathbf{Q})\backslash T(\mathbf{A}) \rightarrow \mathbf{C}^\times$ with trivial Archimedean component.  If $\varphi_f$ is as in Section \ref{subsec:diff} and $K_f \subseteq \GU(2,1)(\mathbf{A}_f)$ is an open compact subgroup, we wish to compute the integral
\begin{equation} \label{eqn:regpairing} I(\varphi_f,\Phi,\nu,K_f,s) = \int_{C_{K_f}}{E(\Phi,\nu,s)\omega_{\varphi_f}}\end{equation}
where the notation is as follows.
\begin{itemize}
\item $C_{K_f}$ is, for the remainder of this section, shorthand for the Shimura curve $C_{W_0,K_f \cap H}$ associated to $H$, which was defined in Section \ref{subsec:shimura}.  There is a close relationship between the curves $C_{K_f}$ and the usual modular curves for $\GL_{2/\mathbf{Q}}$.
\item By abuse of notation, $\omega_{\varphi_f}$ denotes the $(1,1)$-form on $C_{K_f}$ obtained by pulling back the $(1,1)$-form of the same name on $S_{K_f}$.
\item The function $E(\Phi,\nu,s)$ on $C_{K_f}$ is obtained as the composition of the projection map $H(\mathbf{A}) \rightarrow \GL_2(\mathbf{A})$ with the usual real-analytic Eisenstein series $E(g,\Phi,\nu,s)$ on $\GL_2(\mathbf{A})$ defined in terms of the character $\nu$ on $T$ and Schwartz-Bruhat function $\Phi$ on $V_2(\mathbf{A})$, where $V_2$ is the defining representation of $\GL_2$ acting on row vectors on the right.  We will make this more precise in Definition \ref{defin:eis} below.
\end{itemize}

We now explain the relationship between $I(\varphi_f,\Phi,\nu,K_f,s)$ and the Rankin-Selberg integral of Gelbart and Piatetski-Shapiro.  To relate the integral (\ref{eqn:regpairing}) to one of Rankin-Selberg type, we will make the following choices of measure.
\begin{defin}
We normalize the Haar measure for a reductive or unipotent group $G'$ over $\mathbf{Z}_p$ so that $\mathrm{meas}(G'(\mathbf{Z}_p))=1$.  At the Archimedean place, we use the Iwasawa decomposition to determine the measure.  More precisely, we define the measure on the unipotent radical of the Borel of $G(\mathbf{R})$ so that the integral points have covolume 1, choose the Haar measure on the respective maximal compacts $K_\infty$ and $K_{0,\infty}$ of $G(\mathbf{R})$ and $H(\mathbf{R})$ so that $\mathrm{meas}(K_\infty)=\mathrm{meas}(K_{0,\infty})=1$, and on the maximal split torus $\mathrm{diag}(a_1,a_2,*)$, we use the multiplicative Haar measure $\frac{da_1da_2}{|a_1a_2|}$.  For quotients of groups (such as on $C_{K_f}$), we use the unique measure compatible with these choices.
\end{defin}
We will also make the following choice of basis, which will play a role in the definition of the Whittaker period in Section \ref{subsec:whittakerperiod}.
\begin{defin}\label{defin:v0}
Write $X_0$ for the symmetric space of $H = G_{W_0}$, and for $x \in X_0$, consider the $(1,1)$-tangent space $T_P^{(1,1)}C_{K_f}$ at the point $P\in C_{K_f}$ represented by $(x,1)$.  We have an identification $T_P^{(1,1)}C_{K_f} \cong \mathfrak{p}_{\PGL_2}^{+} \otimes \mathfrak{p}_{\PGL_2}^{-}$, so $T_P^{(1,1)}C_{K_f}$ is $1$-dimensional. Note that a basis vector $v_2$ of this space defines an invariant volume form on $C_{K_f}$; we choose $v_2$ so that it induces the quotient Haar measure on $X_0$.  Our embedding of groups gives an inclusion $\mathfrak{p}_{\PGL_2}^{+} \otimes \mathfrak{p}_{\PGL_2}^{-} \inj \mathfrak{p}^+\otimes \mathfrak{p}^-$, and we write $v_0$ for the projection of the image of $v_2$ to $V_\tau$.  (Note that $\alpha$ factors through this projection.)
\end{defin}

Due to our compatible choices of $v_0$ and the Haar measure on $H(\mathbf{A})$, the integral $I(\varphi_f,\Phi,\nu,K_f,s)$ is equal to 
\[[\GL_2(\widehat{\mathbf{Z}}):K_f] I(\alpha_{\varphi_f}(v_0),\Phi,\nu,s),\]
where, for $\varphi$ a cusp form in the space of $\pi$,
\begin{equation} \label{eqn:idef} I(\varphi,\Phi,\nu,s) =  \int_{H(\Q)Z(\mathbf{A})\backslash H(\mathbf{A})}{\varphi(g)E(g_1,\Phi,\nu,s)\,dg}.\end{equation}
This follows from the definition of $\omega_{\varphi_f,\alpha}$, the definition of pullback of differential forms, and the relationship between Haar measure on $\GL_2$ and the tangent vector $v_2$ in $\mathfrak{p}_{\PGL_2}^{+} \otimes \mathfrak{p}_{\PGL_2}^{-}$.  We also need the invariance of $v_0$ by $K_{0,\infty}$, which can be checked easily.  (Also recall that $Z \cong \mathbf{G}_{m / \mathbf{Q}}$ was defined in (\ref{eqn:zdef}) and that we have written $g_1$ in $E(g_1,\Phi,\nu,s)$ to emphasize that it factors through the projection $H \rightarrow \GU(1,1)' \cong \GL_2$.)

The Eisenstein series $E(g,\Phi,\nu,s)$ is a function of $\GL_2(\mathbf{A})$ defined as follows.
\begin{defin} \label{defin:eis}
Let $V_2$ denote the defining representation of $\GL_2$ on row vectors.  Let $\Phi$ denote a Schwartz-Bruhat function on $V_2(\A)$ and as above let $\nu = (\nu_1,\nu_2): T(\A)/T(\Q) \rightarrow \C^\times$ be a character of the diagonal torus of $\GL_2$ trivial at $\infty$.  We now define a section
\begin{equation} \label{eqn:eisdefin} f(g,\Phi,\nu,s) = \nu_1(\det(g))|\det(g)|^{s} \int_{\GL_1(\A)}{\Phi(t(0,1)g)(\nu_1\nu_2^{-1})(t)|t|^{2s}\,dt}\end{equation}
and set $E(g,\Phi,\nu,s) = \sum_{\gamma \in B(\Q) \backslash \GL_2(\Q)}{f(\gamma g,\Phi,\nu,s)}$, a function on $\GL_2(\mathbf{A})$.  Note that the central character of $E(g,\Phi,\nu,s)$ is $\nu_1 \nu_2$.  We write $f(g_1,\Phi,\nu,s)$ and $E(g_1,\Phi,\nu,s)$ for the functions on $H(\mathbf{A}) = \GL_2(\mathbf{A}) \boxtimes \GU(1)(\mathbf{A})$ obtained by projecting to the first factor.
\end{defin}

We give a definition of the standard automorphic $L$-function attached to $\pi$.  See Skinner \cite[\S 2]{skinner} for a similar description for the unitary case.
\begin{defin} \label{defin:stdl}
The connected component of the Langlands dual group of $G$, which we denote by ${}^LG^\circ$, is $\mathbf{G}_m(\mathbf{C})\times \GL_3(\mathbf{C})$ together with the action of the nontrivial element $\sigma \in \mathrm{Gal}(E/\mathbf{Q})$ given by $(x,g) \mapsto (x\det g , \Phi_3^{-1} {}^t g^{-1} \Phi_3)$.  Here, $\Phi_3$ is the $3 \times 3$ matrix given by $\Phi_{ij} = ((-1)^{i-1}\delta_{i,4-j})$, where $\delta_{ij}$ is the Kronecker delta function.  Then ${}^LG = {}^LG^\circ \rtimes \mathrm{Gal}(E/\mathbf{Q})$.

Define $G' = \mathrm{Res}^E_\mathbf{Q} (\mathbf{G}_m \times \GL_3)$.  The connected component of the dual group ${}^LG^{\prime\circ}$ is ${}^LG^\circ \times {}^LG^\circ$ together with an action of the nontrivial element $\sigma \in \mathrm{Gal}(E/\mathbf{Q})$ given by $\sigma(g,h) = (\sigma(h),\sigma(g))$, where $\sigma(g),\sigma(h)$ denotes the action on ${}^LG^\circ$ defined above.  Taking ${}^LG' = {}^LG^{\prime\circ} \rtimes \mathrm{Gal}(E/\mathbf{Q})$, we have an embedding ${}^LG \rightarrow {}^LG'$ extending the diagonal embedding ${}^LG^\circ \rightarrow {}^LG^{\prime \circ}$.  We define a representation ${}^LG^{\prime} \rightarrow \GL_6(\mathbf{C})$ via sending $((x_1,g_1),(x_2,g_2)) \rtimes 1 \mapsto \diag(x_1g_1,x_2\det(g_2)\Phi_3 {}^tg_2^{-1} \Phi_3^{-1})$ and $1 \rtimes \sigma \mapsto \mm{}{\mathbf{1}_3}{\mathbf{1}_3}{}$.  Finally let $r$ be the composition ${}^LG \rightarrow \GL_6(\mathbf{C})$.  The representation $r$ is the standard representation of ${}^LG$.

Note that the standard $L$-function is now well-defined as a meromorphic function whenever $\pi_p$ has a vector fixed by the special maximal compact $G(\mathbf{Z}_p)$ using the Satake transform of Cartier \cite{cartier} or Haines-Rostami \cite{hro}.  We will write $L_p^\mathrm{L}(\pi_p,\mathrm{Std},s)$ for the local $L$-factor in this case.  (We will compute these factors explicitly in Sections \ref{subsec:unram} and \ref{subsec:ramifiedlocal} below.)

If $\nu: \mathbf{Q}^\times \backslash \mathbf{A}^\times \rightarrow \mathbf{C}^\times$ is the Hecke character associated to a Dirichlet character, we define the $\nu$-twisted standard representation of ${}^LG$ as follows.  Let $G_\mathrm{tw} = G \times \mathbf{G}_m$ so that ${}^LG_\mathrm{tw} = {}^L G \times \mathbf{C}^\times$, and define the representation $r_\mathrm{tw}$ to be $r \circ p_1$ multiplied by the character $p_2$, where $p_i$ on ${}^L G \times \mathbf{C}^\times$ denotes projection to the $i^\mathrm{th}$ factor.  A Langlands parameter for $G_\mathrm{tw}$ is just the combination of the one attached to $\pi$ for $G$ with the Langlands parameter for the character $\nu$ on the second factor.

From its definition, it is easy to see that we could have defined the twisted standard $L$-function of $\pi \times \nu$ as the standard $L$-function of $(\nu \circ \mu)\pi$, where $(\nu\circ\mu)\pi$ is the product of the representation $\pi$ by the automorphic character $\nu\circ\mu$ of $G$.

We write $L_p^\mathrm{L}(\pi_p \times \nu_p,\mathrm{Std},s)$ for the $\nu_p$-twisted standard $L$-function of $\pi_p$ whenever $\pi_p$ has a vector fixed by the special maximal compact $G(\mathbf{Z}_p)$ and $\nu_p$ is unramified.
\end{defin}

\begin{remark}
In light of the second definition of the twisted $L$-function, defining the integral to explicitly accommodate a twist by $\nu$ may seem redundant.  However, in view of our geometric interpretation of the Eisenstein series in Section \ref{sec:klf}, we would like to separate the information of $\pi$ and $\nu$.
\end{remark}

\begin{remark}
As can be seen in the second definition, our notion of twisting by a character of $G$ is not the most general possible.  Lemma \ref{lem:obstruction} shows that one may define the twist of $\pi$ by an automorphic character $\nu$ of $\res_\mathbf{Q}^E \mathbf{G}_m$ by composing $\nu$ with the map $\det \cdot \mu^{-1}:G(\mathbf{Q}) \backslash G(\mathbf{A}) \rightarrow \mathrm{Res}^E_\mathbf{Q} \mathbf{G}_m(\mathbf{Q}) \backslash \mathrm{Res}^E_\mathbf{Q} \mathbf{G}_m(\mathbf{A})$.  One could accommodate such a twisting by incorporating a character of the $\GU(1)$ factor of $H$ into the integral representation, but we do not do this here.
\end{remark}

\begin{remark}
One could make a definition at all places using Arthur's approach of using the twisted base change of $\pi$ to $\mathrm{Res}^E_\mathbf{Q} (\mathbf{G}_m \times \GL_3)$ in place of the Langlands parameter of $\pi$, as the Langlands parameter at each place for the latter group is well-defined by strong multiplicity one.  However, we will only define $L_p^\mathrm{L}(\pi_p,\mathrm{Std},s)$ at certain places where it can be easily defined intrinsically (i.e.\ without global input).  These are the ones given in Definition \ref{defin:stdl} together with all places where $p$ is split, since $G \cong \GL_3 \times \mathbf{G}_m$ in that case.  (See Section \ref{subsec:unram}, or Section \ref{subsubsec:finesplit} for an explicit enumeration of possible representations and $L$-factors.)
\end{remark}

\subsection{Unfolding} The Rankin-Selberg integral $I(\varphi,\Phi,\nu,s)$ defined above is essentially that of Gelbart and Piatetski-Shapiro \cite{gelbartPS}.  There are some small differences in our construction due to the interpretation in terms of Shimura varieties, so we briefly give the unfolding of the integral.  We begin by defining local and global Whittaker functions.

\begin{defin} \label{defin:whit}
Define $\Lambda: \pi \rightarrow \mathbf{C}$ to be the nonzero Whittaker functional on $\pi$ given by
\[\Lambda(\varphi) = \int_{U_B(\Q)\backslash U_B(\A)}{\chi^{-1}(u)\varphi(u)\,du}.\]
Here, $U_B$ is the unipotent radical of the upper-triangular Borel subgroup of $G$ and the character $\chi:U_B(\mathbf{Q}) \backslash U_B(\mathbf{A})\rightarrow \mathbf{C}^\times$ is defined by $\chi(u) = \psi(\tr_{E/\Q}(\delta^{-1} u_{23}))$, where $\psi$ is the standard additive character $\psi:\Q\backslash \A \rightarrow \C^\times$ and $u_{23}$ is the entry of $u$ in position $(2,3)$ with respect to the matrix $J$.  Observe that the kernel of $\chi$ includes $U_B(\widehat{\mathbf{Z}})$.

For a cusp form $\varphi$ in the space of $\pi$, denote
\[W_{\varphi}(g) = \int_{U_B(\mathbf{Q})\backslash U_B(\mathbf{A})}{\chi^{-1}(u)\varphi(ug)\,du} = \Lambda(g\varphi).\]
the Whittaker function on $G$ associated to $\varphi$.

Assuming that we have a fixed cusp form $\varphi$ that is a pure tensor in some decomposition $\pi = \otimes_v\pi_v$ and such that $\Lambda(\varphi) = W_\varphi(1) \ne 0$, we make the following definition.  For each place $v$ of $\mathbf{Q}$, we define a local Whittaker function $W_{\varphi,v}: G(\mathbf{Q}_v) \rightarrow \mathbf{C}$ by setting $W_{\varphi,v}(g) = \Lambda(g\varphi)/\Lambda(\varphi)$.  We then have
\[W_{\varphi}(g) = W_\varphi(1)\prod_{v}{W_{\varphi,v}(g_v)}\]
where the product is over places $v$ of $\Q$.  In what follows, we will often view $\varphi$ as a fixed pure tensor and write $W_v$ in place of $W_{\varphi,v}$.
\end{defin}

Since the domain of integration above includes a quotient by $Z(\mathbf{A})$, we must introduce the following hypothesis.
\begin{hypothesis} \label{hypo:central}
Denote the central character of $\pi$ by $\omega_{\pi}:Z_G(\mathbf{Q}) \backslash Z_G(\mathbf{A}) \rightarrow \mathbf{C}^\times$. For the remainder of this paper, we assume that the character $\nu$ and the central character $\omega_\pi$ are related by $\omega_{\pi}|_{Z(\mathbf{A})}=(\nu_1\nu_2)^{-1}$.
\end{hypothesis}

Define $U_2$ to be the upper-triangular unipotent subgroup $\mm{1}{*}{}{1}$ of $\GL_2$.  For a Whittaker function $W_v$ on $\GU(J)(\Q_v)$ and a Schwartz-Bruhat function $\Phi_v$ on $V_2(\Q_v)$, define
\begin{align} \label{eqn:locint} I_v(W_v,\Phi_v,\nu_v,s) &= \int_{Z(\Q_v)U_2(\Q_v)\backslash H(\Q_v)}{f(g_{1,v},\Phi_v,\nu_v,s)W_v(g_v)\,d{g_v}} \\ \label{eqn:locint2} &= \int_{U_2(\Q_v)\backslash H(\Q_v)}{\nu_1(\det(g_{1,v}))\Phi_v((0,1)g_{1,v})W_v(g_v)|\det(g_{1,v})|_v^{s}\,d{g_v}},\end{align}
where $g_{1,v}$ is the projection of $g_v$ to $\GL_2(\mathbf{Q}_v)$ and
\[f(g_{1,v},\Phi_v,\nu_v,s) = \nu_1(\det g_{1,v}))|\det(g_{1,v})|_v^{s}\int_{\GL_1(\Q_v)}{\Phi_v(t(0,1)g_{1,v})(\nu_1\nu_2^{-1})(t)|t|_v^{2s}\,dt}.\]
The following result is due to works of Gelbart and Piatetski-Shapiro \cite{gelbartPS} and Baruch \cite{baruch}.
\begin{proposition} \label{prop:unfolding} With factorizable data and notation as above,
\begin{equation}\label{factorizationGI} I(\varphi,\Phi,\nu,s) = W_{\varphi}(1) \prod_{v}{I_v(W_v,\Phi_v,\nu_v,s)}.\end{equation}
For all finite places $p$ of $\Q$, the local integral $I_p(W_p,\Phi_p,\nu_p,s)$ is a rational function of $p^{-s}$, converges absolutely for $\mathrm{Re}(s) \gg 0$, and has meromorphic continuation in $s$.  If the finite place $p$ is not $2$, $p$ is unramified in $E$, $\pi_p$ and $\nu_p$ are unramified, $\varphi$ is fixed by the maximal compact subgroup $G(\mathbf{Z}_p)$, and $\Phi_p$ is the characteristic function of $V_2(\mathbf{Z}_p)$, then
\[I_p(W_p,\Phi_p,\nu_p,s) = L_p^\mathrm{L}(\pi_p \times \nu_{1,p},\mathrm{Std},s).\]
Hence,
\begin{align*} I(\varphi,\Phi,\nu,s) =& \left(\int_{U_2(\Q_S)\backslash H(\Q_S)}{W_{\varphi}(g)\nu_1(\det(g_1))|\det(g_1)|^{s}\Phi((0,1)g_1)\,dg}\right)\\ &\times I_{\infty}(W_{\infty},\Phi_{\infty},s) \prod_{p \notin S} L_p^\mathrm{L}(\pi_p \times \nu_{1,p},\mathrm{Std},s)\end{align*}
where $S$ is a set of bad finite places.
\end{proposition}
\begin{proof} Using Hypothesis \ref{hypo:central}, the global integral $I(\varphi,\Phi,\nu,s)$ unfolds immediately to give
\begin{align}
  &I(\varphi,\Phi,\nu,s) = \int_{Z(\mathbf{A})U_2(\mathbf{A})\backslash H(\mathbf{A})}{f(g_1,\Phi,\nu,s)W_{\varphi}(g)\,dg} \label{eqn:unfold} \\
  =& \int_{Z(\mathbf{A})U_2(\mathbf{A})\backslash H(\mathbf{A})}\nu_1(\det(g_1))|\det(g_1)|^{s} \int_{\GL_1(\A)}\Phi(t(0,1)g_1)(\nu_1\nu_2^{-1})(t)|t|^{2s}\,dt\, W_{\varphi}(g)\,dg\\
	=& \int_{Z(\mathbf{A})U_2(\mathbf{A})\backslash H(\mathbf{A})} \int_{\GL_1(\A)} \nu_1(\det(tg_1))|\det(tg_1)|^{s} \Phi((0,1)tg_1)W_{\varphi}(tg)\,dt\,dg\\
	=& \int_{U_2(\mathbf{A})\backslash H(\mathbf{A})}{\nu_1(\det(g_1))|\det(g_1)|^{s}\Phi((0,1)g_1)W_{\varphi}(g)\,dg}.\label{eqn:unfold2}
\end{align}
The factorization (\ref{factorizationGI}) follows from the well-known uniqueness of the Whittaker model.

We can relate this integral to that of Gelbart and Piatetski-Shapiro \cite{gelbartPS}.  In fact, let $H' \subseteq G$ be the subgroup stabilizing $W_0$ with trivial similitude; then $H' \cong \U(J_2) =\U(1,1)$.  We claim that
\[I(\varphi,\Phi,\nu,s) = \int_{H'(\Q)\backslash H'(\mathbf{A})}{\varphi(\sigma(g))E(\sigma(g)_1,\Phi,\nu,s)\,dg}\]
for a map $\sigma: H'(\Q)\backslash H'(\mathbf{A}) \stackrel{\sim}{\rightarrow} Z(\mathbf{A})U_2(\mathbf{A})\backslash H(\mathbf{A})$.  Indeed, there is an isomorphism between the domains of integration given as follows.  Consider the map $H \rightarrow H'$ given by $\GU(1,1)' \boxtimes \GU(1) \ni (g,\lambda) \mapsto g\lambda^{-1}$.  The kernel of this map is exactly $Z$.  The map is surjective since $\SU(1,1)$ is in the image of elements of the form $(g,1)$ and the image of elements of the form $(\diag(\lambda\ol{\lambda},1),\lambda)$ surjects (by Hilbert's Theorem 90) via the determinant onto the group $\U(1)$.  It is similarly possible to relate the integrands; the twist in \cite{gelbartPS} and \cite{baruch} becomes $\nu_1 \circ \nm_{\mathbf{A}_\mathbf{Q}}^{\mathbf{A}_E}$.

The claim regarding the local integrals at finite places is given by \cite[Proposition 3.4]{baruch}.  The unramified calculation is in \cite[\S 4]{gelbartPS}.
\end{proof}

\subsection{Nonarchimedean calculations} \label{subsec:unram}

We calculate the local integrals $I_{v}(W_v,\Phi_v,\nu,s)$ when $v$ is finite and all the data are unramified.  All the results in this section are originally due to \cite[\S 4]{gelbartPS}; we reproduce them for the reader's convenience, as the details of the split place computation were omitted from {\it loc. cit.} and the prime $p=2$ was excluded.  To be precise, in this section we prove the following statement.
\begin{proposition}[\cite{gelbartPS}] \label{prop:inertsplit} Suppose $v < \infty$ is unramified in $E$, $\pi_{v}$ is unramified, $W_v$ is the $G(\mathbf{Z}_p)$-spherical Whittaker function satisfying $W_v(1) = 1$, and $\Phi_v$ is the characteristic function of $\Z_{v}^2$.  Then $I_{v}(W_v,\Phi_v,\nu_v,s) = L^\mathrm{L}(\pi_v \times \nu_{1,v},\mathrm{Std}, s)$. \end{proposition}

For $v$ either inert or split, we must compute the integral
\begin{equation} \label{eqn:localint} \int_{Z(\mathbf{Q}_p)\backslash T(\Q_p)}{f(t_1,\Phi,\nu,s)W(t)\delta_{B_H}(t)^{-1} \,dt}.\end{equation}
This is obtained from (\ref{eqn:locint}) by applying the Iwasawa decomposition.  In the remainder of the section, we simplify our notation by writing $p$ for $v$ and writing $G,B$, etc.\ for $G(\mathbf{Q}_p)$, $B(\mathbf{Q}_p)$, etc.\ since everything is completely local.

\subsubsection{Inert case}

We begin by computing the local $L$-factor.  Suppose $p$ is inert and $\pi_p$ is the unramified principal series $\ind_B^G(\delta_{B}^{\frac{1}{2}} \eta)$ where $\eta: T \rightarrow \C^\times$ is a character of the diagonal torus of $G=\GU(J)$.  Assume $\eta(\mathrm{diag}(t_1 t_2, t_2, \ol{t_1}^{-1} t_2)) = \eta_1(t_1) \eta_2(t_2)$ for $\eta_1, \eta_2$ characters of $\mathbf{G}_m(E_p)$.  Write $\alpha_{\eta_i} = \eta_i(p)$.  Then the Frobenius conjugacy class associated to $\pi_p$ is $(\alpha_{\eta_2},\mathrm{diag}(\alpha_{\eta_1}, (\alpha_{\eta_1} \alpha_{\eta_2})^{-1}, 1)) \rtimes \sigma \in (\mathbf{G}_m(\C) \times \GL_3(\C)) \rtimes \mathrm{Gal}(E_p/\Q_p)$.  We also write $\alpha_{\nu_i} = \nu_i(p)$.

From Definition \ref{defin:stdl}, we obtain
\[L^\mathrm{L}(\pi_p \times \nu_{1,p},\mathrm{Std},s) = \left((1 - \alpha_{\nu_1}^2\alpha_{\eta_2}\alpha_{\eta_1} p^{-2s})(1-\alpha_{\nu_1}^2\alpha_{\eta_2} p^{-2s})(1- \alpha_{\nu_1}^2\alpha_{\eta_2}\alpha_{\eta_1}^{-1} p^{-2s})\right)^{-1}.\]

We carry out the unramified computation.  Define $\tau = \mathrm{diag}(p,1,p^{-1})$.  We have $f(\tau_1^n,\Phi,\nu,s) = (1-\alpha_{\nu_1}\alpha_{\nu_2}^{-1}p^{-2s})^{-1} \left(\alpha_{\nu_1}\alpha_{\nu_2}^{-1} |p|^{2s}\right)^{n}.$  Since $\GU(J) = Z_G\cdot \U(J)$, the restriction of $\pi_p$ from $\GU$ to $\U$ is irreducible.  Thus the Whittaker function on $\GU(J)$ restricts to one on $\U(J)$.  The latter Whittaker function is given in \cite[\S 4.5]{gelbartPS}, which implies
\[\delta_{B_H}^{-1}(\tau^n) W(\tau^n) = \frac{\alpha_{\eta_1}^{n+1} -\alpha_{\eta_1}^{-n-1}}{\alpha_{\eta_1} - \alpha_{\eta_1}^{-1}}.\]

Define $X = \alpha_{\nu_1}\alpha_{\nu_2}^{-1}|p|^{2s}$ and write $\alpha$ for $\alpha_{\eta_1}$. The local integral is then
\begin{align*} &\sum_{n \geq 0}{f(\tau_1^n,\Phi,\nu,s)\delta_{B_H}^{-1}(\tau^n) W(\tau^n)} = (1-X)^{-1} \sum_{n \geq 0}{X^n \frac{\alpha^{n+1} - \alpha^{-n-1}}{\alpha - \alpha^{-1}}} \\ &= (1-X)^{-1} \frac{1}{\alpha - \alpha^{-1}}\left(\frac{\alpha}{1 - \alpha X} - \frac{\alpha^{-1}}{1-\alpha^{-1}X}\right) = \frac{1}{1-X} \frac{1}{1-\alpha X} \frac{1}{1-\alpha^{-1} X} \\ &= (1-\alpha_{\nu_1}\alpha_{\nu_2}^{-1}p^{-2s})^{-1} (1-\alpha_{\nu_1}\alpha_{\nu_2}^{-1}\alpha_{\eta_1}p^{-2s})^{-1} (1-\alpha_{\nu_1}\alpha_{\nu_2}^{-1} \alpha_{\eta_1}^{-1}p^{-2s})^{-1} = L^\mathrm{L}(\pi_p \times \nu_{1,p}, \mathrm{Std},s).\end{align*}
For the last equality, recall that the central character $\alpha_{\eta_2}$ of $\pi_p$ is $(\alpha_{\nu_1}\alpha_{\nu_2})^{-1}$.

\subsubsection{Split case} \label{subsubsec:splitcase}

We begin by using the splitting of $p$ in order to rewrite the domain of integration $Z(\mathbf{Q}_p)\backslash T(\mathbf{Q}_p)$ from (\ref{eqn:localint}) as well as the groups $H$ and $G$; this will allow us to compute the local $L$-factor and simplify the local integral.

Recall that $T(\mathbf{Q}_p)$ is the group of diagonal matrices in $H(\mathbf{Q}_p)$.  We rewrite this subgroup by fixing a splitting
\begin{equation}\label{eqn:splitting} E \otimes_\mathbf{Q} \mathbf{Q}_p = \mathbf{Q}_p \oplus \mathbf{Q}_p\end{equation}
and writing a diagonal matrix in $H(\mathbf{Q}_p) = (\GU(1,1)' \boxtimes \GU(1))(\mathbf{Q}_p)$ as
\[(\diag((t_1,t_1'),(t_2,t_2')),(t_3,t_3')),\]
where each pair $(t_i,t_i')$ denotes the two $\mathbf{Q}_p$ factors of that entry.  Then $T(\mathbf{Q}_p)$ is cut out by the defining conditions of $H$, which simplify to $t_1=t_1',t_2=t_2'$, and $t_1t_2'=t_3t_3'$.  The embedded $Z(\mathbf{Q}_p) \cong \mathbf{Q}_p^\times$ consists of elements of the form $(\diag((t,t),(t,t)),(t,t))$, so we may take $t=t_3'$ and divide to see that $Z(\mathbf{Q}_p)\backslash T(\mathbf{Q}_p)$ maps isomorphically to its image under the first embedding of the $(\GU(1,1)')(\mathbf{Q}_p)$ factor.  In particular, we may write the integral over $(\mathbf{Q}_p^\times)^2$, where $(t_1,t_2) \in (\mathbf{Q}_p^\times)^2$ corresponds to the element $(\diag((t_1,t_1),(t_2,t_2)),(t_1t_2,1))$.

Similarly, $H$ is isomorphic to its first projection to $\GL_2 \times \mathbf{G}_m$ over $\mathbf{Q}_p$.  The inverse of this isomorphism is $(\mm{a}{b}{c}{d},\lambda) \mapsto ((\mm{a}{b}{c}{d},\mm{a}{b}{c}{d}),(\lambda,\lambda^{-1}(ad-bc))) \in (\GU(1,1)' \boxtimes \GU(1))(\mathbf{Q}_p)$.

The group $G(\mathbf{Q}_p)=\GU(2,1)(\mathbf{Q}_p)$ is isomorphic to $(\GL_3 \times \mathbf{G}_m)(\mathbf{Q}_p)$ via the product of the projection to the first factor of the splitting in (\ref{eqn:splitting}) with the similitude map.  In this notation, the map of $H(\mathbf{Q}_p) \cong (\GL_2 \times \mathbf{G}_m)(\mathbf{Q}_p)$ to $G(\mathbf{Q}_p)$ is given by
\begin{equation}\label{eqn:hembedsplit} (\mb{a}{b}{c}{d},\lambda) \mapsto (\(\begin{array}{ccc} a & & b \\  & \lambda & \\ c & & d\end{array}\),ad-bc).\end{equation}
Note that the embedded subgroup $Z = \set{(\mm{\lambda}{}{}{\lambda},\lambda):\lambda \in \mathbf{G}_m}$ maps to
\begin{equation}\label{eqn:zg} \set{(\diag(\lambda,\lambda,\lambda),\lambda^2):\lambda \in \mathbf{G}_m} \subseteq \GL_3 \times \mathbf{G}_m.\end{equation}

We regard the local component of the central character $\omega_{\pi_v}:(\GL_3 \times \mathbf{G}_m)(\mathbf{Q}_p) \rightarrow \mathbf{C}^\times$ as the pair of the character $\omega_\mu:\mathbf{G}_m(\mathbf{Q}_p) \rightarrow \mathbf{C}^\times$ given by restriction to the second factor together with $\omega_d: \mathbf{G}_m(\mathbf{Q}_p) \rightarrow \mathbf{C}^\times$, defined to be the restriction of $\omega_{\pi_v}$ to the center of $\GL_3(\mathbf{Q}_p)$.

We will use these identifications to write the functions $f$ and $W$ in (\ref{eqn:localint}) as functions of $(t_1,t_2) \in (\mathbf{Q}_p^\times)^2$ in the computation below.

We may write the Satake parameter of $\pi$ as the conjugacy class of $(\diag(\alpha_1,\alpha_2,\alpha_3),\alpha_\mu) \in (\GL_3 \times \mathbf{G}_m)(\mathbf{C})$.  Then $\alpha_\mu = \omega_\mu(p)$, $\alpha_1$ is the image of $(1,(p,1,1))$ under the character of the torus induced to obtain $\pi_p$, and $\alpha_2$ and $\alpha_3$ are similarly the images of $(1,(1,p,1))$ and $(1,(1,1,p))$, respectively.  We write $\alpha_{\nu_i} = \nu_i(p)$.  Also note that $\omega_d(p) = \alpha_1\alpha_2\alpha_3$.  We write $\alpha$ for $\diag(\alpha_1,\alpha_2,\alpha_3)$.  Using (\ref{eqn:zg}), the compatibility condition $\omega_\pi|_{Z(\mathbf{A})} = (\nu_1\nu_2)^{-1}$ becomes the statement $\alpha_1\alpha_2\alpha_3\alpha_\mu^2 = (\alpha_{\nu_1}\alpha_{\nu_2})^{-1}$.

It now follows from Definition \ref{defin:stdl} that the local $L$-factor is
\[L^\mathrm{L}(\pi_p \times \nu_{1,p},\mathrm{Std},s) = \prod_{i=1}^3 (1-\alpha_\mu\alpha_{\nu_1}\alpha_i p^{-s})^{-1}(1-\alpha_\mu\alpha_{\nu_1}\widehat{\alpha_i}p^{-s})^{-1},\]
where $\widehat{\alpha_i}$ is the product of the $\alpha_j$ with $j\ne i$.

Note that $\delta_{B_H}(t) = \delta_B(t)^\frac{1}{2}$ on the torus $T$.  We are therefore interested in
\[\int_{(\mathbf{Q}_p^\times)^2}{f(\diag(t_1,t_2),\Phi,\nu,s)W((\diag(t_1,t_1t_2,t_2),t_1t_2))\delta_B(t)^{-\frac{1}{2}} \,dt_1 dt_2}\]
Since all the data is unramified, we may write $t_1 = p^a$ and $t_2 = p^{-b}$ for $a,b \in \mathbf{Z}$, and rewrite this integral as a sum
\[\sum_{a,b \ge 0} f(\diag(p^a,p^{-b})),\Phi,\nu,s)W((\diag(p^a,p^{a-b},p^{-b}),p^{a-b}))\delta_B(t)^{-\frac{1}{2}},\]
where the condition $a,b \ge 0$ follows from the nonvanishing conditions for the Whittaker model.  We rewrite this using the Casselman-Shalika formula as
\[\sum_{a,b \ge 0} f(\diag(p^a,p^{-b})),\Phi,\nu,s)\alpha_\mu^{a-b}A_2^\alpha[b,a,-b],\]
where $A_2^\alpha[b,a,-b]$ denotes the trace of $\alpha$ on the representation of $\GL_3$ with fundamental weights $[b,a,-b]$.  Note that $A_2^\alpha[b,a,-b] = (\det\alpha)^{-b}A_2^\alpha[b,a,0]$.  Moreover,
\[f(\diag(p^a,p^{-b}),\Phi,\nu,s) = (1-\alpha_{\nu_1}\alpha_{\nu_2}^{-1} p^{-2s})^{-1}\alpha_{\nu_1}^a\alpha_{\nu_2}^{-b}|p^{a+b}|^s.\]
The sum becomes
\[(1-\alpha_{\nu_1}\alpha_{\nu_2}^{-1} p^{-2s})^{-1}\sum_{a,b \ge 0} (\alpha_{\nu_1}\alpha_{\nu_2})^{-b}X^{a+b}\alpha_\mu^{-2b}(\det \alpha)^{-b}A_2^\alpha[b,a,0],\]
where we have written $X = \alpha_\mu\alpha_{\nu_1}p^{-s}$.  We have the identity $\alpha_\mu^2\det\alpha = (\alpha_{\nu_1}\alpha_{\nu_2})^{-1}$ coming from the compatibility assumption on central characters.  So the expression simplifies to
\[(1-\alpha_{\nu_1}^2\alpha_\mu^2\det \alpha p^{-2s})^{-1}\sum_{a,b \ge 0} X^{a+b}A_2^\alpha[b,a,0].\]
We rewrite the factor in front to get
\[(1-A_2^\alpha[0,0,1] X^2)^{-1}\sum_{a,b \ge 0} X^{a+b}A_2^\alpha[b,a,0],\]
which by a simple application of the Pieri rule is equal to
\[(\sum_{b \ge 0} A_2^\alpha[b,0,0]X^b)(\sum_{a \ge 0} A_2^\alpha[0,a,0]X^a).\]
Since the $L$-function is attached to the twist by $\omega_\mu$ and $\nu_1$ of the sum of the standard and exterior square representations of the $\GL_3$ factor at a split place, and the $n^\mathrm{th}$ symmetric powers of the standard and exterior square representations are $A_2[n,0,0]$ and $A_2[0,n,0]$, this is exactly the local $L$-factor.

\subsection{Archimedean calculation} \label{subsec:arch}

We compute the Archimedean local integral in terms of $\Gamma$-functions.  This section is merely a translation to our setting of the more general results of Koseki-Oda \cite{ko}.  In this section we omit $\nu$ from the notation as we have assumed that $\nu_\infty$ is trivial.

Koseki-Oda \cite{ko} use the form (\ref{eqn:j2}), which is related to the form $J$ by the matrix $C$ in (\ref{eqn:chgjtoj2}).  The map $U(J')\rightarrow U(J)$ induced by $C$ takes the maximal split torus
\[ T_\mathrm{spl}' = \set{\(\begin{array}{ccc} \cosh t & & \sinh t \\  & 1 &  \\ \sinh t & & \cosh t \end{array}\): t \in \mathbf{R}} \]
and the compact torus $T_\mathrm{c}'= \set{\diag(z,1,z^{-1}): z \in S^1 \subseteq \mathbf{C}^\times}$ written in the form $J'$ to the subgroups $T_\mathrm{spl}=\set{\diag(\exp(t),1,\exp(-t)): t\in \mathbf{R}}$ and
\[T_\mathrm{c}= \set{\(\begin{array}{ccc} x & & y \\  & 1 &  \\ -y & & x \end{array}\): z = x+iy \in S^1 \subseteq \mathbf{C}}, \]
respectively.  We lastly note that the subgroup $T_\mathrm{\U(1)} = \set{\diag(z^{-1},z^2,z^{-1}):z \in S^1 \subseteq \mathbf{C}^\times} \subseteq U(J')$ is taken to itself in $U(J)$.  Since Koseki-Oda write their calculations with respect to $T_\mathrm{spl}'$ and $T_\mathrm{c}'$, we can use their results once we write our Archimedean integral in terms of functions on $T_\mathrm{spl}$ and $T_\mathrm{c}$.

Let $V_\tau$ denote the minimal $K_\infty$-type of $\pi_\infty$.  The representation $V_\tau$ is isomorphic to $V_\mu$ as defined on page 965 of \cite{ko}, where $\mu = (\mu_1,\mu_2) = (1,-1)$, the Blattner parameter of $\pi_\infty$.  Up to $\mathbf{C}^\times$ multiple, there is a unique vector in $V_\tau$ that maps to to the vector $v_1^{(-1,1)}$ in {\it loc. cit.}; we call it $v_\tau$.  It follows from the formulas on page 965 of \cite{ko} that $\mathbf{C}v_1^{(1,-1)}$ is the kernel of the action of $H_{13}'$, which is an operator that generates the Lie algebra of $T_\mathrm{c}'$.  The corresponding vector $v_\tau$ is the unique vector stabilized by $T_\mathrm{c}$.  Note that $T_\mathrm{c}$ is the maximal compact subgroup of $\SL_2(\mathbf{R})$.

We assume that $\varphi_\infty = v_\tau$ and define $W_\infty$ with respect to this choice of $v_\tau$ as in Definition \ref{defin:whit} so that $W_\infty(1)=1$.  (We will see shortly that any nonzero Whittaker functional is nonvanishing on $v_\tau$, so this definition is valid.)  Also note that scaling $v_\tau$ yields the same function $W_\infty$.  Our goal is to compute the integral
\begin{align}\label{IminInt} I_\infty(W_\infty,\Phi_\infty,s) &= \int_{Z(\R)U_2(\R)\backslash H(\R)}{f(g_1,\Phi_\infty,s)W_\infty(g)\,d{g}} \\ &= \int_{U_2(\R)\backslash H(\R)}{\Phi_{\infty}((0,1)g_1)W_\infty(g)|\det(g_1)|^{s}\,d{g}}\nonumber \end{align}
where 
\[f(g_1,\Phi_\infty,s) = |\det(g_1)|^{s}\int_{\GL_1(\R)}{\Phi_\infty(t(0,1)g_1)|t|^{2s}\,dt}\]
and $\Phi_{\infty}(x_1,x_2) = e^{-\pi(x_1^2 + x_2^2)}$ is a Gaussian.  Observe that
\[f(g_1,\Phi_\infty,s) = \pi^{-s}\Gamma(s) |\det(g_1)|^s |j(g_1,i)|^{-2s},\]
where $j(g_1,i) = ci+d$ is the usual automorphy factor of $g_1 = \mm{a}{b}{c}{d} \in \GL_2(\mathbf{R})$.

The following is the main result of this section.
\begin{proposition}[{\cite{ko}}]
We have
\[I_{\infty}(W_\infty,\Phi_\infty,s)= 8iW_{0,0}(8\sqrt{2}\pi D^{-\frac{3}{4}})^{-1} \pi^4 D^{\frac{3s-3}{2}}\Gamma_\mathbf{C}(s)\Gamma_\mathbf{C}(s+1) \Gamma_\mathbf{C}(s+1),\]
where $\Gamma_\C(s) = 2(2\pi)^{-s}\Gamma(s)$.
\end{proposition}

We now apply the Iwasawa decomposition to (\ref{IminInt}).  Using the fact that $\Phi_\infty$ and $v_\tau$ are stable by the maximal compact subgroup $T_\mathrm{c}$ of $\SL_2(\mathbf{R}) \subseteq H(\mathbf{R})$, (\ref{IminInt}) simplifies to
\begin{equation} \label{eqn:wints1} \pi^{-s} \Gamma(s) \int_{\GL_1(\mathbf{R}) \times S^1} |t|^{2s-2}W_\infty(\iota(t,c))dt\,dc,\end{equation}
where the map $\iota:\GL_1(\mathbf{R})\times S^1 \rightarrow H(\mathbf{R})$ is given by $(t,c) \mapsto \left(\mm{t}{}{}{t^{-1}},c\right)$.  The image of this element in $G(\mathbf{R})$ is $\mathrm{diag}(t,c,t^{-1})$.  The group $T_{\U(1)}$ acts trivially on $v_\tau$.  Combining this with the assumed triviality of the central character yields that $\mathrm{diag}(1,c,1)$ also acts trivially on $v_\tau$.  It follows that we may simplify the integral to
\begin{equation} \label{eqn:wint} \pi^{-s} \Gamma(s) \int_{\GL_1(\mathbf{R})} |t|^{2s-2}W_\infty(\iota(t,1))dt = 2\pi^{-s} \Gamma(s) \int_{\mathbf{R}_{>0}} |t|^{2s-2}W_\infty(\iota(t,1))d^\times t,\end{equation}
where we write $d^\times t$ to emphasize that we are using the multiplicative Haar measure.  We therefore need to understand the function $W_\infty|_{T_\mathrm{spl}}$.  This is computed in \cite{ko}, but we need to make some identifications.

We identify $T_{\mathrm{spl}}$ with $\mathbf{R}^\times$, and note that $W_\infty$ is determined by its restriction to $\mathbf{R}_{>0}$.  The integral (\ref{eqn:wint}) amounts to the computation of a Mellin transform of $W_\infty$.  Using the identifications of the preceding paragraph, we see that up to a constant, $W_\infty$ is the same as the function $c_1^{(1,-1)}$ defined on page 974 of \cite{ko}.

Since we would like a precise result, we must compute this constant.  In the notation of \cite{ko}, the matrices $E_{2,+},E_{2,-} \in \mathrm{Lie}\ \U(J')(\mathbf{R})$ defined by
\[E_{2,+} = \(\begin{array}{ccc} & -1 & \\ 1 & & -1 \\ & -1 & \end{array}\) \textrm{ and }E_{2,-} = \(\begin{array}{ccc} & -i & \\ -i & & i \\ & -i & \end{array}\)\]
conjugate to upper-triangular nilpotent matrices with $-\sqrt{-2}D^{-\frac{1}{4}}$ and $-\sqrt{2}D^{-\frac{1}{4}}$ in position (2,3), respectively.  Recall that we have defined the character $\chi:U_B(\mathbf{Q}) \backslash U_B(\mathbf{A}) \rightarrow \mathbf{C}^\times$ to be $\chi(u)=\psi(\tr_{E/\mathbf{Q}}(\delta^{-1} u_{23}))$, where $u_{23}$ is the entry in position (2,3) when $u$ is written with respect to $J$ and $\psi$ is the standard additive character, whose Archimedean component is $\psi_\infty(x) = \exp(2\pi i x)$.  Then $\chi$ sends $E_{2,+}$ to $-4\pi i\sqrt{2} D^{-\frac{3}{4}}$ and sends $E_{2,-}$ to $0$, which yields $\eta_+ = \eta_- = -4\pi i\sqrt{2} D^{-\frac{3}{4}}$ in the notation on page 974 of \cite{ko}.  We then define $b = -\eta_+\eta_- = 32\pi^2D^{-\frac{3}{2}}>0$, so $\gamma_1=\frac{\eta_-}{\sqrt{b}} = -i$ in the notation of \cite[Theorem 4.5]{ko}, which shows that $c_1^{(1,-1)}(t) = -it^{\frac{7}{2}}W_{0,0}(8\sqrt{2}\pi D^{-\frac{3}{4}}t)$ for $t \in \mathbf{R}_{>0}$, where $W_{\cdot,\cdot}$ denotes the classical Whittaker function.  In particular, we set $c = c_1^{(1,-1)}(1) = iW_{0,0}(8\sqrt{2}\pi D^{-\frac{3}{4}})$, so our function $W_\infty(\iota(t,1))$ is given by $c^{-1} \cdot c_1^{(1,-1)}(t)$.

The formula \cite[(5.2)]{ko} computes
\[ \int_0^\infty c_1^{(1,-1)}(t) t^{s}\,d^\times t = \left(\frac{\sqrt{b}}{2}\right)^{-s}\Gamma(\frac{s}{2}+2) \Gamma(\frac{s}{2}+2).\]
The integral (\ref{eqn:wint}) is thus given by
\begin{align*} 2\pi^{-s} \Gamma(s) \int_{\mathbf{R}_{>0}} |t|^{2s-2}W_\infty(\iota(t,1))d^\times t&= 2c^{-1} \pi^{-s} \Gamma(s) \(\frac{b}{4}\)^{-s+1}\Gamma(s+1) \Gamma(s+1)\\
  &=2^{-3s+4}c^{-1} \pi^{-3s+2} D^{\frac{3s-3}{2}}\Gamma(s)\Gamma(s+1) \Gamma(s+1)\\
	&=8iW_{0,0}(8\sqrt{2}\pi D^{-\frac{3}{4}})^{-1} \pi^4 D^{\frac{3s-3}{2}}\Gamma_\mathbf{C}(s)\Gamma_\mathbf{C}(s+1) \Gamma_\mathbf{C}(s+1),
\end{align*}
where $\Gamma_{\C}(s) = 2 (2\pi)^{-s}\Gamma(s)$.

\subsection{Ramified unitary group integral} \label{subsec:ramifiedlocal}

In this subsection we compute the local integral when $p \neq 2$ is finite and ramified in the quadratic extension $E$, but the local representation $\pi_p$ is spherical for the special maximal compact subgroup $K = \GU(J)(\mathbf{Z}_p)$ of $G= \GU(J)(\mathbf{Q}_p)$.  (Recall that $\GU(J)(R)$ is defined by (\ref{eqn:gudef}) for any $\mathbf{Z}$-algebra $R$.)  Write $E_p = E \otimes_\mathbf{Q} \mathbf{Q}_p$.  One convenient reference for the Iwasawa decomposition used below is \cite[Proposition 9.1.1]{hro}.

\subsubsection{A test vector for the Whittaker functional}  In order to make use of the hypothesis that $\pi_p$ is spherical, we need to verify the following.
\begin{prop}\label{prop:testvector}
Suppose that $\pi_p$ is a generic $K$-spherical representation of $\GU(J)(\mathbf{Q}_p)$, where $p \ne 2$ ramifies in $E$.  Then the spherical vector $v \in \pi_p$ fixed by $K$ has nonzero image under the Whittaker functional $\Lambda$.
\end{prop}

\begin{proof}
We first claim that it suffices to replace $\pi_p$ with its restriction to $\SU(J)$.  Since $\U(J)Z = G$, the restriction of $\pi_p$ to $\U(J)$ remains irreducible.  Moreover, $\SU(J)$ and $\U(J)$ have the same Hecke algebra with respect to $K$, so due to the equivalence of categories between $K$-spherical representations and representations of this Hecke algebra, the restriction of $\pi_p$ to $\SU(J)$ remains irreducible.  Then by uniqueness of the Whittaker function, it suffices to prove the result for $\SU(J)$.

Suppose that $\Lambda(v)=0$.  The Hecke algebra is generated by the double coset operator $T$ associated to the characteristic function of $K\diag(\varpi,1,\varpi^{-1})K$, where $\varpi$ (which we can and do take to be $\delta$ since $p \ne 2$) is a uniformizer.  A set $\Sigma_T$ of left $K$-coset representatives for $T$ are given by
\begin{equation} \label{eqn:mab} \underbrace{\(\begin{array}{ccc} 1 & \ol{y}\delta & x+y\ol{y}\delta/2 \\ & 1 & y \\ & & 1 \end{array}\)}_{M_{x,y}}\(\begin{array}{ccc} \varpi & &  \\ & 1 & \\ & & \varpi^{-1} \end{array}\) \textrm{ and } \(\begin{array}{ccc} \varpi^{-1} &  &  \\ & 1 & \\ & & \varpi \end{array}\)\end{equation}
where $x, y$ range from 0 to $p-1$.  Moreover, $T$ scales $v$, so $0 = \Lambda(T*v) = \sum_{\gamma \in \Sigma_T} \Lambda(\gamma v)$.

Consider the identity
\[\(\begin{array}{ccc} \varpi^{-k} &  &  \\ & 1 &  \\ & & \varpi^k \end{array}\)\underbrace{\(\begin{array}{ccc} 1 & \varpi^{\ell+1} & \varpi^{2\ell+1}/2 \\ & 1 & \varpi^\ell \\ & & 1 \end{array}\)}_{M_{0,\varpi^\ell}}=\(\begin{array}{ccc} 1 & \varpi^{\ell-k} & \varpi^{2\ell+1-2k}/2 \\ & 1 & \varpi^{\ell-k} \\ & & 1 \end{array}\)\(\begin{array}{ccc} \varpi^{-k} &  &  \\ & 1 &  \\ & & \varpi^k \end{array}\).\]
We deduce from this that $\Lambda(\diag(\varpi^{-k},1,\varpi^k)v)=0$ for $k >0$ by conjugating the element $M_{0,1}$ or $M_{0,\varpi} \in K$ past $\diag(\varpi^{-k},1,\varpi^k)$ and applying the definition of the functional.  The case $k=1$ together with the invariance property of the functional with respect to $M_{x,y}$ on the left forces $\sum_{\gamma \in \Sigma_T} \Lambda(\gamma v) = p^2\Lambda(\diag(\varpi,1,\varpi^{-1})v)$, so $\Lambda(\diag(\varpi,1,\varpi^{-1})v)$ vanishes as well.  It now follows easily by induction on $k$ that $\Lambda(\diag(\varpi^k,1,\varpi^{-k})v)$ vanishes for all positive integers $k$ by applying the operator $T^k$.  It follows immediately from this and the Iwasawa decomposition that $\pi_p$ is not generic, a contradiction to our hypothesis.
\end{proof}

\subsubsection{Satake transform of the standard $L$-function} We maintain the assumption that $p \ne 2$.  Define $\Delta(g)^s = |\nu(g)|^s \charf_{M_3(\mathfrak{O}_{E_p})}(g),$ where $\charf_{M_3(\mathfrak{O}_{E_p})}(\cdot)$ denotes the characteristic function of the ring of integers of $E_p$ and $|\cdot|$ is the $p$-adic valuation on $E_p$ with $|p| = p^{-1}$ and $|\varpi|=p^{-\frac{1}{2}}$.  (Here, $\varpi=\delta$ is a uniformizer of $\mathfrak{O}_{E_p}$ as before.)  For simplicity, we write $T,B,G,$ and $U$ for $T(\mathbf{Q}_p)$, $B(\mathbf{Q}_p)$, $G(\mathbf{Q}_p)$, and $U_B(\mathbf{Q}_p)$.  Suppose $\eta: T \rightarrow \C^\times$ is an unramified character and $\phi_{\eta} \in \ind_B^G(\delta_B^{1/2} \eta)$ is the $K$-spherical vector normalized so that $\phi_\eta(1)=1$.  We apply the Iwasawa decomposition to show that $I(\eta,s)=\int_{G}{\phi_{\eta}(g)\Delta(g)^s\,dg}$ is given by
\begin{align*} I(\eta,s) &= \int_{B}{\delta_B^{-1/2}(g)\eta(g) \Delta_{s}(g)\,dg} = \int_{T}{\delta_B^{-1/2}(t)\eta(t)|\nu(t)|^s \left(\int_{U}{\charf(ut)\,du}\right)\,dt}, \end{align*}
where we have shortened $\charf_{M_3(\mathfrak{O}_{E_p})}(\cdot)$ to $\charf(\cdot)$.

We now compute the unipotent integral.
\begin{lemma} \label{lem:charfactor} Write $u=M_{x,y}$, where $M_{x,y}$ is the matrix defined in (\ref{eqn:mab}) and $x \in \mathbf{Q}_p$, $y \in E_p$.  Let $t = \mathrm{diag}(t_1,t_2,t_3)$.  Then
\[\charf(ut) = \charf(t)\charf_{\mathfrak{O}_{E_p}}(x t_3) \charf_{\mathfrak{O}_{E_p}}(\delta y \ol{y} t_3).\]
Consequently, one has
\[\int_{U}{\charf(ut)\,du} = \charf(t)|t_3|^{-2}.\] \end{lemma}
\begin{proof} First suppose that $ut$ is integral.  Then $t$ is integral.  Examining the matrix $ut$, it also follows that $y t_3$, $\delta \ol{y} t_2$, and $(x+ \delta y\ol{y}/2)t_3$ are integral.  Since $\ol{x} = x$ and $\ol{\delta} = -\delta$, we deduce that both $x t_3$ and $\delta y\ol{y} t_3$ are integral.

Conversely, assume that $t$, $x t_3$, and $\delta y\ol{y} t_3$ are integral.  Set $a_i =\ord_{\varpi}(t_i) \in \mathbf{Z}$.  Since $t_1\ol{t_3} = t_2 \ol{t_2}$, we have $a_1 + a_3 = 2a_2$, and thus $2a_2 \geq a_3$.  Recalling that $\varpi =\delta$, for $ut$ to be integral, we only need
\[ \ord_{\varpi}(y) \geq -a_3 \textrm{ and } \ord_{\varpi}(y) \geq -(1 +a_2),\]
since the entry in position (1,3) is integral by our assumptions on $x t_3$ and $\delta y\ol{y} t_3$ (using $p \ne 2$).  The integrality of $\delta y\ol{y} t_3$ gives $2\ord_\varpi(y) \ge -(1+a_3)$ or $\ord_\varpi(y) \ge \lceil -(1+a_3)/2 \rceil$.  By considering the cases $a_3=0$ or $a_3 > 0$, we obtain $\ord_{\varpi}(y) \geq -a_3$.  Also, the inequality $2a_2\ge a_3$ implies $\lceil -(1+a_3)/2 \rceil \ge -a_2$, yielding the first part of the lemma.

For the second part, we consider the cases where $a_3 = 2k$ is even or $a_3= 2k+1$ is odd.  The first part gives a separation of the variables $x$ and $y$ from which it is easy to read off the result.  If $a_3 = 2k$, the measure of the $u$ for which $\charf(ut) \neq 0$ is the product of $p^k$ from $x$ and $p^k$ from $y$.  This product is thus $p^{2k}=|t_3|^{-2}$, as desired.  If $a_3 = 2k+1$, the measure is $p^k$ from $x$ and $p^{k+1}$ from $y$, which again combine to give $|t_3|^{-2}$. \end{proof}
 
Writing $t = \diag(t_1,t_2,t_3)$, we have
\[\delta_B(t) = |t_1/t_2|^2 |t_1/t_3| = |t_1/t_3|^2 = |\nu(t)|^2 |t_3|^{-4}.\]
Hence $\delta_B^{-\frac{1}{2}}(t)|t_3|^{-2} = |\nu(t)|^{-1}$, so
\begin{equation}\label{Ietas}I(\eta,s) = \int_{T}{\charf(t)\eta(t)|\nu(t)|^{s-1}\,dt}.\end{equation}
Define the characters $\eta_1,\eta_2$ via
\begin{equation}\label{eqn:etaidef} \eta(\diag(\tau_1 \tau_2,\tau_2, \ol{\tau_1}^{-1} \tau_2)) = \eta_1(\tau_1)\eta_2(\tau_2).\end{equation}
We write $\alpha_{\eta_1}$ and $\alpha_{\eta_2}$ for the values $\eta_1(\varpi)$ and $\eta_2(\varpi)$.  Then we may relate $I(\eta,s)$ to the standard local $L$-factor at $p$.  We include $\left(1-\alpha_{\eta_2} X\right)$ in the denominator of the following result to emphasize the link to this $L$-factor.
\begin{proposition} One has
\begin{equation}\label{eqn:ietas} I(\eta,s) = \frac{1-\alpha_{\eta_2}^2 X^2}{\left(1-\alpha_{\eta_1} \alpha_{\eta_2} X\right)\left(1-\alpha_{\eta_2} X\right)\left(1-\alpha_{\eta_1}^{-1}\alpha_{\eta_2} X\right)} = (1-\alpha_{\eta_2}^2 X)L^\mathrm{L}(\pi_p,\mathrm{Std},s),\end{equation}
where $X = p^{1-s}$. \end{proposition}

\begin{proof} For simplicity, write $\alpha = \alpha_{\eta_1}$ and $\beta = \alpha_{\eta_2}$.  From (\ref{Ietas}), we obtain
\[ I(\eta,s) = \sum_{k \geq 0, k \geq j \geq -k}{\alpha^j \beta^k X^k}.\]
For each $k \ge 0$, we obtain the identity
\[\sum_{k \geq j \geq -k}{\alpha^j} = \frac{\alpha^{k+1}+\alpha^k-\alpha^{-k}-\alpha^{-(k+1)}}{\alpha- \alpha^{-1}}\]
by considering the geometric series in $\alpha^j$ with $j\equiv k\pmod{2}$ separately from those $j$ with $j \equiv k-1\pmod{2}$.  Hence
\begin{align*} I(\eta,s) &= \frac{1}{\alpha - \alpha^{-1}} \left(\frac{\alpha+1}{1-\alpha \beta X} - \frac{1+\alpha^{-1}}{1-\alpha^{-1}\beta  X}\right) = \frac{1+\beta X}{(1-\alpha \beta X)(1-\alpha^{-1}\beta X)} \\ &= \frac{1-\beta^2 X^2}{\left(1-\alpha\beta X\right)\left(1-\beta X\right)\left(1-\alpha^{-1}\beta X\right)},\end{align*}
giving the first equality in (\ref{eqn:ietas}).

Observe that since the extension $E_p/\mathbf{Q}_p$ is ramified, the local $L$-factor from Definition (\ref{defin:stdl}) is given by the determinant of the element $(\mathrm{diag}(\alpha_{\eta_1}, (\alpha_{\eta_1} \alpha_{\eta_2})^{-1}, 1),\alpha_{\eta_2}) \rtimes 1 \in (\GL_3(\C) \times \mathbf{G}_m(\C)) \rtimes \mathrm{Gal}(E_p/\Q_p)$ acting on the inertial invariants of the 6-dimensional representation $r$.  The action of inertia factors through $\mathrm{Gal}(E_p/\mathbf{Q}_p)$, with the non-trivial element acting via $\mm{}{\mathbf{1}_3}{\mathbf{1}_3}{}$.  The second equality in (\ref{eqn:ietas}) follows immediately. \end{proof}

\subsubsection{Ramified integral}  We now compute the Whittaker integral
\[I(W,s) = \int_{G}{W(g)\Delta(g)^s\,dg},\]
where we have shortened $W_p$ to $W$.  We obtain
\begin{align*} I(W,s) &= \int_{B}{\delta_{B}^{-1}(g)W(g)|\mu(g)|^s\charf(g)\,dg} \\ &=\int_{B}{\delta_{B}^{-1}(t)W(t)|\mu(t)|^s \left(\int_{U}{\chi(u)\charf(ut)\,du}\right)\,dt}.\end{align*}
If $u=M_{x,y}$ and $t$ is integral, we use the notation $a_i = \ord_\varpi t_i$ from the proof of Lemma \ref{lem:charfactor}.  Write $a_3 = 2k$ or $a_3 = 2k+1$ according to whether it is odd or even.  Then Lemma \ref{lem:charfactor} yields, by separation of variables,
\begin{align*}\int_{U}{\chi(u)\charf(ut)\,du} &= \begin{cases} p^{k}\displaystyle \int_{\varpi^{-(k+1)}\mathfrak{O}_{E_p}}{\psi(\tr_{E_p/\mathbf{Q}_p}(\delta^{-1}y))\,dy} & a_3 \textrm{ odd}\\p^{k}\displaystyle \int_{\varpi^{-k}\mathfrak{O}_{E_p}}{\psi(\tr_{E_p/\mathbf{Q}_p}(\delta^{-1}y))\,dy} & a_3 \textrm{ even}.\end{cases}\end{align*}
It follows that $\int_{U}{\chi(u)\charf(ut)\,du} = 0$ if $a_3 > 0 $ and $\int_{U}{\chi(u)\charf(ut)\,du}=1$ if $a_3 = 0$.  Write $\charf_{a_3=0}(t)$ for the set of $t \in T$ satisfying $a_1 \ge 0, a_2\ge 0$, and $a_3 = 0$.  We have
\begin{align*} I(W,s) &= \int_{T}{\delta_{B}^{-1}(t) \charf_{a_3=0}(t)W(t)|\mu(t)|^{s}\,dt} = \int_{T}{\delta_{B_H}^{-1}(t) \charf_{a_3=0}(t)W(t)|\mu(t)|^{s-1}\,dt} \\ &= \int_{T_H}{\delta_{B_H}^{-1}(t) \charf_{a_3=0}(t)W(t)|\mu(t)|^{s-1}\,dt},\end{align*}
where for the first equality we use that $\delta_{B_H} = \delta_B^{\frac{1}{2}}$ and, when $a_3 = 0$, $\delta_B^{-\frac{1}{2}} = |\mu|^{-1}$.  For the second equality, note that the domain of integration can be restricted to $T_H$ since $a_3 = 0$.  Thus
\[ (1-\alpha_{\eta_2}^2 X^2)^{-1}I(\eta,s)W(1) = (1-\alpha_{\eta_2}^2 X^2)^{-1}I(W,s) = \int_{T_H}{\delta_{B_H}^{-1}(t) \charf(t)W(t)|\mu(t)|^{s-1}\,dt}\]
since $\alpha_{\eta_2}^2 = \eta_2(p)$ and $\eta_2$ is the central character.  Note also that the equality $I(W,s) = I(\eta,s)W(1)$ follows formally from the right $K$-invariance of the Whittaker function.  (See, for instance, \cite[(4.2)]{pollackShahKS}.)  Combining the above results, we have proved the following fact.
\begin{proposition} Suppose $\eta, \alpha_{\eta_1},$ and $\alpha_{\eta_2}$ are as above.  Then
\begin{equation} \label{eqn:ramlocal} \int_{T_H}{\delta_{B_H}^{-1}(t) \charf(t)W(t)|\mu(t)|^{s}\,dt} = \frac{1}{\left(1-\alpha_{\eta_1}\alpha_{\eta_2}p^{-s}\right)\left(1-\alpha_{\eta_2}p^{-s}\right)\left(1-\alpha_{\eta_1}^{-1}\alpha_{\eta_2}p^{-s}\right)}.\end{equation}
\end{proposition}
We observe that the left-hand side of (\ref{eqn:ramlocal}) is what is given by the local integral from the Rankin-Selberg calculation when $\nu_1$ is trivial and $\Phi$ is the characteristic function of $\mathbf{Z}_p \oplus \mathbf{Z}_p$.  To see this, note that after applying the Iwasawa decomposition, the only difference between the expression in (\ref{eqn:unfold2}) and the left-hand side of (\ref{eqn:ramlocal}) is that the former imposes no integrality condition on $t_1$.  However, the argument in Proposition \ref{prop:testvector} shows that if $W(t)\ne 0$, we must have $\ord_\varpi(t_1)\ge \ord_\varpi(t_3)$.

As a simple corollary, we also obtain the analogous result for the $\nu_1$-twisted $L$-function.

\begin{corollary} Suppose $\eta, \alpha_{\eta_1},$ and $\alpha_{\eta_2}$ are as above.  Furthermore, assume that the character $\nu_1$ is unramified at $p$, and write $\alpha_{\nu_1} = \nu_1(\varpi)$. Then
\begin{align} &\int_{T_H}{\delta_{B_H}^{-1}(t) \nu_1(\mu(t))\charf(t)W(t)|\nu(t)|^{s}\,dt} = \label{eqn:ramlocaltwist} \\ &\;\; \frac{1}{\left(1-\alpha_{\nu_1}^2\alpha_{\eta_1}\alpha_{\eta_2}p^{-s}\right)\left(1-\alpha_{\nu_1}^2\alpha_{\eta_2}p^{-s}\right)\left(1-\alpha_{\nu_1}^2\alpha_{\eta_1}^{-1}\alpha_{\eta_2}p^{-s}\right)}.\nonumber\end{align}
\end{corollary}
\begin{proof} Define $\eta': T\rightarrow \C^\times$ by $\eta'(t) = \eta(t) \nu_1(\mu(t))$, let $W'$ denote the associated Whittaker function, and let $\eta_1'$ and $\eta_2'$ denote the characters of $E_p^\times$ associated to $\eta'$ in the same way as in (\ref{eqn:etaidef}).  Then
\[\int_{T_H}{\delta_{B_H}^{-1}(t) \nu_1(\nu(t))\charf(t)W(t)|\nu(t)|^{s}\,dt} = \int_{T_H}{\delta_{B_H}^{-1}(t)\charf(t)W'(t)|\nu(t)|^{s}\,dt}.\]
Furthermore, $\eta_1'(\varpi)=\eta_1(\varpi)$ and $\eta_2'(\varpi) =\eta_2(\varpi)\nu_1(p)$.  The corollary follows. \end{proof}

\section{Hecke operators and the regulator pairing}\label{sec:hecke}
In this section, we explain how to replace the integral studied in Section \ref{sec:diff} with one that is compatible with Beilinson's regulator as described in Section \ref{subsec:embeddings}.  As explained there, in order to produce a class in $\mathrm{Ch}^2(\ol{S},1)$, we must write down a higher Chow cycle $\sum_i (D_i,f_i)$, where $D_i$ is a curve on $\overline{S}$ and $f_i$ is a rational function on $D_i$.  The key point is the vanishing condition (\ref{eqn:vancond}) on the divisors of the $f_i$.  In order to satisfy this condition while maintaining other needed properties of the class, we employ a technique of Manin and Drinfel'd.

\subsection{Overview} \label{subsec:heckeoverview}

Assume that we begin with a formal sum $\Xi = \sum_i (\ol{C}_i,\Xi_i)$ as in Proposition \ref{prop:cycles}, i.e.\ a sum of degree zero divisors $\Xi_i$ on connected components $\ol{C}_i \subseteq \ol{S}_{K_f}$ of compactified modular curves such that $\Xi_i$ is supported on the cusps of $\ol{C}_i$, the regulator pairing with a given form $\omega_{\varphi_f}$ is non-vanishing, and the sum is defined over the reflex field $E$ (or $\mathbf{Q}$ after the restriction of scalars discussed in Section \ref{subsec:beilinson}).  We would like a simple way of transforming this formal sum into a new one meeting the conditions of Proposition \ref{prop:cycles}.(2).  More precisely, we want to define an operation $T: \Xi \mapsto \Xi'$ yielding a new formal sum $\Xi' =\sum_k (\ol{C}_k',\Xi_k')$ such that $\deg(\Xi_k)=0$ and such that there exist divisors $\Psi_j$ supported on torsion points on each elliptic curve $E_j$ with $\deg(\Theta_j) = 0$ and
\begin{equation} \label{eqn:maineq2} \sum_k \Xi_k' + \sum_{j \in \sigma} \Psi_j = 0\end{equation}
as a formal sum of points.  Note that the $\Psi_j$ are uniquely determined by the $\Xi_k'$, so if $\sum_k (\ol{C}_k',\Xi_k')$ is defined over $E$, then the formal sum $\sum_{j \in \sigma} (E_j,\Psi_j)$ is also defined over $E$.  Proposition \ref{prop:cycles} would now produces the relevant higher Chow class as an integer multiple of $\sum_k (\ol{C}_k',\Xi_k')+\sum_{j \in \sigma} (E_j,\Psi_j)$.  It is also important that the transformation $\Xi \mapsto \Xi'$ also preserves desired properties of the regulator integral, such as nonvanishing.

To produce the needed transformation, we define a Hecke operator $T$ that annihilates all functions on the boundary of the minimal compactification of $S = S_{K_f}$ and has the property that the action $T:\Xi \mapsto \Xi'$ is adjoint via the regulator pairing with differential forms $\omega_{\varphi_f},$ $\varphi_f \in \pi_f$, to the usual action of $T$ on vectors in $\pi_f$. (Here $\pi_f$ is the finite part of a cuspidal automorphic representation, and $\omega_{\varphi_f}$ is defined as in Section \ref{subsec:diff}.)  In the remainder of this section, we explain how to define this action of $T$ as well as other actions on related objects.  We explain the relationships between these actions, and end by showing how to choose the Hecke operator $T$ such that it
\begin{enumerate}
	\item preserves the field of definition of the formal sum,
	\item forces $\deg(\Xi_k') = \deg(\Psi_j) = 0$ for all $j,k$, and
	\item acts by multiplication by a nonzero element of $\ol{\mathbf{Q}}^\times$ on the regulator integral.
\end{enumerate}
These will respectively be addressed in Sections \ref{subsec:pairs}, \ref{subsec:heckecusps}, and \ref{subsec:heckenonzero}.

Let $K_f$ be an open compact subgroup of $G(\A_f)$ and let $T = T_{\xi}$ denote a bi-$K_f$-invariant Hecke operator.  That is, $\xi \in C_c^\infty(K_f \backslash G(A_f) / K_f, \Z)$ and $T_{\xi}$ is the Hecke operator associated to $\xi$.  We define right and left coset decompositions of $\xi$ by the formula
\begin{equation} \label{eqn:xidef} \xi = \sum_{\alpha}{n_{\alpha} \charf(K_f g_{\alpha})} = \sum_{\beta}{n_{\beta} \charf(g_\beta K_f)}\end{equation}
with $n_{\alpha}, n_{\beta}$ in $\Z$.  We also write $\xi^\vee \in C_c^\infty(K_f \backslash G(\mathbf{A}_f) / K_f, \Z)$ for the function defined by $\xi^\vee(x) = \xi(x^{-1})$.  We next describe how $T$ acts on various objects that occur in this paper, and prove relations satisfied by such operators.

\subsection{Hecke actions on functions and differential forms}

We begin by reviewing the various automorphic interpretations of the action of a Hecke operator.  We begin with an action on functions.  Everything in this section will work for a general reductive group $G$ over $\mathbf{Q}$, where $S_{K_f}$ then denotes the locally symmetric space associated to $G$ and a neat open compact $K_f \subseteq G(\mathbf{A}_f)$.
\begin{defin}[Left action on functions of $G(\A_f)$] \label{defin:heckeleftfunc}
Suppose $f: G(\A_f) \rightarrow \C$ is a smooth (i.e.\ locally constant) function.  Then we define $T_\xi \cdot f$ by the formula
\[ (T_{\xi} \cdot f)(x) = \int_{G(\A_f)}{f(xg)\xi(g)\,dg}.\]
Here the Haar measure on $G(\mathbf{A}_f)$ is chosen so that $K_f$ has measure $1$.  If $f$ is right invariant by $K_f$, then $(T_{\xi} \cdot f)(x) = \sum_{\beta}{n_{\beta} f(xg_\beta)}$.
\end{defin}

We also define a right Hecke action on the space of functions on $G(\mathbf{A}_f)$. 
\begin{defin}[Right action on functions of $G(\mathbf{A}_f)$] \label{defin:heckerightfunc}
Suppose $f:G(\mathbf{A}_f) \rightarrow \mathbf{C}$ is a function.  We define $f \cdot T_{\xi}$ by
\[(f \cdot T_{\xi})(x) = \int_{G(\A_f)}{f(xg^{-1})\xi(g)\,dg} = \int_{G(\A_f)}{f(xg)\xi^\vee(g)\,dg}.\]
Thus $f \cdot T_{\xi} = T_{\xi^\vee} \cdot f$.  If $f$ is right-invariant by $K_f$, then $(f \cdot T_\xi)(x) = \sum_{\alpha}{n_{\alpha}f(x g_{\alpha}^{-1})}$.
\end{defin}

We next consider the case of differential forms.  Recall that $S_{K_f} = G(\Q)\backslash X \times G(\A_f)/K_f$.  Maintain the notation in (\ref{eqn:xidef}).  We define
\[K_f^\mathrm{R} = K_f \cap \bigcap_{\alpha} g_\alpha^{-1} K_f g_{\alpha}\textrm{ and }K_f^\mathrm{L} = K_f \cap \bigcap_{\beta} g_\beta K_f g_{\beta}^{-1}.\]

\begin{defin}[Action on differential forms, first form]
Suppose $\omega$ is a differential form on $S_{K_f}$, and $\xi$, $g_{\beta},$ and $n_{\beta}$ are as in (\ref{eqn:xidef}).  Define $\omega_{\beta} = T(g_\beta)^* \omega_f$, a differential form on $S_{g_\beta K g_{\beta}^{-1}}$.  (The translation operators $T(g)$ are defined in Definition \ref{defin:translation}.)  Define $\omega_\beta^\mathrm{L}$ to be the pullback of $\omega_\beta$ to $S_{K_f^\mathrm{L}}$ under the natural map $S_{K_f^\mathrm{L}} \rightarrow S_{g_\beta K_f g_{\beta}^{-1}}$ induced by the inclusion $K_f^\mathrm{L} \subseteq g_\beta K_f g_{\beta}^{-1}$.  Set $\omega^\mathrm{L} = \sum_{\beta}{n_{\beta} \omega_{\beta}^\mathrm{L}}$, a differential form on $S_{K_f^\mathrm{L}}$.
\end{defin}

We have the following simple lemma, whose proof we give for completeness.
\begin{lemma}\label{lem:pullbacks} Let $K_f' \subseteq K_f^\mathrm{L}$ be a subgroup of $K_f^\mathrm{L}$ that is normal inside the larger group $K_f$, so that for $k \in K_f$, the translation operator $T(k)^*$ is an endomorphism on the space of differential forms on $S_{K_f'}$.  Then the pullback $\omega'$ of $\omega^\mathrm{L}$ to $S_{K_f'}$ is invariant by the translation action of elements of $K_f$.  Consequently, there is a unique differential form $\widetilde{\omega}$ on $S_{K_f}$ with the property that $\omega^\mathrm{L}$ is the pullback of $\widetilde{\omega}$ to $S_{K_f^\mathrm{L}}$. \end{lemma}
\begin{proof} Since $S_{K_f'} \rightarrow S_{K_f}$ is \'{e}tale, the uniqueness is immediate.  For the existence of $\widetilde{\omega}$, it suffices to check the lemma in the case that $\xi$ is the characteristic function of a single double coset.

Suppose $k \in K_f$.  We must check that $T(k)^* \omega' = \omega'$.  Denote by $\omega_{\beta,k}$ the differential form $T(k)^*T(g_\beta)^* \omega = T(kg_\beta)^* \omega$ on $S_{kg_{\beta}K_fg_\beta^{-1}k^{-1}}$.  Then $\omega'$ is the sum of the pullbacks of the $\omega_{\beta}$'s to $S_{K_f'}$, and $T(k)^* \omega'$ is the sum of pullbacks of the $\omega_{\beta,k}$'s to $S_{K_f'}$.  But since $\xi = \charf(K_f g K_f)$ for some $g$, there exists a unique left coset representative $g_{\beta'}$ and a unique element $k'\in K_f$ such that $k g_{\beta} = g_{\beta'}k'$.  Moreover, the map $g_\beta \mapsto g_{\beta'}$ is a permutation of the coset representatives.  Hence 
\[\omega_{\beta,k} = T(k g_{\beta})^* \omega = T(g_{\beta'}k')^*\omega = T(g_{\beta'})^* T(k')^* \omega = T(g_{\beta'})^* \omega = \omega_{\beta'}.\]
Acting by $T(k)^*$ thus permutes the $\omega_{\beta}$'s, so $\omega'$ is invariant. \end{proof}

We will also need the following very similar construction later.  It is none other than the pushforward of a differential form by integration along fibers, but we make it explicit in terms of translation operators.
\begin{lemma}\label{lem:pushforward} Suppose that $K_f' \subseteq K_f$ is any containment of open compact subgroups and $\eta$ is any differential form on $S_{K_f'}$.  Let $p: S_{K_f'} \rightarrow S_{K_f}$ be the natural covering map, and let $K_f = \bigsqcup_{j} h_j K_f'$  be a left coset decomposition.  Then the formula
\begin{equation} \label{eqn:pushforward} p_{*}\eta = \sum_j T(h_j)^* \eta\end{equation}
can be used to define a differential form $p_*\eta$ on $S_{K_f}$.
\end{lemma}

\begin{proof}
We pick $K_f''$ normal in $K_f$ and contained in $K_f'$.  Then the termwise pullback of the right-hand side of (\ref{eqn:pushforward}) gives a well-defined differential form on $S_{K_f''}$, which is invariant by the translations $T(h_j)$.  Then the proof of Lemma \ref{lem:pullbacks} applies to define $p_*\eta$ as the unique differential form on $S_{K_f}$ that pulls back to this form on $S_{K_f''}$.
\end{proof}

\begin{remark}
A more concise definition of the pullback and pushforward can be given by using asymmetric Hecke operators $T_{\xi}, \xi \in C^\infty( K_f \backslash G(\mathbf{A}_f) / K_f')$.  We do not take this approach here.
\end{remark}

We may now give an action on differential forms that does not change the level.

\begin{defin}[Action on differential forms, second form] \label{defin:heckediff}
Using Lemma \ref{lem:pullbacks}, we may define $T_{\xi} \cdot \omega = \widetilde{\omega}$, a differential form on $S_{K_f}$.
\end{defin}

\subsection{Hecke action on pairs of cycles and smooth functions} \label{subsec:pairs}

In this section, we define a Hecke action on the following data.  We assume that $(G,X)$ is a Shimura datum with reflex field $E$ (which for now can be any number field, which we regard as having a fixed embedding in $\ol{\mathbf{Q}} \subseteq \mathbf{C}$) and write $S_{K_f}$ for the Shimura variety associated to a neat open compact $K_f \subseteq G(\mathbf{A}_f)$.  We also let $C$ be an irreducible cycle in $S_{K_f}$ and let $f$ be a smooth function on $C$.

For $g \in G(\A_f)$, define $(C,f) \cdot g = (T(g)(C), f\cdot g)$.  Here $T(g)(C)$ is the cycle on $S_{g^{-1}K_fg}$ given on points by the image of $C$ under $T(g)$ and $f\cdot g$ is defined as the pullback $f \cdot g = T(g^{-1})^*f$.  Note that since $T(g)$ is an isomorphism of varieties defined over $E$, the association $C \mapsto T(g)(C)$ is $\gal(\ol{\mathbf{Q}}/E)$-equivariant and $T(g^{-1})^*f$ is a smooth function on $T(g)(C)$.  Once we define the Hecke action, this will automatically give us the first needed property of $T_\xi$ discussed in Section \ref{subsec:heckeoverview}.

We have covering maps $p_{\alpha}: S_{K_f^\mathrm{R}} \rightarrow S_{ g_{\alpha}^{-1}K_fg_{\alpha}}$ induced by the inclusion $K_f^\mathrm{R} \subseteq g_{\alpha}^{-1}Kg_{\alpha}$, with notation as in (\ref{eqn:xidef}).  We define $(C,f) \cdot T_{\xi} = \sum_{\alpha}{n_{\alpha} p_{\alpha}^{*}\left( (C,f) \cdot g_{\alpha}\right)}$.  Here $p_{\alpha}^*((C',f'))$ means the pair $(p_{\alpha}^{-1}(C'),p_{\alpha}^*f')$.  Thus $(C,f) \cdot T_{\xi}$ lives on $S_{K_f^\mathrm{R}}$.

We have the following stability result concerning the pullback of a cycle.
\begin{lemma}\label{lem:stable}
Let $K_f' \subseteq K_f$ be a normal subgroup and let $p:S_{K_f'} \rightarrow S_{K_f}$ be the corresponding covering.  Then for any $h \in K_f$, $(p^*((C,f)))\cdot h = p^*((C,f))$.
\end{lemma}
\begin{proof}
We have two translation maps $T(h)=T_{K_f'}(h):S_{K_f'} \rightarrow S_{K_f'}$ and $T(h) = T_{K_f}(h):S_{K_f} \rightarrow S_{K_f}$ associated to $h \in K_f$, with the second map the identity.  It is clear from the definition of $T(h)$ as written in terms of double coset representatives that $p \circ T_{K_f'}(h)=T_{K_f}(h) \circ p=p$.  The needed result follows.
\end{proof}

\begin{defin}
Suppose the irreducible cycle $C$ on $S_{K_f}$ has real dimension $e$.  Given a differential $e$-form $\omega$ on $S_{K_f}$ together with a smooth function $f$ on $C$, we define
\begin{equation} \langle (C,f),\omega \rangle_{K_f} = \mathrm{meas}(K_f) \int_{C}{f \omega} \end{equation}
whenever this integral is absolutely convergent.
\end{defin}

The following lemma is clear from the definitions.
\begin{lemma}
Let $K_f' \subseteq K_f$ and let $p:S_{K_f'} \rightarrow S_{K_f}$ be the covering.  We have $\langle (C,f),\omega \rangle_{K_f} = \langle p^*((C,f)),p^*\omega \rangle_{K_f'}$.
\end{lemma}

We check that the pullback and pushforward maps are adjoint.
\begin{lemma} \label{lem:pbadj} Suppose $K_f' \subseteq K_f$ are two open compact subgroups, $C \subseteq S_{K_f}$ is an irreducible cycle of real dimension $e$, $f$ is a smooth function on $C$, and $\omega$ is a differential $e$-form on $S_{K_f'}$.  Let $p:S_{K_f'} \rightarrow S_{K_f}$ be the covering map.  Then
\[\gen{p^*((C,f)),\omega }_{K_f'} = \frac{1}{[K_f:K_f']}\gen{(C,f),p_*\omega}_{K_f}.\]
\end{lemma}

\begin{proof}
Fix $h_j$ such that $K_f = \bigsqcup_j h_j K_f'$.  Let $K_f'' \subseteq K_f$ be normal and contained in $K_f'$.  Let $q: S_{K_f''} \rightarrow S_{K_f}$ be the covering map.  Then the left and right-hand sides are respectively equal to $\gen{q^*p^*((C,f)),q^*\omega }_{K_f''}$ and $\frac{1}{[K_f:K_f']}\gen{q^*p^*((C,f)),q^*p^*p_*\omega }_{K_f''}$.  By definition, we have $q^*p^*p_*\omega = \sum_j q_{h_j}^*T(h_j)^*\omega$, where $q_{h_j}: S_{K_f''} \rightarrow S_{h_j^{-1}K_f'h_j}$ is the covering.  We have $T(h_j)\circ q_{h_j} = q \circ T(h_j)$, so we may rewrite $q^*p^*p_*\omega = \sum_j T(h_j)^*q^*\omega$.  Moreover, we have
\[\gen{q^*p^*((C,f)),T(h_j)^*q^*\omega}_{K_f''}=\langle T(h_j^{-1})^*q^*p^*((C,f)),q^*\omega\rangle_{K_f''} = \gen{q^*p^*((C,f)),q^*\omega}_{K_f''}\]
by Lemma \ref{lem:stable}.  The desired equality follows.
\end{proof}

We now show that the Hecke action on the data $(C,f)$ is adjoint to the one considered previously on differential forms.
\begin{lemma} \label{lem:heckeduality} Suppose $C \subseteq S_{K_f}$ is an irreducible cycle of real dimension $e$, $f$ is a smooth function on $C$, and $\omega$ is a differential $e$-form on $S_{K_f}$.  Let $\omega'$ be the pullback of $\omega$ to $S_{K_f^\mathrm{R}}$.  Then $\langle (C,f) \cdot T_{\xi}, \omega' \rangle_{K_f^\mathrm{R}} = \langle (C,f), T_{\xi} \cdot \omega \rangle_{K_f}$. \end{lemma}
\begin{proof} We define
\begin{itemize}
\item $K_f^{\alpha} = g_{\alpha}^{-1}K_fg_{\alpha}$,
\item maps $p_{\alpha}: S_{K_f^\mathrm{R}} \rightarrow S_{K_f^{\alpha}}$, $q_{\alpha}: S_{K_f \cap K_f^\alpha} \rightarrow S_{K_f}$, and $r_{\alpha}: S_{K_f \cap K_f^\alpha} \rightarrow S_{K_f^\alpha}$ induced by inclusions of open compacts, and
\item elements $h_{\alpha,j} \in G(\mathbf{A}_f)$ for $j \in J_\alpha$ chosen such that $K_f^\alpha = \bigsqcup_j h_{\alpha,j} (K_f \cap K_f^\alpha).$
\end{itemize}

We have $\langle (C,f) \cdot T_{\xi}, \omega'\rangle_{K_f^\mathrm{R}} = \sum_{\alpha}{n_{\alpha} I_{\alpha}}$, where $I_{\alpha} = \langle p_{\alpha}^*((C,f) \cdot g_{\alpha}), \omega'\rangle_{K_f^\mathrm{R}}$.  Now
\begin{align*} I_{\alpha} = \langle r_{\alpha}^*\left( (C,f) \cdot g_{\alpha}\right),q_{\alpha}^* \omega \rangle_{K_f \cap K_f^{\alpha}} &= \frac{1}{[K_f^\alpha:K_f^\alpha \cap K_f]}\langle (C,f) \cdot g_{\alpha}, (r_{\alpha})_{*} q_{\alpha}^* \omega\rangle_{K_f^\alpha} \\ &= \frac{1}{[K_f^\alpha:K_f^\alpha \cap K_f]}\langle (C,f), T(g_{\alpha})^* (r_{\alpha})_{*} q_{\alpha}^* \omega\rangle_{K_f}.\end{align*}

We have
\[K_f g_{\alpha} = g_{\alpha}K_f^\alpha = \bigsqcup_j g_{\alpha}h_{\alpha,j} (K_f\cap K_f^{\alpha})\]
and thus
\[\xi = \sum_{\alpha} \sum_{j \in J_\alpha} n_{\alpha} \charf(g_{\alpha}h_{j_{\alpha}} (K_f \cap K_f^{\alpha})).\]
It follows that
\begin{align*} \langle (C,f) \cdot T_{\xi}, \omega' \rangle_{K_f^\mathrm{R}} = \sum_{\alpha}{n_{\alpha}I_{\alpha}} &= \sum_{\alpha}{\frac{n_\alpha}{[K_f^\alpha:K_f^\alpha \cap K_f]} \langle (C,f), T(g_{\alpha})^* (r_\alpha)_* q_{\alpha}^* \omega\rangle_{K_f}}\\ &= \langle (C,f), T_{\xi} \cdot \omega \rangle_{K_f},\end{align*} 
where for the final equality we compare the two sides at an auxilliary level normal in $K_f$ as in the proof of Lemma \ref{lem:pbadj}.\end{proof}

\subsection{Actions on cusps} \label{subsec:heckecusps}

We now return to the case considered in Section \ref{sec:higherchow}, so that the boundary of the Baily-Borel compactification of $S_{K_f}$ is just a finite set of points.  Recall that the cusps $\sigma_{K_f}$ of $S_{K_f}$ are naturally in bijection with the double coset space $G(\Q) \backslash \mathbf{P}(V^0) \times G(\A_f)\slash K_f$, where $\mathbf{P}(V)^0$ denotes the set of isotropic $E$-lines in $V$.  Fixing an isotropic $E$-line $\ell_0$, one obtains $\sigma_{K_f} = B(\Q) \backslash G(\A_f)\slash K_f$, where $B \subseteq G$ denotes the Borel subgroup stabilizing $\ell_0$.  

There is a natural translation action on the cusps $\sigma_{K_f}$.  If $P = G(\Q)(\ell, x)K_f \in \sigma_K$, and $g \in G(\A_f)$, then 
\[T(g)(P) = P \cdot g = G(\Q)(\ell, xg)g^{-1}Kg \in \sigma_{g^{-1}Kg}\]
is a cusp on $S_{g^{-1}Kg}$.

Recall that we have studied a class of cycles $C$ on $S_{K_f}$ in Section \ref{sec:higherchow}.  We can, of course, consider a slightly broader class of cycles by translating the $C$'s there via the maps $T(g)$, but in our final result we will only need these cycles (see Section \ref{subsec:heckenonzero}).  (To simplify the notation, we will simply write $C$ for the cycle rather than $C_{W,K_{W,f}}$.)  Write $\iota_C^\mathrm{min}: \ol{C} \rightarrow S_{K_f}^\mathrm{min}$ for the map of minimal compactifications and write $\sigma_C$ for the cusps of $C$.  The obvious map of double cosets gives the map $\iota_C^\mathrm{min}|_{\sigma_C}: \sigma_C \rightarrow \sigma_{K_f}$.  This map is not necessarily injective in general.

We write $\mathbf{Z}[\sigma_{K_f}]$ and $\mathbf{Z}[\sigma_C]$ for the free abelian groups generated by the elements of $\sigma_{K_f}$ and $\sigma_C$, respectively.  We write $\iota_C: \mathbf{Z}[\sigma_C] \rightarrow \mathbf{Z}[\sigma_{K_f}]$ for the map induced by $\iota_C^\mathrm{min}|_{\sigma_C}$.  By abuse of notation we use the same letter to denote either a point of $\sigma_{K_f}$ or the corresponding generator of $\mathbf{Z}[\sigma_{K_f}]$ (and similarly for $\sigma_C$).  The action of $T_{\xi}$ on $\Z[\sigma_K]$ may be defined just as we defined the action on $(C,f)$ above.  More precisely, for $P \in \sigma_{K_f}$, we define $P \cdot T_{\xi} \in \Z[\sigma_{K_f^\mathrm{R}}]$ by
\[P \cdot T_{\xi} = \sum_{\alpha}{n_{\alpha} p_{\alpha}^{-1}(T(g_{\alpha})(P))}.\]
Here, we are using the notations of (\ref{eqn:xidef}) and defining $p_{\alpha}: S_{K_f^\mathrm{R}} \rightarrow S_{K_f^\alpha}$ to be the covering associated to $K_f^\mathrm{R}\subseteq K_f^\alpha$.  In particular, we obtain a map that we denote $T_\xi^{\partial\mathrm{-div}}: \mathbf{Z}[\sigma_{K_f}] \rightarrow \mathbf{Z}[\sigma_{K_f^\mathrm{R}}]$.

We write $\mathbf{CD}_{K_f}$ for the set of $\mathbf{Z}$-linear formal sums of pairs $(C,u)$ such that $C$ is one of the aforementioned cycles on $S_{K_f}$ and $u$ is a rational function on $\ol{C}$ with zeros and poles only at the boundary.  Given a rational function $u$ on $C$, we obtain as the image of the divisor of $u$ under $\iota_C$ the image divisor $\mathrm{div}((C,u)) \in \mathbf{Z}[\sigma_{K_f}]$.  In particular, we obtain a map
\[\mathbf{div}_{K_f}: \mathbf{CD}_{K_f} \rightarrow \mathbf{Z}[\sigma_{K_f}].\]
Our left action of the Hecke operator $T_\xi$ is a morphism we denote $T_\xi^\mathbf{CD}: \mathbf{CD}_{K_f} \rightarrow \mathbf{CD}_{K_f^\mathrm{R}}$.  It follows formally from writing the definitions in terms of double cosets that the Hecke action and boundary maps are compatible in the sense that
\begin{equation}\label{eqn:compat} T_\xi^{\partial\textrm{-}\mathrm{div}} \circ \mathbf{div}_{K_f} = \mathbf{div}_{K_f^\mathrm{R}} \circ T_\xi^\mathbf{CD}.\end{equation}

The action of $T_\xi$ on $\mathbf{Z}[\sigma_{K_f}]$ is well-behaved in the following sense.
\begin{lem}
Suppose that $K_f' \subseteq K_f^\mathrm{R}$ is normal in $K_f$.  Then for each $P \in \sigma_{K_f}$, the pullback of $P \cdot T_\xi$ to $\sigma_{K_f'}$ is $T(h)$-invariant for every $h \in K_f$.  It follows that $P \cdot T_\xi$ is the unique pullback of an element of $\mathbf{Z}[\sigma_{K_f}]$.  In particular, we may define an action $T_\xi^{\partial\mathrm{-div}}:\mathbf{Z}[\sigma_{K_f}] \rightarrow \mathbf{Z}[\sigma_{K_f}]$.
\end{lem}
\begin{proof}
Without loss of generality, assume that $\xi$ is the characteristic function of a single double coset operator, so that the $n_\alpha=1$ in (\ref{eqn:xidef}).  By definition, the pullback is the sum
\[\sum_{\alpha} q^{-1} \circ p_{\alpha}^{-1}(T(g_{\alpha})(P)),\]
where $q: S_{K_f'} \rightarrow S_{K_f^\mathrm{R}}$ is the covering.  (These preimage maps are well-defined on the divisor groups.)  As an action on double cosets, the map $T(h)$ is represented by right-multiplication of the $G(\mathbf{A}_f)$ factor.  We write $\alpha_h$ for the index of the unique representative $g_{\alpha_h}$ such that $g_\alpha h_\alpha = kg_{\alpha_h}$ for $k \in K_f$.  Then $T(h)T(g_\alpha) =T(g_{\alpha_h})$ as translations on $S_{K_f}$.  The pullback maps $q^{-1}$ and $p_\alpha^{-1}$ correspond to taking full preimages.  It follows from the definition on coset representatives as in Lemma \ref{lem:stable} that $T(h)$ commutes with the pullback maps in the sense that $T(h) \circ q^{-1} \circ p_\alpha^{-1} = q^{-1} \circ p_{\alpha_h}^{-1} \circ T(h)$.  In particular, the action of $T(h)$ merely permutes the elements of the above sum.
\end{proof}

We now would like to understand the action of $T_\xi$ on $\mathbf{Z}[\sigma_{K_f}]$ in a computable form.  There is an identification
\begin{equation} \label{eqn:dual} \mathbf{Z}[\sigma_{K_f}] \otimes_\mathbf{Z} \mathbf{C} \iso \mathrm{Maps}(\sigma_{K_f},\C)\end{equation}
obtained by defining the image of $P$ to be its characteristic function.  More importantly, these spaces are naturally dual via the obvious pairing, so as one might expect, there is a natural left action of a Hecke operator $T_\xi$, $\xi \in C_{c}^\infty\left(K_f \backslash G(\Q_p) \slash K_f,\mathbf{C}\right)$ on $\Psi \in \mathrm{Maps}(\sigma_{K_f},\C)$.  Such an action was already given in Definition \ref{defin:heckeleftfunc} by
\[(T_\xi \cdot \Psi)(x) = \int_{G(\mathbf{A}_f)} \Psi(xg)\xi(g)dg.\]
If we write $\langle \cdot,\cdot \rangle$ for the natural pairing
\[ \langle \cdot, \cdot \rangle: \mathbf{C}[\sigma_{K_f}] \times \mathrm{Maps}(\sigma_{K_f},\C) \rightarrow \mathbf{C},\]
then the right action of $T_\xi$ on the first factor is the adjoint of the left action of $T_{\xi^\vee}$ on the right factor.  This can be seen from the definition by using the inverses of the representatives $g_\alpha$ used in the right coset decomposition of $\xi$ to give a left coset decomposition for $\xi^\vee$; it is then clear that $T_\xi$ goes to $T_{\xi^\vee}$ under the identification (\ref{eqn:dual}), from which the result follows by symmetry of the induced pairing on $\mathbf{C}[\sigma_{K_f}]  \times \mathbf{C}[\sigma_{K_f}]$.   If we determine a Hecke operator $T_{\xi}$ that annihilates $\mathrm{Maps}(\sigma_{K_f},\C)$, it follows that the Hecke operator $T_{\xi^\vee}$ annihilates $\mathbf{C}[\sigma_{K_f}]$.

We have the following lemma.
\begin{lemma} \label{lem:cuspsind}
Denote by $K_f' = \cap_{k \in K_\mathrm{max}}{k K_f k^{-1}}$ the largest normal subgroup of $K_\mathrm{max}$ contained in $K_f$.  Here, $K_\mathrm{max}$ is the maximal compact $K_\mathrm{max} \subseteq G(\A_f)$ defined in Section \ref{subsec:shimura}.  Then there are natural identifications between the spaces
\begin{enumerate}
  \item $\mathrm{Maps}(\sigma_{K_f},\C)$,
  \item $\mathrm{Maps}(B(\Q)U_{B}(\A_f) \backslash G(\A_f) \slash K,\C)$, where $U_B$ is the unipotent radical of $B$, and
  \item the direct sum of the (finitely many) induced spaces $\bigoplus_{\chi}{\mathrm{Ind}_{B(\A_f)}^{G(\A_f)}(\chi)^{K_f}}$, where $\chi$ runs over the characters $T(\Q)\backslash T(\A_f)/K_{T} \rightarrow \C^\times$, where $K_T = K_f' \cap T(\A_f)$.
\end{enumerate} \end{lemma}
\begin{proof} The equivalence of the first and second spaces follows from an approximation argument: every function $B(\Q)\backslash G(\A_f) \slash K \rightarrow \C$ is automatically left-invariant under $U_{B}(\A_f)$.

The equivalence of the second and third spaces is essentially the definition of the induced space.  More precisely, if $f$ is a nonzero section in $\mathrm{Ind}_{B(\A_f)}^{G(\A_f)}(\chi)$ for some character $\chi: T(\Q) \backslash T(\A_f) \rightarrow \C^\times$, and if $f$ is right-invariant under the open compact subgroup $K_f$, then $\chi$ is trivial on $K_T = T(\A_f)\cap K_f'$.  Indeed, if $k \in K_{T}$ and $h \in K_\mathrm{max}$ is such that $f(h) \neq 0$ -- such a $h$ always exists by the Iwasawa decomposition -- then
\[\chi(k)f(h) = f(kh) = f(h (h^{-1}kh)) = f(h)\]
since $K_f'$ is normal in $K_\mathrm{max}$.

Finally, the direct sum in the third space is finite since $T(\Q) \backslash T(\A_f) \slash K_{T}$ is a finite abelian group.\end{proof}

The following lemma can be used to show how to achieve the second needed property of $T_\xi$ discussed in Section \ref{subsec:heckeoverview}.
\begin{lemma} \label{lem:annihilate}
Suppose that the open compact subgroup $K_f$ contains $K_f(N)$, which we define to be the subgroup of $G(\A_f)$ consisting of matrices congruent to $1$ modulo $N$.  Let $p$ be a prime congruent to 1 modulo $N$, and let $T_\xi$ for $\xi = \charf{K_f^p} \otimes \xi_p$, $\xi_p \in C_{c}^\infty\left(K_{p,\mathrm{max}} \backslash G(\Q_p) \slash K_{p,\mathrm{max}},\mathbf{C}\right)$, be a non-zero Hecke operator at $p$.  (We write $K_f^p$ for the part of $K_f$ away from $p$.)  Then $T_\xi - \mu(\xi)$ annihilates the space of maps $B(\Q)\backslash G(\A_f)\slash K \rightarrow \C$, where $\mu(\xi) \in \C$ is defined by
\[\mu(\xi) = \int_{G(\Q_p)}{\xi(g)\,dg}.\]
\end{lemma}
\begin{proof}
By Lemma \ref{lem:cuspsind}, we must check that $T_\xi - \mu(\xi)$ annihilates $\mathrm{Ind}_{B(\A_f)}^{G(\A_f)}(\chi)^{K}$.  Since $p$ does not divide $N$, $K_p = K_{p,\mathrm{max}}$.  Thus it suffices to assume that $\xi$ is the characteristic function of a single double coset of the form
\[K_{p,\mathrm{max}} t K_{p,\mathrm{max}} = \bigsqcup_{\alpha} u_{\alpha}t_{\alpha}K_{p,\mathrm{max}},\]
where $u_\alpha t_\alpha$ are respectively in $U(\mathbf{Q}_p)$ and $T(\mathbf{Q}_p)$.

If $f \in \mathrm{Ind}_{B(\A_f)}^{G(\A_f)}(\chi)^{K}$, then $(T_\xi\cdot f)(1) = \sum_{\alpha}{\chi(t_\alpha) f(1)}$.  Since $p$ is $1$ modulo $N$, $\chi(t_\alpha) = 1$.  Thus, $(T_\xi\cdot f)(1) = \left(\sum_{\alpha}{1}\right)f(1) = \mu(\xi) f(1)$, and the lemma follows. \end{proof}

\subsection{A nonvanishing argument} \label{subsec:heckenonzero}

We end the section by showing how to construct $T_\xi$ so that the effect of $T_\xi - \mu(\xi)$ upon the Rankin-Selberg integral is by multiplication by a nonzero element of $\ol{\mathbf{Q}}^\times$.  Let $p$ be a prime so that the finite data $\varphi_f$ used to define the differential form $\omega_{\varphi_f,\alpha}$ considered in Section \ref{subsec:diff} is spherical at the prime $p$, and so that the level subgroup $K_f$ is maximal at $p$.  We consider only Hecke operators of the form $T_\xi - \mu(\xi)$ considered in Lemma \ref{lem:annihilate}.  By Lemma \ref{lem:heckeduality}, such a Hecke operator can be interpreted on the one hand as providing a cycle whose boundary divisor along the minimal compactification is trivial (due to Lemma \ref{lem:annihilate}), and on the other, since $\phi_p$ is spherical, by the \emph{scalar} action of $T_\xi - \mu(\xi)$ on the vector $\phi_p$.  It remains to show that this scalar can be made nonzero.  There is a very simple proof that such a $\xi$ exists.
\begin{lem} \label{lem:heckenonzero}
Suppose that $\pi_p$ is not the trivial representation, and that $\varphi_f$ is spherical at $p$.  There exists an $\mathbf{Q}$-valued $\xi$ as in Lemma \ref{lem:annihilate} such that $T_\xi - \mu(\xi)$ acts nontrivially on $\varphi_f$.
\end{lem}

\begin{proof}
There is an equivalence of categories between spherical representations of $G(\mathbf{Q}_p)$ and representations of the spherical Hecke algebra given by mapping a representation $\pi_p$ to the Hecke module of a spherical vector.  If every $\mathbf{Q}$-valued $T_\xi$ acts on $\phi_p$ by $\mu(\xi)$, so does every $\mathbf{C}$-valued $T_\xi$, so $\pi_p$ is the trivial representation.
\end{proof}

Note that the representations $\pi_p$ arising in Section \ref{sec:diff} are generic and thus not the trivial representation.  This explains how to obtain the final desired property described in Section \ref{subsec:heckeoverview}.

\begin{remark}
Assume that $p$ is an inert place.  Then $G$ is the product of the center with the unitary group $\U(2,1)(\mathbf{Q}_p)$, so the restriction of $\pi_p$ to $\U(2,1)(\mathbf{Q}_p)$ remains irreducible.  Since $\U(2,1)(\mathbf{Q}_p)$ and $\SU(2,1)(\mathbf{Q}_p)$ have the same spherical Hecke algebra, the further restriction of $\pi_p$ to $\SU(2,1)(\mathbf{Q}_p)$ is again irreducible.  If this restriction is not the trivial representation, we can obtain a $\xi$ whose double coset (by strong approximation) may be represented by an element of $G(\mathbf{Q})$.  Then, examining the definition of $T_\xi$, we may represent the cycle of integration in the translated Rankin-Selberg integral by only the curves $C_W$ of Section \ref{sec:higherchow} rather than the wider class of adelic transformations $T(g)(C_W)$.
\end{remark}


\section{Kronecker limit formula} \label{sec:klf}


We give three Kronecker limit formulae: one for classical Eisenstein series and two for adelic Eisenstein series with $\Gamma(N)$ or $\Gamma_0(N)$ level structure, respectively.  We begin by proving the formula (\ref{eqn:adelicclassical}), which relates our adelic $\GL_2$ Eisenstein series with a classical Eisenstein series on the upper half-plane.  We then give the classical Kronecker limit formula for this series and use (\ref{eqn:adelicclassical}) to deduce the two adelic versions.

Everything in this section is based on classical results; references include the books of Kubert and Lang \cite{kl} and Kato \cite[Chapter I.3]{kato}.

\subsection{Eisenstein series: from adelic to classical} For a ring $R$, define $V_{2}(R)$ to be the $\GL_2(R)$-module of row vectors of length 2 with entries in $R$.  Write $e=(1,0)$ and $f=(0,1)$ for the standard basis of $V_2(R)$.  Recall from Definition \ref{defin:eis} that for a Schwartz-Bruhat function $\Phi= \Phi_f \otimes \Phi_{\infty}$ on $V_2(\mathbf{A})$, we set
\[f(g,\Phi,\nu,s) = \nu_1(\det(g))|\det(g)|^{s}\int_{\GL_1(\A)}{\Phi(t(0,1)g)(\nu_1\nu_2^{-1})(t)|t|^{2s}\,dt}.\]
Here $\nu = (\nu_1,\nu_2): T(\mathbf{A})/T(\mathbf{Q}) \rightarrow \mathbf{C}^\times$ is an automorphic character of the diagonal torus of $\GL_2$ that we assume to be trivial at $\infty$.  The Eisenstein series is then
\[E(g,\Phi,\nu,s) = \sum_{\gamma \in B(\Q)\backslash \GL_2(\Q)}{f(\gamma g,\Phi,\nu,s)}.\]
Let $\widehat{\nu} = (\nu_2,\nu_1)$.  The Eisenstein series satisfies the functional equation
\begin{equation}\label{EisFE} E(g,\Phi,\nu,s) = E(g,\widehat{\Phi},\widehat{\nu},1-s),\end{equation}
where $\widehat{\Phi}: V_2(\mathbf{A}) \rightarrow \mathbf{C}$ is the Fourier transform of $\Phi$ given by
\[\widehat{\Phi}(\upsilon) = \int_{V_2(\A)}{\psi(\langle \xi,\upsilon\rangle_\mathbf{A})\Phi(\xi)\,d\xi}.\]
Here, $\psi$ is the standard additive character and $\langle \cdot,\cdot \rangle_R$ denotes the symplectic pairing on $V_2(R)$ given by $\langle e,f \rangle = 1 = -\langle e,f \rangle$, $\langle e,e \rangle = \langle f, f\rangle=0$.  If $\Phi_v$ is a Schwartz-Bruhat function on $V_2(\mathbf{Q}_v)$, where $v$ is any place of $\mathbf{Q}$, we again write $\widehat{\Phi}_v(\upsilon) = \int_{V_2(\Q_v)}{\psi(\langle \xi,\upsilon\rangle_{\mathbf{Q}_v})\Phi_v(\xi)\,d\xi}$.

Given a Schwartz-Bruhat function $\Phi=\Phi_f \otimes \Phi_{\infty}$ and $M \in \mathbf{A}_f$, if we define $\Phi_M$ by $\Phi_M(v) = \Phi_f(M^{-1}v_f) \otimes \Phi_{\infty}(v_\infty)$, then $\mathrm{Supp}(\Phi_{M,f}) = M \cdot\mathrm{Supp}(\Phi_f)$ and $E(g,\Phi_M,\nu,s) = |M|_{\A_f}^{-2s}E(g,\Phi,\nu,s)$.  So to understand $E(g,\Phi,\nu,s)$ in general, it is enough to consider functions $\Phi_f$ that are supported on $\widehat{\Z}^2$.
\begin{defin}
Throughout this section, $\Phi_{\infty}$ will denote the Gaussian on $V_2(\R)$ defined by $\Phi_{\infty}(x,y) = e^{-\pi(x^2+y^2)}$.
\end{defin}

Write $\GL_2(\mathbf{R})^+$ for the subgroup of matrices with positive determinant and recall that this group acts on the upper half-plane by M\"obius transformations.  Since we have fixed $\Phi_\infty$ and assumed $\nu_\infty$ is trivial, we write $f_{\infty}(g_{\infty},s)$ for $f(g_\infty,\Phi_\infty,\nu_\infty,s)$.  If $g_{\infty} \in \GL_2(\R)^{+}$, then
\begin{equation} \label{eqn:finfty} f_{\infty}(g_{\infty},s) = \Gamma_{\R}(2s) y^{s},\end{equation}
where $g_{\infty}\cdot i = x + iy$.  Here, $\Gamma_{\R}(s) = \pi^{-s/2}\Gamma(s/2)$.

For a character $\eta: \mathbf{Q}^\times \backslash \mathbf{A}^\times \rightarrow \mathbf{C}^\times$ and $(m,n) \in V_2(\mathbf{Z})$, define
\begin{equation} c(\Phi_f,\eta,(m,n)) = \int_{\GL_1(\widehat{\Z})}{\eta(r)\Phi_f(r(m,n)_f)\,dr},\end{equation}
where we write $(m,n)_\infty$ and $(m,n)_f$ for the images of $(m,n)$ under the natural maps $\mathbf{Z} \rightarrow \mathbf{R}$ and $\mathbf{Z} \rightarrow \mathbf{A}_f$, respectively.  For fixed $\eta$ and $\Phi$, choose $N \in \mathbf{Z}_{\ge 1}$ with the property that $\Phi$ is stable by $N \widehat{\Z}^2 \subseteq V_2(\mathbf{A}_f)$ and $\eta$ is trivial on $(1 + N\widehat{\Z})^\times \subseteq \A_f^{\times}$.  Then for any fixed $k \in \GL_2(\widehat{\mathbf{Z}})$, the value of $c(k\cdot \Phi,\eta,(m,n))$ only depends on the residue class of $(m,n)$ in $V_2(\Z/N\mathbf{Z})$.

For a residue class $\overline{w} \in V_2(\Z/N\mathbf{Z})$, we define the function $E_{\overline{w},N}(z,s)$ on the upper half-plane by
\begin{equation} E_{\overline{w},N}(z,s) =\Gamma_{\R}(2s) \sum_{\substack{(m,n) \in \Z^2\setminus\{(0,0)\}\\ (m,n) \equiv \overline{w}\ (\textrm{mod }N)}}{\frac{y^s}{|mz+n|^{2s}}}.\end{equation}
We may write $E(g,\Phi,\nu,s)$ in terms of the functions $E_{\overline{w},N}(z,s)$ as follows.
\begin{lemma} \label{lem:adelicclassical} Suppose $\Phi_f$ is supported in $\widehat{\Z}^2$, and write $g \in \GL_2(\mathbf{A})$ in the form $g = g_\Q g_{\infty}k$ for $g_{\Q} \in \GL_2(\Q)$, $g_\infty \in \GL_2(\R)^{+}$ and $k \in \GL_2(\widehat{\Z})$.  Define $z = g_{\infty} \cdot i$.  Then
\begin{equation} \label{eqn:adelicclassical} E(g,\Phi,\nu,s) = \nu_1(\det(k))\sum_{\overline{w} \in V_2(\Z/N\mathbf{Z})}{c(k \cdot\Phi_f,\nu_1\nu_2^{-1},\overline{w}) E_{\overline{w},N}(z,s)}.\end{equation}
\end{lemma}
\begin{proof} Note that $B^{+}(\Z)\backslash \SL_2(\Z) \rightarrow B(\Q)\backslash \GL_2(\Q)$ is a bijection.  Using (\ref{eqn:finfty}), we obtain
\[E(g,\Phi,\nu,s) = \nu_1(\det(k))\Gamma_{\R}(2s)\sum_{\substack{\gamma \in B^{+}(\Z)\backslash \SL_2(\Z)\\ \gamma = \mm{a}{b}{c}{d}}}{\frac{y^s}{|cz+d|^{2s}}\int_{\GL_1(\A_f)}{(\nu_1\nu_2^{-1})(t)|t|_{\A_f}^{2s}(k\cdot\Phi)(t(c,d))\,dt}}.\]

Note that $\GL_1(\A_f) = \bigsqcup_{t \in \mathbf{Q}_{>0}}{t\GL_1(\widehat{\Z})}$. With the assumption $\mathrm{Supp}(\Phi)\subseteq \widehat{\mathbf{Z}}^2$, we have
\begin{align*}
   &E(g,\Phi,\nu,s) \\
	=& \nu_1(\det(k))\Gamma_{\R}(2s)\sum_{\gamma \in B^{+}(\Z)\backslash \SL_2(\Z), t \in \Z_{>0}}{\frac{y^s}{|cz+d|^{2s}}\int_{\GL_1(\widehat{\Z})}{(\nu_1\nu_2^{-1})(r)|t|_{\A_f}^{2s}(k\cdot\Phi)(tr(c,d)_f)\,dr}}\\
	=& \nu_1(\det(k))\Gamma_{\R}(2s)\sum_{(m,n) \in \Z^2\setminus{\set{(0,0)}}}{\frac{y^s}{|mz+n|^{2s}}\int_{\GL_1(\widehat{\Z})}{(\nu_1\nu_2^{-1})(r)(k\cdot\Phi)(r(m,n)_f)\,dr}} \\
	=& \nu_1(\det(k))\Gamma_{\R}(2s)\sum_{(m,n) \in \Z^2\setminus{\set{(0,0)}}}{\frac{c(k\cdot\Phi_f,\nu_1\nu_2^{-1},(m,n)) y^s}{|mz+n|^{2s}}}.
\end{align*}
The formula (\ref{eqn:adelicclassical}) follows.\end{proof}

\subsection{Classical Kronecker limit formula}

We require the Kronecker limit formula for $E_{\overline{w},N}(z,s)$ at $s=0$.  We prove it by using Poisson summation to relate the Taylor expansion of $E_{\overline{w},N}(z,s)$ at $s=1$ to the Taylor expansion at $s=1$ of the Eisenstein series defined by
\[\widehat{E}_{\overline{w},N}(z,s) = \Gamma_{\R}(2s)\sum_{(m,n) \in \Z^2\setminus\{(0,0)\}}\frac{ y^{s} \exp\(\frac{2\pi i}{N} \langle (m,n),\overline{w}\rangle_{\mathbf{Z}/N\mathbf{Z}}\)}{|mz+n|^{2s}}.\]
We define $q:V_2(\mathbf{R}) \rightarrow \mathbf{R}_{\ge 0}$ by $q((x,y)) = x^2+y^2$.  For $g \in \SL_2(\R)^{+}$, $t \in \R^{\times}_{>0}$, and $w_0 = (u_0,v_0) = \overline{w}/N \in (\frac{1}{N}\mathbf{Z} / \mathbf{Z})^2$, define
\begin{align*}\Theta_{w_0}(t,g) &= \sum_{\substack{\xi \in V_2(\Z)\\ \xi+w_0 \ne (0,0)}}\exp\(-\pi tq((\xi+w_0) g)\) = \sum_{\substack{(m,n) \in V_2(\Z)\\(m,n)+w_0\ne (0,0)}}\exp\(-\pi t |(m+u_0)z+(n+v_0)|^2/y\)\end{align*}
and
\begin{align*}\widehat{\Theta}_{w_0}(t,g) &= \sum_{\xi \in V_2(\Z)}\exp\(2\pi i \langle \xi,w_0\rangle_\mathbf{R}\) \exp\(-\pi tq(\xi g)\) \\ &= \sum_{(m,n) \in V_2(\Z)}\exp\(2 \pi i \langle (m,n),(u_0,v_0)\rangle_\mathbf{R}\) \exp\(-\pi t |mz+n|^2/y\).\end{align*}
\begin{lemma}\label{wcharFE} We have the following identities.
\begin{enumerate}
\item For $\mathrm{Re}(s) \gg 0$, we have
\[N^{2s}E_{\overline{w},N}(z,s) = \int_{0}^{\infty}{\Theta_{w_0}(t,g) t^{s} \,\frac{dt}{t}}\quad\textrm{ and }\quad N^{2s}\widehat{E}_{\overline{w},N}(z,s) = \int_{0}^{\infty}{\left(\widehat{\Theta}_{w_0}(t,g)-1\right) t^{s} \,\frac{dt}{t}}.\]
\item For all $s$, we have
\[N^{2s}E_{\overline{w},N}(z,s) = -\frac{1}{1-s} + \int_{1}^{\infty}{ \Theta_{w_0}(t,g) t^{s}\,\frac{dt}{t}} + \int_{1}^{\infty}{ \left(\widehat{\Theta}_{w_0}(t,g)-1\right)t^{1-s}\,\frac{dt}{t}}\]
and
\[N^{2s}\widehat{E}_{\overline{w},N}(z,s) = -\frac{1}{s} + \int_{1}^{\infty}{ \Theta_{w_0}(t,g) t^{1-s}\,\frac{dt}{t}} + \int_{1}^{\infty}{ \left(\widehat{\Theta}_{w_0}(t,g)-1\right)t^{s}\,\frac{dt}{t}}.\]
\end{enumerate}
In particular, $E_{\overline{w},N}(z,s) = \widehat{E}_{\overline{w},N}(z,1-s)$.\end{lemma}
\begin{proof} The first part of the lemma follows immediately from the definitions.  We will deduce the second part from Poisson summation.

For a Schwartz function $f$ on $V_2(\R)$, define $f_{w_0}(v) = f(v+w_0)$ and $f_{(t,g)}(v) = f(t^{1/2}vg)$.  Then $\widehat{f_{w_0}}(v) = e^{2\pi i \langle w_0,v\rangle}\widehat{f}(v)$ and $\widehat{f_{(t,g)}}(v) = t^{-1}\widehat{f}(t^{-1/2}v g)$.  Thus Poisson summation $\sum_{\xi \in \Z^2}{f(\xi)} = \sum_{\xi \in \Z^2}{\widehat{f}(\xi)}$ gives
\[\sum_{\xi \in V_2(\mathbf{Z})}{f(t^{1/2}(\xi+w_0)g)} = t^{-1}\sum_{\xi \in V_2(\mathbf{Z})}{e^{2\pi i \langle w_0,\xi\rangle}\widehat{f}(t^{1/2}vg)}.\]
Taking $f(v) = \exp\(-\pi q(v)\)$ so that $\widehat{f} = f$, we obtain
\[\Theta_{w_0}(t,g) = t^{-1}\widehat{\Theta}_{w_0}(t^{-1},g).\]
Assume $\mathrm{Re}(s) \gg 0$.  Then
\begin{align*} N^{2s}E_{\overline{w},N}(z,s) &= \int_{0}^{1}{\Theta_{w_0}(t,g)t^{s}\,\frac{dt}{t}} + \int_{1}^{\infty}{\Theta_{w_0}(t,g) t^{s}\,\frac{dt}{t}} \\ &= \int_{0}^{1}{t^{-1}\widehat{\Theta}_{w_0}(t^{-1},g)t^{s}\,\frac{dt}{t}} + \int_{1}^{\infty}{\Theta_{w_0}(t,g) t^{s}\,\frac{dt}{t}} \\ &= \int_{0}^{1}{t^{s-1}\,\frac{dt}{t}}+ \int_{0}^{1}{t^{-1}\left(\widehat{\Theta}_{w_0}(t^{-1},g)-1\right)t^{s}\,\frac{dt}{t}} + \int_{1}^{\infty}{\Theta_{w_0}(t,g) t^{s}\,\frac{dt}{t}} \\ &=-\frac{1}{1-s}+ \int_{1}^{\infty}{\left(\widehat{\Theta}_{w_0}(t,g)-1\right)t^{1-s}\,\frac{dt}{t}} + \int_{1}^{\infty}{\Theta_{w_0}(t,g) t^{s}\,\frac{dt}{t}}.\end{align*}
This give the first part of (2), and the second part is obtained similarly.  The lemma follows. \end{proof}

\subsection{Formula for level $\Gamma(N)$}

Lang \cite[\S 20.5]{lang} proves a formula relating the Eisenstein series $\widehat{E}_{\overline{w},N}(z,s)$ to an explicit function on the upper half-plane.  Let $(\alpha,\beta) \in \Q^2 \setminus \Z^2$.  Let $g_{(\alpha,\beta)}$ denote the function on the upper half-plane defined by the infinite product
\[g_{(\alpha,\beta)}(z) = -q^{B_2(\alpha)/2}\exp(\pi i (\alpha-1)\beta)\prod_{n=1}^\infty{(1-q^n q_z)(1-q^n/q_z)}\]
where $q = e^{2\pi i z}$, $q_z = e^{2\pi i(\alpha - \beta z)}$, and $B_2(X) = X^2-X+ 1/6$.  In {\it loc.\ cit.}, the Taylor expansion
\begin{equation}\widehat{E}_{\overline{w},N}(z,s) = \log|g_{w_0}(z)| + O(s-1)\end{equation}
is proved whenever $\overline{w} \neq (0,0)\in V_2(\Z/N\Z)$.  Lemma \ref{wcharFE} then gives
\begin{equation} \label{eqn:classicalklf} E_{\overline{w},N}(z,s) = \log|g_{w_0}(z)| + O(s).\end{equation}
Combining this with Lemma \ref{lem:adelicclassical}, we may deduce the following.
\begin{proposition}\label{prop:EisPhi} Suppose $\mathrm{Supp}(\Phi_f)\subseteq V_2(\widehat{\Z})$, and let $g = g_\Q g_{\infty}k$ with $g_{\Q} \in \GL_2(\Q)$, $g_\infty \in \GL_2(\R)^{+}$, and $k \in \GL_2(\widehat{\Z})$.  Set $z = g_{\infty} \cdot i$ in the upper half plane, and let $N \in \mathbf{Z}_{\ge 1}$ be such that $\Phi_f$ is stable by $N\widehat{\Z}^2$ and $\nu_1\nu_2^{-1}$ is trivial on $(1 + N\widehat{\Z})^\times \subseteq \GL_1(\A_f)$.  Finally, if $\nu_1\nu_2^{-1} =1$, assume that $\Phi(0) = 0$.  Then
\[E(g,\Phi,\nu,s) = \nu_1(\det(k))\sum_{\overline{w} \in V_2(\Z/N\Z) \setminus\{(0,0)\}}{c(k \cdot \Phi_f,\nu_1\nu_2^{-1},\overline{w}) \log|g_{w_0}(z)|} +O(s).\]
\end{proposition}
\begin{proof}  If $\nu_1\nu_2^{-1} \neq 1$, or if $\Phi(0) =  0$, then $c(k \cdot \Phi,\nu_1\nu_2^{-1},(0,0)) = 0$.  Thus this term disappears, and the proposition follows from (\ref{eqn:classicalklf}) and Lemma \ref{lem:adelicclassical}. \end{proof}

We have the following fact regarding the functions $g_{w_0}$.
\begin{prop}[{\cite[Theorem 19.2.2]{lang}}]
The function $g_{w_0}(z)^{12N}$ defines a nowhere vanishing holomorphic modular function of level $\Gamma(N)$ on the upper half-plane.
\end{prop}
In particular, $g_{w_0}(z)^{12N}$ defines a function on the classical modular curve $Y(N)$.  Let $K(N)\subseteq \GL_2(\mathbf{A}_f)$ denote the open compact subgroup of matrices congrent to $1$ modulo $N$.  We write $\mathrm{Sh}(\GL_{2/\mathbf{Q}},K(N))$ for the Shimura variety of $\GL_{2/\mathbf{Q}}$ of level $K(N)$, which is a disjoint union of connected components isomorphic to $Y(N)$.  The variety $\mathrm{Sh}(\GL_{2/\mathbf{Q}},K(N))$ is defined over $\mathbf{Q}$.

The connected components of $\mathrm{Sh}(\GL_{2/\mathbf{Q}},K(N))$ are indexed by the class of $\det g$ (or simply $\det k$ in the notation $g = g_\mathbf{Q}g_\infty k$) in $\mathbf{Q}^\times \backslash \mathbf{A}^\times / \mathbf{R}^+K(N) = \widehat{\mathbf{Z}}^\times/(1+N\widehat{Z})^\times$.  Let $k_d =\prod_{p|N}  \mm{1}{}{}{d}$ for $d\nmid N$ be a set of representatives for these components.  Write $\widetilde{z}=(z,k_d)$ for a point of $\mathrm{Sh}(\GL_{2/\mathbf{Q}},K(N))$ given by $z = g_\infty \cdot i$ on the component indexed by $k_d$.  It follows from the description of the Galois action in Lang \cite[\S 19.2]{lang} that the function $G_{w_0}(\widetilde{z}) = g_{w_0 k_d^{-1}}^{12N}(z)$ is a nowhere vanishing holomorphic function on $\mathrm{Sh}(\GL_{2/\mathbf{Q}},K(N))$ defined over $\mathbf{Q}$.

We may rewrite Proposition \ref{prop:EisPhi} as
\begin{equation}\label{eqn:detleft} \nu_1(\det(k))^{-1}E(g,\Phi,\nu,s) =\sum_{\overline{w} \in V_2(\Z/N\Z) \setminus\{(0,0)\}}{c(\Phi,\nu_1\nu_2^{-1},\overline{w}) \log|g_{w_0k^{-1}}(z)|} +O(s),\end{equation}
since $c(k\cdot \Phi_f,\eta,\ol{w}) = c(\Phi_f,\eta,\ol{w}k)$.  Increase $N$ if necessary so that $\nu_1$ is trivial on $(1+N \widehat{Z})^\times$.  (Recall that only $\nu_1\nu_2^{-1}$ was assumed trivial on this group before.)  Then $\nu_1(\det(k))^{-1}$ becomes a well-defined function of $\widetilde{z}=(z,k_d)$, namely $\nu_1(\det(k_d))^{-1}$.  So the left-hand side of (\ref{eqn:detleft}) can be written as a function of $\widetilde{z}$, namely
\begin{align*}
	\nu_1^{-1}(\det(\widetilde{z}))E(\widetilde{z},\Phi,\nu,s) &= \sum_{w \in V_2(\Z/N\Z) \setminus\{(0,0)\}} c(\Phi,\nu_1\nu_2^{-1},\overline{w})\log|g_{w_0 \cdot k_d^{-1}}(z)| + O(s)\\
	&= \frac{1}{12N}\sum_{w \in V_2(\Z/N\Z) \setminus\{(0,0)\}} c(\Phi,\nu_1\nu_2^{-1},\overline{w})\log|G_{w_0}(\widetilde{z})| + O(s).
\end{align*}

We summarize this discussion in the following corollary.
\begin{corollary}  \label{coro:allphi} Denote by $\Q(\Phi_f,\nu_1\nu_2^{-1})$ the extension field of $\Q$ generated by the values of $\Phi_f$ and $\nu_1\nu_2^{-1}$. Given $\Phi$ such that $\Phi(0)=0$ if $\nu_1\nu_2^{-1} \ne 1$, there is a unit $u(\Phi) \in \mathfrak{O}(\mathrm{Sh}(\GL_2,K)_{\Q})^\times \otimes \Q(\Phi_f,\nu_1\nu_2^{-1})$ (i.e.\ defined over $\Q$ with coefficients in $\Q(\Phi_f,\nu_1\nu_2^{-1})$) for which
\[\nu_1^{-1}(\det(g_f))E(g,\Phi,\nu,s) = \log|u(\Phi)| + O(s).\]
There is also a unit $u(\Phi,\nu_1) \in \mathfrak{O}(\mathrm{Sh}(\GL_2,K)_{\Q(\nu_1)})^\times \otimes \Q(\Phi_f,\nu_1,\nu_2)$ such that
\[E(g,\Phi,\nu,s) = \log|u(\Phi,\nu_1)| + O(s).\]
\end{corollary}

\begin{proof}
For the first part, take
\[u(\Phi)(\widetilde{z})=\frac{1}{12N}\sum_{w \in V_2(\Z/N\Z) \setminus\{(0,0)\}} c(\Phi,\nu_1\nu_2^{-1},\overline{w})\log|G_{w_0}(\widetilde{z})|.\]
For the second, take
\[u(\Phi,\nu_1)(\widetilde{z}) = \frac{1}{12N}\nu_1(\det(k_d))\sum_{w \in V_2(\Z/N\Z) \setminus\{(0,0)\}} c(\Phi,\nu_1\nu_2^{-1},\overline{w})\log|G_{w_0}(\widetilde{z})|.\]
It is easy to see that this is stable by $\gal(\Q(\nu_1)/\mathbf{Q})$.
\end{proof}

\subsection{Formula for level $\Gamma_0(N)$}

We now give an explicit Kronecker limit formula for level $\Gamma_0(N)$.  Assume $\nu_1 = \nu_2 = 1$, which is forced on us by our choice of level.  Let $\Delta$ be the usual function of Ramanujan. Following Baba and Sreekantan \cite{bs}, let
\[\Delta_N(z) = \prod_{d \mid N}{\Delta\left(\frac{Nz}{d}\right)^{\mu(d)}},\]
where $\mu$ is the M\"obius function.  Write $\Phi_N^{m,n}$ for the characteristic function of $(m,n)+N\widehat{Z} \subseteq V_2(\mathbf{A}_f)$.  Define $\Phi_N$ to be the Schwartz-Bruhat function $\Phi_{N} = \sum_{m \in (\Z/N\mathbf{Z})^{\times}}{\Phi_N^{0,m}}$ if $N >1$ and set $\Phi_1 = \Phi_1^{(0,0)}$ if $N=1$.  We have the following formula in this special case.
\begin{lemma} Write $E(g,s)$ for  $E(g,\Phi_1,1,s)$ and $E(g,\Phi_N,s)$ for $E(g,\Phi_N,1,s)$.  Let $g = g_\infty$ with $g \cdot i = z = x + iy$. Then
\begin{equation}\label{eqn:level1toN} E(g,\Phi_N,s) = N^{-s} \sum_{d \mid N}{\mu(d) d^{-s} E\left(\frac{Nz}{d},s\right)}.\end{equation}
There are $3$ cases.
\begin{enumerate}
\item We have $E(g,\Phi_1,s) = E(g,s)= -\frac{1}{s}+ (\gamma - \log(4 \pi)) - \frac{1}{6}\log(y^6|\Delta(z)|) + O(s)$.
\item If $N = p^r$ for $r \geq 1$ is a prime power, then $E(g,\Phi_N,s) = -2\log(p) -\frac{1}{6}\log|\Delta_N(z)| + O(s)$.
\item If $N$ is divisible by $2$ distinct primes, $E(g,\Phi_N,s) = -\frac{1}{6}\log|\Delta_N(z)| + O(s)$. \end{enumerate} \end{lemma}
\begin{proof}   Writing $\sum'$ to indicate the removal of $(0,0)$, we have
\begin{align*}
  (Nd)^{-s} E(Nz/d,s) &= \Gamma_\mathbf{R}(2s)(Nd)^{-s} \sideset{}{^{\prime}}\sum_{m,n}{(Ny/d)^{s}|m(Nz/d)+n|^{-2s}} \\&= \Gamma_\mathbf{R}(2s)\sideset{}{^{\prime}}\sum_{m,n}{y^s|(mN)z+dn|^{-2s}}.
\end{align*}
Thus 
\begin{equation}\label{muNeqn}N^{-s} \sum_{d \mid N}{\mu(d) d^{-s} E\left(\frac{Nz}{d},s\right)} = \Gamma_{\R}(2s)\sum_{m\in \Z, n \in \Z, (n,N)=1}{\frac{y^s}{|(mN)z+n|^{2s}}}.\end{equation}
It is also clear that 
\[c(\Phi_N,1,(m,n)) = \begin{cases} 1&\mbox{if } m \equiv 0\ \textrm{mod }N, n \in (\Z/N\mathbf{Z})^{\times} \\ 0& \text{otherwise}.\end{cases}\]
Thus the right-hand side of (\ref{muNeqn}) is $E(g,\Phi_N,s)$ by Proposition \ref{prop:EisPhi}, which yields (\ref{eqn:level1toN}).

Claim (1) is just the Kronecker limit formula given in \cite[\S 20.4]{lang} combined with the functional equation $E(g,s) = E(g,1-s)$ of the Eisenstein series.  

Claims (2) and (3) follow from claim (1) and the formula for $E(g,\Phi_N,s)$.  Both of these calculations follow easily from the identity
\[\sum_{d|N}{\mu(d)\log(d)} = \begin{cases} -\log(p) & \textrm{if }N =p^r\\ 0 & \textrm{if }N\textrm{ is divisible by 2 primes}.\end{cases}\]
\end{proof}

\section{Prime number theorem} \label{sec:pnt}

In this section, we appeal to base change results as well as results concerning automorphic representations on $\GL_n$ in order to deduce consequences for the nonvanishing of the value of the standard $L$-function of certain generic cuspidal automorphic representations $\pi$ on $\GU(2,1)$ at $s=1$.  Nothing in this section is original; we merely combine known results.

\subsection{Twisted base change}

By the twisted base change of an automorphic representation $\pi$ on $G = \GU(2,1)$, we mean an automorphic representation $\tau$ on $\mathrm{Res}^E_\mathbf{Q} (\mathbf{G}_m \times \GL_3)$ that satisfies certain local compatibility results with $\pi$ that characterize it uniquely.  Twisted base change results for unitary groups of rank three were originally studied by Rogawski and others in \cite{rogawski,rogawski2,pms}, where essentially complete results were obtained.

Although Rogawski \cite[\S 13.2]{rogawski} gives a complete description of the local base change even at inert and ramified places, we state a limited version of the base change theorem here that will suffice for the statement of our fine regulator formula below, and for which the local compatibilities have simple descriptions.  Namely, in the case where $v$ is inert or ramified, compatibility may be defined using the Satake isomorphism of Cartier \cite{cartier} (or Haines-Rostami \cite{hro}) with respect to $K = G(\mathbf{Z}_p)$.  At split places, the definition of compatibility is even simpler: the corresponding Weil-Deligne representations are identified by the diagonal map ${}^LG \rightarrow {}^LG'$ from Definition \ref{defin:stdl}.

Note that Rogawski's results are for the ordinary unitary group.  However, once the global base change has been constructed in the similitude setting -- which was carried out by Morel \cite[Corollary 8.5.3]{morel} and Shin \cite{shin} -- the rest of the statement below reduces easily to the unitary case due to the known compatibility of central characters and the fact that $\GU(\mathbf{Q}_p) = \U(\mathbf{Q}_p)Z_G(\mathbf{Q}_p)$.

\begin{thm}[{\cite[\S 13]{rogawski}}] \label{thm:morel}

Let $\pi$ be a cuspidal automorphic representation of $G=\GU(J)$ over $\mathbf{Q}$ such that $\pi_\infty$ is discrete series.  Then there exists an automorphic representation $\tau$ on $\mathrm{Res}^E_\mathbf{Q} (\mathbf{G}_m \times \GL_3)$ that is compatible with $\pi$ at every place $v$ of $\mathbf{Q}$ such that
\begin{itemize}
	\item $v$ splits in $E$ or
	\item $v$ is inert in $E$ and $\pi_v$ is unramified.
	\item $v$ is ramified in $E$ and $\pi_v$ has a vector fixed by $G(\mathbf{Z}_p)$.
\end{itemize}
\end{thm}

\begin{remark}
Since $\tau$ is uniquely determined by $\pi$, one can use $\tau$ to make a well-defined $L$-factor at \emph{all} places using the Langlands group representation of Definition \ref{defin:stdl} and the known local $L$-parameters of representations of $\mathrm{Res}^E_\mathbf{Q} (\mathbf{G}_m \times \GL_3)$.
\end{remark}

\subsection{Prime number theorem}

We have the following generalization of the prime number theorem due to Jacquet and Shalika \cite{js}.

\begin{thm}[\cite{js}] \label{thm:js}
Let $L(\tau,s)$ be the standard automorphic $L$-function attached to a cuspidal automorphic representation on $\GL_n$ over a number field.  Assume that the central character of $\tau$ is unitary. Then $L(\tau,s) \ne 0$ for any $s$ with $\mathrm{Re}(s)=1$.
\end{thm}

\subsection{Generalized Ramanujan conjecture}

Since we need to use the prime number theorem on $\GL_n$ to obtain a consequence for the $L$-function of $\pi$ on $\GU(2,1)$, we will need to rule out certain poles of local $L$-factors.  The following bound of Jacquet--Shalika \cite{js2} and Rudnick--Sarnak \cite{rs} will suffice for our purposes.  Note that the argument in Rudnick--Sarnak takes $F=\mathbf{Q}$ but applies without modification to the general case.  Also see \cite{bb} for an overview of bounds towards the generalized Ramanujan conjecture that are applicable to any number field $F$.
\begin{thm}[{\cite[\S 2.5, Corollary]{js2}, \cite[p.\ 317]{rs}}]\label{thm:rambound}
Let $\tau$ a cuspidal automorphic representation on $\GL_{n/F}$ with unitary central character, where $F$ is a number field.  Let $v$ be a place of $F$ and write $\tau_v$ for the local component of $\tau$ at $v$.  The standard local $L$-function $L_v(\tau_v,\mathrm{Std},s)$ has no poles for $\mathrm{Re}(s)\ge \frac{1}{2}$.
\end{thm}

\begin{remark} \label{remark:temperedness} Although we do not need a stronger result than Theorem \ref{thm:rambound} for the proof of Theorem \ref{thm:pnt}, we remark that the results of \cite{blggt,blggt2,caraiani,caraiani2} show that a conjugate self-dual regular cuspidal automorphic representation on $\GL_n$ over a CM field is tempered at every finite place.  In particular, if the restriction of $\tau$ in Theorem \ref{thm:morel} to $\mathrm{Res}_\mathbf{Q}^E \GL_3$ is cuspidal, it is tempered at every finite place.
\end{remark}

\subsection{Nonvanishing}

We may prove the following theorem by combining the aforementioned results.

\begin{thm} \label{thm:pnt}
Let $\pi$ be a generic cuspidal automorphic representation on $\GU(2,1)$ with unitary central character.  Assume that the transfer of $\pi$ to $\mathrm{Res}_\mathbf{Q}^E (\GL_3 \times \mathbf{G}_m)$ is cuspidal.  For any finite set $S$ of places of $\mathbf{Q}$, the value of $L^S(\pi,\mathrm{Std},1)$ is nonzero.
\end{thm}

\begin{proof}

Write $(\chi,\tau_0)$ for the transfer of $\pi$ to $\mathbf{G}_m \times \GL_3$.  Skinner \cite[Theorem A]{skinner} computes the character $\chi=\chi_\pi^c$ and central character $\chi_\tau = \chi_\pi/\chi_\pi^c$ in terms of the central character $\chi_\pi$ of $\pi$, where $c$ denotes complex conjugation.  It follows that $\chi$ and $\chi_\tau$ are unitary.  We conclude that the twist $\tau=\tau_0 \otimes \chi$ is again a cuspidal automorphic representation of $\GL_n$ with unitary central character.  So $L(\tau,s) \ne 1$ for $\mathrm{Re}(s) =1$ by Theorem \ref{thm:js}.

Write $S' \supseteq S$ for the possibly larger finite set of places containing $S$ as well as all primes where either $\pi$ or $E$ is ramified.  We write $L(\tau,s) = \prod_{v \in S'} L_v(\tau,s) \cdot L^{S'}(\tau,s)$.  We have $L^{S'}(\tau,s) = L^{S'}(\pi,s)$ by Theorem \ref{thm:morel} and $L^S(\pi,s) = L^{S'}(\pi,s)\cdot \prod_{v \in S \setminus S'} L_v(\pi,s)$.  Combining these equations,
\[L^S(\pi,s) = L(\tau,s)\cdot \prod_{v \in S'} L_v(\tau,s)^{-1} \cdot \prod_{v \in S \setminus S'} L_v(\pi,s).\]
By definition, the non-Archimedean local $L$-factors of $\pi$ are inverses of polynomials in $p^{-s}$.  So $\prod_{v \in S \setminus S'} L_v(\pi,s)$ is nowhere vanishing.  On the other hand, $\prod_{v \in S'} L_v(\tau,s)^{-1}$ has zeroes constrained to the plane $\mathrm{Re}(s) < \frac{1}{2}$ by Theorem \ref{thm:rambound}.  (In fact, one can deduce from Remark \ref{remark:temperedness} and the Bernstein--Zelevinsky classification that the zeroes must satisfy $\mathrm{Re}(s) \in \set{0,-\frac{1}{2},-1}$.)  The desired result for $L^S(\pi,s)$ follows. \end{proof}

\section{Local theory} \label{sec:localtheory}

We prove algebraicity and non-vanishing results for local integrals.  The latter may be thought of as a local complement to the global non-vanishing proved in Theorem \ref{thm:pnt}.

\subsection{Algebraicity} \label{subsec:locrationality}

We maintain the notations of Section \ref{sec:diff}.  In this section, we will work locally at a finite place $p$ of $\mathbf{Q}$.  For ease of notation, we usually suppress subscripts of $p$ and omit the local points; for instance, we write $G$ for $G(\mathbf{Q}_p)$.  Recall from (\ref{eqn:unfold2}) that the local integral of interest is given by
\[I(W,\Phi,\nu,s) = \int_{U_2\backslash H}{\nu_1(\det(g_1))|\det(g_1)|^{s}\Phi((0,1)g_1)W(g)\,dg},\]
where $g_1$ denotes the projection $H \rightarrow \GL_2$.  As mentioned in Proposition \ref{prop:unfolding}, $I(\Phi,W,\nu,s)$ is a rational function in $p^{-s}$ with coefficients in $\mathbf{C}$.  We extend this to prove the following algebraicity statement.

\begin{prop}  \label{prop:locrationality} Suppose the Schwartz-Bruhat function $\Phi$ on $\Q_p^2$ and the Whittaker function $W$ are both valued in $\Qbar$.  Then the local integral $I(\Phi,W,\nu,s)$ has a meromorphic continuation to a rational function $\frac{P(X)}{Q(X)}$ of $X=p^{-s}$, where $P$ and $Q$ have algebraic coefficients.  In particular, the meromorphic continuation of $I(\Phi,W,\nu,s)$ to any $s \in \mathbf{Q}$ is valued in $\Qbar$ if it is finite. \end{prop}
\begin{proof}  If $E_p$ is inert or ramified, we write the Whittaker function $W(\diag(\tau,1,\ol{\tau}^{-1}))$ as a finite sum $\sum_{\alpha \in A} \phi_\alpha(\tau) \alpha(\tau)$ of products of Schwartz-Bruhat functions $\phi_\alpha \in \mathcal{S}(E_p)$ multiplied by finite functions $\alpha$ on $E_p^\times$.  (See \cite[Proposition 3.2]{baruch}.)  Letting $\mathrm{Re}(s) \gg 0$, writing $X = p^{-s}$, and applying the Iwasawa decomposition, we obtain a Laurent series expansion for $I(W,\Phi,\nu,s)$ in $X$ with coefficients
\[a_n(W,\Phi,\nu) = \int_{T}{\delta_{B_H}^{-1}(t)\charf_{|\det t_1|=p^{-n}}(t)\nu_1(\det(t_1))\Phi((0,1)t_1)W(t)\,dt},\]
on $X^n$, where $t_1$ denotes the projection of $t$ to $\GL_2(\mathbf{Q}_p)$.  Elements $t$ may be written in the form $\diag(\tau_1\tau_2,\tau_2,\ol{\tau}_1^{-1}\tau_2)$, so the condition on $\det t_1$ forces $\tau_2$ to lie in a set of compact support.  The condition $\Phi((0,1)t_1)$ gives an upper bound on $\mathrm{ord}_p(\tau_1)$, while the central character combined with the fact that $W(\diag(\tau,1,\ol{\tau}^{-1}))=\sum_{\alpha \in A} \phi_\alpha(\tau) \alpha(\tau)$ gives a lower bound.  In particular, the domain can be made compact.  Since the integrand is locally constant and algebraically valued, $a_n(W,\Phi,\nu)$ is algebraic.

In the split case, we first recall from Section \ref{subsubsec:splitcase} that $T$ may be represented uniquely by elements of the form $(\diag(\tau_1,\tau_2),(\tau_3,\tau_1\tau_2\tau_3^{-1})) \in \GL_2(\mathbf{Q}_p) \times \GU(1)(\mathbf{Q}_p)$, where we have written an element of $\GU(1)(\mathbf{Q}_p)$ as a pair of elements $\mathbf{Q}_p^\times$ via the two projections.  Since the second projection is determined by the rest of the data, we may simplify this further to $(\diag(\tau_1,\tau_2),\tau_3) \in \GL_2(\mathbf{Q}_p) \times \mathbf{G}_m(\mathbf{Q}_p) \cong H(\mathbf{Q}_p)$ using the first projection.  Via the identification of $G(\mathbf{Q}_p)$ with $\GL_3 \times \mathbf{G}_m$, these elements correspond to $(\diag(\tau_1,\tau_3,\tau_2),\tau_1\tau_2)$.  The condition $\charf_{|\det t_1|=p^{-n}}(t)$ appearing in the integral computing $a_n(W,\Phi,\nu)$ forces $\ord_p(\tau_1\tau_2)=n$, while $\ord_p(\tau_2)$ is bounded below by a quantity determined by $\Phi$.  By \cite[Proposition 2.1]{jpss}, the restricted Whittaker function $W(\diag(\sigma_1\sigma_2,\sigma_2,1))$ can again be written as a finite sum of finite functions multiplied by Schwartz-Bruhat functions.  Using the central character, this gives lower bounds on $\ord_p(\tau_1/\tau_3)$ and $\ord_p(\tau_3/\tau_2)$.  Combined with $\ord_p(\tau_1\tau_2)=n$, we get an upper bound on $\ord_p(\tau_2)$ as well as upper and lower bounds on $\ord_p(\tau_3)$.  Now the algebraicity of $a_n(W,\Phi,\nu)$ follows as before.

By the aforementioned result of Baruch, $I(\Phi,W,\nu,s)=\sum_n a_n(W,\Phi,\nu)X^n = \frac{P(X)}{Q(X)}$ for polynomials $P,Q \in \mathbf{C}[X]$.  We rearrange this to $Q(X)\sum_n a_n(W,\Phi,\nu)X^n = P(X)$.  Let the degrees of $P$ and $Q$ be $d$ and $e$, respectively.  Then for all $k>d$, the coefficient on $X^k$ in $Q(X)\sum_n a_n(W,\Phi,\nu)X^n$ is zero.  Now consider a polynomial $Q'(X) \in \ol{\mathbf{Q}}[X]$ of degree $e$ and regard the coefficients $\set{a_i}_{i=0}^e$ as variables.  For each $k > d$, the condition that the coefficient of $X^k$ in $Q(X)\sum_n a_n(W,\Phi,\nu)X^n$ is zero imposes a linear constraint on $\set{a_i}_{i=0}^e$.  There are infinitely many such constraints, but finitely many must cut out the full set of solutions.  The resulting system of linear equations has coefficients in $\Qbar$ and a non-zero solution over $\mathbf{C}$, so it must also have a solution $Q'(X) \in \Qbar[X]$.  Then $Q'(X)\sum_n a_n(W,\Phi,\nu)X^n = P'(X)$, where $P'(X) \in \Qbar[X]$ as well, and $\frac{P'(X)}{Q'(X)}$ is the needed rational function.
\end{proof}

We now address the hypothesis in the preceding proposition that $W$ is $\ol{\mathbf{Q}}$-valued.  In the following definition and proposition, we work more generally so as to accomodate global constructions later.
\begin{defin} \label{defin:defined} Suppose that $\pi$ is a representation of a group $G$ on a $\mathbf{C}$-vector space $V$.  We say that $\pi$ is defined over $L \subseteq \mathbf{C}$ if there exists a representation $\pi_0$ of $G$ on a vector space $V_0$ over $L$ such that $V_0 \otimes_{L} \mathbf{C}$ is isomorphic to $\pi$.\end{defin}
Then we have the following result.
\begin{prop}[{\cite[Proposition 3.2]{grobseb}}]  \label{prop:whitrational}
Suppose that a representation $\pi$ of a group $G$ is defined over $L$, and let $V$ and $V_0$ be as in Definition \ref{defin:defined}.  Assume that $\Lambda: V \rightarrow \mathbf{C}$ is a nonzero functional on $V$ that has a left transformation property with respect to a subgroup $H$ and character $\chi:H \rightarrow L^\times$ that characterizes it uniquely up to scaling by an element of $\mathbf{C}^\times$.  (For instance, $\Lambda$ can be a Whittaker functional.)  Then there exists a functional $\Lambda_0: V \rightarrow \mathbf{C}$ with $\mathrm{Im}(\Lambda_0|_{V_p})\subseteq L$.
\end{prop}

\subsection{Non-vanishing} \label{subsec:nonvanishing}

Baruch \cite[pg.\ 331]{baruch} proves a precise statement regarding the ideal \emph{spanned} by the local integral $I(\Phi,W,\nu,s)$ as the data varies, but we will need to construct a single $W$ and $\Phi$ so that our local integral is constant and non-zero.  For the remainder of this section, we assume that $\Lambda$ is a fixed Whittaker functional on the space $V$ of $\pi_p$ and, for $\varphi \in V$, we define $W_\varphi: G \rightarrow \mathbf{C}$ by $W_\varphi(g) = \Lambda(g\varphi)$.  (Note that this is a slightly different convention than Definition \ref{defin:whit}, which we use since we will vary our choice of $\varphi$.)
\begin{lemma}\label{trivLoc} Suppose that $\varphi_0 \in V$ satisfies $\Lambda(\varphi_0)\ne 0$.  There exists $\eta \in C_c^\infty(G(\mathbf{Q}_p),\mathbf{Q})$ and a $\Q$-valued Schwartz-Bruhat function $\Phi$ so that $I(\Phi,W_{\pi(\eta)\varphi_0},\nu,s)$ is a nonzero constant independent of $s$. \end{lemma}
\begin{proof} Suppose that $\varphi_0 \in V$ has $\Lambda(\varphi_0) \ne 0$.  Let $n$ be large enough so that $\varphi_0$ is fixed by the congruence subgroup $K_n$, where we write $K_n = \mathrm{Ker}(G(\mathbf{Z}_p) \rightarrow G(\mathbf{Z}_p/p^n\mathbf{Z}_p))$.  Let $\eta_0\in C_{c}^\infty(E_p,\mathbf{Q})$ and set $\varphi = \int_{E_p}{\eta_0(y) \pi(M_{0,y}) \cdot \varphi_0 \,dy}$ where
\[M_{x,y} = \left(\begin{array}{ccc}1&\ol{y}&x+y\ol{y}\delta/2 \\ & 1& y \\ &&1\end{array}\right).\]
Then $\varphi$ is fixed by $K_N$ for $N$ sufficiently large.  It also follows from its definition that $\varphi = \pi(\eta)\varphi_0$ for a $\mathbf{Q}$-valued Hecke operator $\eta$.

Now define a Schwartz-Bruhat function on $\mathbf{Q}_p^2$ by $\Phi_0 = \charf_{(0,1) + p^{N}\Z_p^2}$ and for $g\in H$, write $g_1$ for the projection to $\GL_2(\mathbf{Q}_p)$.  Observe that if $\Phi_0((0,1)g_1) \neq 0$, then $g_1 \in B_2K'_N$, where $K'_N$ is the projection to $\GL_2$ of $K_N \cap H$ and $B$ is the upper-triangular Borel.  Hence
\begin{align*}I(\Phi_0,W_\varphi,\nu,s) &= \int_{U\backslash (B_2K'_N \boxtimes E_p^\times)}{\nu_1(\det(g_1))|\det(g_1)|^{s}\Phi_{0}((0,1)g_1)\Lambda(\pi(g) \varphi)\,dg} \\ &= \int_{T}{\delta_{B_H}^{-1}(t)\nu_1(\det(t_1))|\det(t_1)|^{s}\Phi_{0}((0,1)t_1)\Lambda(\pi(t) \eta_0* \varphi_0)\,dt} \\&= \int_{T}{\delta_{B_H}^{-1}(t)\nu_1(\det(t_1))|\det(t_1)|^{s}\Phi_0((0,1)t_1)\widehat{\eta}_0(\tau_2/\tau_3)\Lambda(\pi(t)\varphi_0)\,dt},\end{align*}
where we have written $t = \diag(\tau_1,\tau_2,\tau_3)$ and $\widehat{\eta}_0$ is the Fourier transform of $\eta$ (with respect to the fixed additive character $\psi$).  We may choose $\widehat{\eta}_0$ such that its support is concentrated near $1$ and $\Lambda(\varphi) \ne 0$.  With such a choice, for nonvanishing of the integrand, $\tau_3$ is forced to be near $1$ by our choice of $\Phi_0$, and then $\tau_2/\tau_3$ and thus $\tau_2$ is forced to be near $1$ by the choice of $\widehat{\eta}_0$.  Finally, $\tau_1$ is forced to be near $1$ by the similitude condition.  Hence the integrand is only supported near $t=1$, and hence is a constant independent of $s$ that can be made nonzero.\end{proof}

We make a simple observation regarding the dependence of $I(\Phi,W,\nu,s)$ on $\Phi$.
\begin{lemma} \label{lem:phipoly}
Let $P(X^2) \in \mathbf{C}[X^2,X^{-2}]$ be a polynomial in $X^2 = p^{-2s}$ and $X^{-2}=p^{2s}$.  Given $\Phi$, there exists another Schwarz-Bruhat function $P\cdot \Phi$ so that
\[I(P\cdot \Phi,W,\nu,s) = P(X^2,X^{-2})I(\Phi,W,\nu,s).\]
This construction preserves algebraicity, in the sense that if $\Phi$ and $P$ are $\ol{\mathbf{Q}}$-valued, then so is $P \cdot \Phi$.
\end{lemma}
\begin{proof}
It suffices to do this for $X^2$ and $X^{-2}$.  One computes easily by using the change of variable $g \mapsto (p\mathbf{1}_3)g$ that we may define $(X^2 \cdot \Phi)(\xi) = \nu_1(p)^{-2}\omega_\pi(p)^{-1}\Phi(p^{-1}\xi)$ and $(X^{-2} \cdot \Phi)(\xi) = \nu_1(p)^2\omega_\pi(p)\Phi(p\xi)$.
\end{proof}

We now study a more delicate question, motivated by the usage of modular units above, that arises if $\nu_1^2 \cdot \omega_\pi|_{Z(\mathbf{A})}$ is trivial.  We would like the local integral $I(\Phi,W_\varphi,\nu,s)$ to be nonvanishing at $s=0$ or $s=1$ while $\Phi$ has the property that $\Phi(0)=0$ or $\widehat{\Phi}(0)=0$, respectively.

We assume in the remainder of this section that $\nu_1^2 \cdot \omega_\pi|_{Z(\mathbf{A})}$ is trivial; otherwise, we do not need to impose this condition on $\Phi$.

\begin{lemma}\label{lem:pole} Suppose there exist $\Phi$ and $\varphi$ such that the meromorphic continuation of $I(\Phi,W_\varphi,\nu,s)$ has a pole at $s=0$ or $s=1$.  Then there exists $\Phi'$ such that $I(\Phi',W_\varphi,\nu,s)$ is holomorphic and nonzero at $s=0$ or $s=1$, respectively, and such that $\Phi'(0)=0$ or $\widehat{\Phi}'(0)=0$, respectively.
\end{lemma}

\begin{proof}
Note that the pole has finite order, say $d$, by Baruch \cite[pg. 331]{baruch}.  If $s=0$, we apply Lemma \ref{lem:phipoly} to obtain $\Phi' = (1-X^2)^d\cdot \Phi$.  It follows immediately from the construction of $(1-X^2)^d\cdot \Phi$ that $\Phi(0)=0$.  If $s=1$, we use $\Phi' = (1-p^{-2}X^{-2})^d\cdot \Phi$, which is easily seen to have $\widehat{\Phi}'(0)=0$.
\end{proof}

\begin{definition} Suppose $\eta:H(\mathbf{Q}_p) \rightarrow \C^\times$ is a character.  A representation $(\pi_p,V)$ of $G(\mathbf{Q}_p)$ is said to be $(H,\eta)$-distinguished if there exists a nonzero linear functional $\ell: V \rightarrow \C$ satisfying $\ell(\pi(h) \cdot \varphi) = \eta(h)\ell(\varphi)$ for all $\varphi \in V$ and $h \in H$.\end{definition}

\begin{lemma}\label{distLem} Suppose that the representation $(\pi_p,V)$ of $G(\mathbf{Q}_p)$ has the property that the value $I(\Phi,W_{\varphi},\nu,s)|_{s=0}$ is finite and equal to $0$ for every $\varphi \in V$ and Schwartz-Bruhat function $\Phi$ satisfying $\Phi(0) =0$.   Then $\pi_p$ is $(H,\nu_1^{-1})$-distinguished.  Similarly, if $I(\Phi,W_{v},\nu,s)|_{s=1}$ is finite and equal to $0$ for all $\varphi \in V$ and $\Phi$ with $\widehat{\Phi}(0) = 0$, then $\pi_p$ is $(H,\nu_1^{-1})$-distinguished.  Here, $\nu_1$ defines a character of $H$ via composition with the similitude: $\nu_1(h)=\nu_1(\mu(h))$.\end{lemma}

\begin{proof}
First consider the case $s=0$.  By Lemma \ref{lem:pole}, $I(\Phi,W_\varphi,\nu,s)$ is holomorphic at $s=0$ for all $\Phi$ and $\varphi$.  Let $V$ denote the space of $\pi_p$.  Using Lemma \ref{trivLoc}, we may fix a Schwartz-Bruhat function $\Phi$ such that the linear functional $\ell=I(\Phi,W_\cdot,\nu,s)|_{s=0}: V \rightarrow \mathbf{C}$ is not identically $0$ on $V$.  (We write $I(\Phi,W_\cdot,\nu,s)|_{s=0}$ for the evaluation of the meromorphic continuation.)

We claim $\ell$ is $(H,\nu_1^{-1})$-invariant.  For $h \in H$ and $\mathrm{Re}(s) \gg 0$, one has
\begin{align*}
  I(\Phi,W_{h^{-1} \cdot \varphi},\nu,s) &= \int_{U_2\backslash H}{\nu_1(\det(g_1))|\det(g_1)|^{s}\Phi((0,1)g_1)\Lambda(\pi(g) \pi(h^{-1}) \varphi)\,dg} \\
  &= \nu_1(h)|\mu(h)|^{s}\int_{U_2\backslash H}{\nu_1(\det(g_1))|\det(g_1)|^{s}\Phi^{h}((0,1)g_1)\Lambda(\pi(g) \varphi)\,dg}\\
	&= \nu_1(h)|\mu(h)|^{s}I(\Phi^{h},W_\varphi,\nu,s).
\end{align*}
Here $\Lambda: V \rightarrow \C$ is the Whittaker functional fixed earlier and $\Phi^{h}(v) =\Phi(vh)$. Thus
\[\nu_1^{-1}(h)|\mu(h)|^{-s}I(\Phi,W_{h^{-1}\cdot v},\nu,s) - I(\Phi,W_v,\nu,s) = I(\Phi^{h}-\Phi,W_{v},s)\]
in the range of absolute convergence.  Note that $\Phi^{h}-\Phi$ is $0$ at $0$, so that by assumption, the meromorphic continuation of the right-hand side to $s=0$ is $0$.  Hence $\ell(h \cdot \varphi) = \nu_1^{-1}(h)\ell(\varphi)$ for all $h \in H$, proving the lemma in the $s=0$ case.

For the $s=1$ case, Lemma \ref{lem:pole} again implies that $I(\Phi,W_\varphi,\nu,s)$ is holomorphic at $s=1$ for all $\Phi$ and $\varphi$.  Using Lemma \ref{trivLoc}, choose $\Phi$ so that $\ell = I(\Phi,W_\varphi,\nu,s)|_{s=1}: V \rightarrow \mathbf{C}$ is not identically $0$.  Identically to the $s=0$ case, for $h \in H$, one has
\[\nu_1^{-1}(h)\ell( h^{-1} \cdot \varphi) - \ell(\varphi) = I(|\mu(h)|\Phi^h-\Phi,W_{\varphi},s)|_{s=1}.\]
If $\Phi' = |\mu(h)|\Phi^h - \Phi$, then $\widehat{\Phi}'(0) = 0$ since
\[\widehat{\Phi}(0) = \int_{\mathbf{Q}_p^2}{\Phi(v)\,dv} = \int_{\mathbf{Q}_p^2}|\det(h_1)|^{-1}\Phi(vh_1)\,dv,\]
where $h_1$ denotes the image of the projection $H \rightarrow \GL_2$ (so that $\det(h_1)=\mu(h)$).  Thus $\ell(h^{-1} \cdot \varphi) = \nu_1(h)\ell(\varphi)$, completing the proof in the $s=1$ case.
\end{proof}

The Moeglin-Vigneras-Waldspurger involution \cite{mvw} allows one to relate $(H,\nu_1^{-1})$-distinction of $\pi$ with that of its contragredient at any finite place $p$.  We will not use this result in the sequel, but it will clarify the arguments using the functional equation in Section \ref{sec:regulatorcomp}, as it explains why we need not mention the contragredient representation explicitly.
\begin{prop} \label{prop:contradist}
If $(x_1,x_2,x_3) \in E_p^3$ denotes a vector in the underlying Hermitian space $(W,J)$ of $G$, we define the endomorphism $\eta \in \GL_{\mathbf{Q}_p}(W)$ by $\eta((x_1,x_2,x_3)) = (\ol{x_1},\ol{x_2},-\ol{x_3})$.  Note that conjugation by $\eta$ preserves $G$ and $H$.  Then the following hold.
\begin{enumerate}
	\item For $(\pi,V)$ a representation of $G(\mathbf{Q}_p)$, write $(\pi^\eta,V)$ for the representation defined by $\pi^\eta(g) = \pi(\eta g \eta^{-1})$.  We have $\widetilde{\pi} \cong \pi^\eta \otimes (\omega_\pi^{-1}|_{Z(\mathbf{Q}_p)} \circ \mu)$.
  \item If $\pi$ is $(H,\nu_1^{-1})$-distinguished, then $\widetilde{\pi}$ is $(H,(\nu_1\omega_\pi|_{Z(\mathbf{Q}_p)})^{-1})$-distinguished, and conversely.
\end{enumerate}
Under the assumptions on $\pi_\infty$ in Section \ref{subsec:diff}, $\widetilde{\pi}_\infty$ is isomorphic to $\pi_\infty$.
\end{prop}

\begin{proof}
The transformation $\eta$ has the defining properties of the element $\delta$ described in \cite[p.\ 90]{mvw}, which implies the statement (1).  (It is shown that this is true for $\U(J)$ in {\it loc.\ cit.}, and one obtains the statement for $G$ by comparing central characters.)

For part (2), first identify $(\widetilde{\pi},\widetilde{V})$ with $(\pi^\eta \otimes (\omega_\pi^{-1}|_{Z(\mathbf{A})} \circ \mu),V)$ using part (1).  If $\Lambda: V \rightarrow \mathbf{C}$ transforms by $\nu_1^{-1}$ under $H$, both implications in part (2) follow from the computation
\[\Lambda((\pi^\eta \otimes (\omega_\pi^{-1}|_{Z(\mathbf{Q}_p)} \circ \mu))(h)v) = \omega_\pi^{-1}|_{Z(\mathbf{A})}(\mu(h))\Lambda(\pi(\eta h \eta^{-1})v) = (\nu_1\omega_\pi|_{Z(\mathbf{Q}_p)})^{-1}(\mu(h))\Lambda(v)\]
because $\eta h \eta^{-1}$ has the same similitude as $h$.

For the final claim, noting our assumption that $\pi_\infty$ has trivial central character, $\widetilde{\pi}_\infty$ is discrete series with the same minimal $K$-type and thus is the same.
\end{proof}

\subsection{Unramified GCD computation}

We verify that the GCD of local integrals defined by Baruch \cite{baruch} has the expected form when all the data is unramified.

\begin{remark} \label{remark:nonu1}
In order to more easily utilize the many papers in the literature on the local Gelbart--Piatetski-Shapiro \cite{gelbartPS} and Bump-Friedberg \cite{bf} integrals, we will sometimes make use of the reduction that the standard $L$-function of the representation $\pi(\nu_1 \circ \mu)$ is the same as the standard $\nu_1$-twisted $L$-function of $\pi$, as mentioned in Definition \ref{defin:stdl}.  This allows us to assume that $\nu_1=1$.
\end{remark}

When $\pi$ is unramified at $p$, we can show that the GCD of the local integral is equal to $L_{p}(\pi_p \times \nu_{1,p},\mathrm{Std},s)$, which was computed explicitly in Section \ref{subsec:unram}.  We will deduce this from special cases of far more general results of Miyauchi \cite{my4} if $p$ is inert and Matringe \cite{matringe2} if $p$ is split.  We remark that the construction of Matringe differs from ours, since our group $\GU(J)(\mathbf{Q}_p)$ is isomorphic to $\GL_3 \times \mathbf{G}_m$ rather than $\GL_3$.  If we were to insert the $\GL_1$-twist of $\pi_p|_{\GL_3}$ into Matringe's construction, this would produce the wrong $L$-function, which is why the proof below is slightly indirect.
\begin{thm} \label{thm:unramgcd}
If $\pi_p$ and $\nu_{1,p}$ are unramified and $p$ is inert or split in $E$, then the GCD of the integral $I_p(\Phi_p,W,\nu_p,s)$ as $\Phi_p$ and $W$ vary is given by $L_{p}^\mathrm{L}(\pi_p \times \nu_{1,p},\mathrm{Std},s)$.
\end{thm}
\begin{proof}
We make use of the reduction in Remark \ref{remark:nonu1} to reduce to the case when $\nu_1$ is trivial.  If $p$ is inert, this is now a case of \cite[Theorem 3.4]{my4}.

If $p$ is split, we may deduce this from the work of Matringe \cite{matringe2} as follows.  Suppose that $\pi_0$ is any irreducible admissible generic representation of $G_0(\mathbf{Q}_p)=\GL_3(\mathbf{Q}_p)$ and let $W_0(g_v)$ be any Whittaker function for $\pi_0$ on $G_0$.  We write $H_0(\mathbf{Q}_p) = \GL_2(\mathbf{Q}_p) \times \GL_1(\mathbf{Q}_p)$ with the embedding
\[\iota_0:(\(\begin{array}{cc} a & b \\c & d\end{array}\),\lambda) \mapsto \(\begin{array}{ccc} a &  & b \\ & \lambda & \\ c & & d\end{array}\).\]
Then Matringe \cite[Proposition 3.2]{matringe2} constructs an analogue of a GCD of the \emph{two}-variable integrals
\begin{equation} \label{eqn:mat} I_\mathrm{BF}(\Phi_p,W_0,s,t)= \int_{U_{2,0}(\Q_p)\backslash H_0(\mathbf{Q}_p)} \Phi_p((0,1)g_{1,p})|\det(g_{1,p})|_p^{s+t+\frac{1}{2}}|\det(g_{2,p})|^{-s+\frac{1}{2}+t}W_0(g_p)\,d{g_p}.\end{equation}
Here, $U_{2,0}$ is the group of upper-triangular unipotent matrices and $g_{1,p}$ and $g_{2,p}$ denote the projections to $\GL_2(\mathbf{Q}_p)$ and $\GL_1(\mathbf{Q}_p)$, respectively.  The GCD constructed by Matringe is a certain uniquely defined function $\frac{1}{P_{\mathrm{BF}}(p^{-s},p^{-t})}$, where $P_{\mathrm{BF}}(p^{-s},p^{-t})$ is a polynomial in $p^{-s},p^{-t}$.  It is normalized such that $P_{\mathrm{BF}}(0,p^{-t})=1$ and has the property that
\begin{equation} \label{eqn:propbf} P_{\mathrm{BF}}(p^{-s},p^{-t})I_\mathrm{BF}(\Phi_p,W_0,s,t) \in \mathbf{C}[p^{-s},p^{-t}]\end{equation}
for all $\Phi_p$ and $W_0$.  (It is not obvious that this polynomial exists.  Matringe first defines a GCD in the one-dimensional ring $\mathbf{C}(q^{-t})[q^{\pm s}]$ in the usual way and identifies its properties afterwards.)

Matringe computes that the GCD is equal to
\[\frac{1}{P_{\mathrm{BF}}(p^{-s},p^{-t})} = L(\pi,\mathrm{Std},s+t+\frac{1}{2})L(\pi,\wedge^2,2s),\]
where $\mathrm{Std}$ and $\wedge^2$ are the usual standard and exterior square $L$-functions on $G_0 = \GL_3$.

Now consider the specialization of the integral (\ref{eqn:mat}) to the case $s = \frac{s'}{2}$, $t=\frac{s'-1}{2}$.  This yields
\begin{equation} \label{eqn:mat2} \int_{U_{2,0}(\Q_p)\backslash H_0(\mathbf{Q}_p)} \Phi_p((0,1)g_{1,p})|\det(g_{1,p})|_p^{s'}W_0(g_p)\,d{g_p}.\end{equation}
One can again form the GCD (now in the usual sense) of the family of specialized integrals as $\Phi_v$ and $W_0$ vary, which we call $\frac{1}{P_{\mathrm{BF}}'(\pi,s')}$.  Abusing notation, we write $P_{\mathrm{BF}}(X)$ for the polynomial in $\mathbf{C}[X]$ obtained by specializing $P_{\mathrm{BF}}(p^{-s},p^{-t})$ to $P_{\mathrm{BF}}(X^2,p^{-\frac{1}{2}}X^2)$.  It follows from (\ref{eqn:propbf}), the definition of $P'_{\mathrm{BF}}$ as a GCD, and the fact that specialization yields a smaller class of integrals that we have the one divisibility $P_{\mathrm{BF}}'(X)|P_{\mathrm{BF}}(X)$ (c.f.\ \cite[Proposition 3.10]{matringe2}).

We note that if $\pi_0$ is unramified with the Satake parameter $\diag(\alpha_1,\alpha_2,\alpha_3) \in \GL_3(\mathbf{C})$, Matringe's computation shows that
\begin{equation} \label{eqn:matcomp} P_{\mathrm{BF}}(X) = L(\pi,\mathrm{Std},s')^{-1}L(\pi,\wedge^2,s')^{-1} = \prod_{i=1}^3 (1-\alpha_i X)(1-\widehat{\alpha_i}X).\end{equation}

The GCD $\frac{1}{P_{\mathrm{BF}}'(\pi,s')}$ is formally defined as the GCD of certain elements in $\mathbf{C}[X]$, obtained by taking the rational functions of $p^{-s'}$ given by the integrals (\ref{eqn:mat2}) and substituting $X$ for $p^{-s'}$.  Write $\chi_\alpha: \mathbf{Q}_p^\times \rightarrow \mathbf{C}^\times$ for the unramified character with $\chi(p)=\alpha'$.  Consider instead the integrals
\begin{equation} \label{eqn:mat3} \int_{U_{2,0}(\Q_p)\backslash H_0(\mathbf{Q}_p)} \Phi_p((0,1)g_{1,p})|\det(g_{1,p})|_p^{s'}\chi_\alpha(\det(g_{1,p}))W_0(g_p)\,d{g_p}.\end{equation}
If we substitute $\alpha p^{-s'}$ for $X$ in these integrals, we obtain the same rational functions of $X$ as we do from (\ref{eqn:mat2}).

Now let $\pi = \pi_0 \otimes \lambda$ be an irreducible admissible representation of $G(\mathbf{Q}_p) = \GL_3(\mathbf{Q}_p) \times \GL_1(\mathbf{Q}_p)$ with unramified central character.  (So $\lambda$ is unramified.)  Our local integral at $p$ is
\begin{equation} \label{eqn:mat4}\int_{U_2(\Q_v)\backslash H(\mathbf{Q}_p)}{\Phi_v((0,1)g_{1,v})W(g_v)|\det(g_{1,v})|_v^{s'}\,d{g_v}},\end{equation}
where $W$ is now a Whittaker function on $G(\mathbf{Q}_p)=\GL_3(\mathbf{Q}_p) \times \GL_1(\mathbf{Q}_p)$ and $H=\GL_2(\mathbf{Q}_p) \times \GL_1(\mathbf{Q}_p)$ is embedded using the formula (\ref{eqn:hembedsplit}).  Write $\iota: H(\mathbf{Q}_p) \rightarrow G(\mathbf{Q}_p)$ for this embedding.  We identify $G_0=\GL_3(\mathbf{Q}_p)$ with the first factor of $G$.  Note that with this identification, the embeddings of $H_0(\mathbf{Q}_p)$ (via $\iota_0$) and $H(\mathbf{Q}_p)$ (via $\iota$) do not define the same subgroups of $G(\mathbf{Q}_p)$.

Let $\alpha = \lambda(p)$ in (\ref{eqn:mat3}).  We would like to compare (\ref{eqn:mat3}) with (\ref{eqn:mat4}).  Let $W$ be any Whittaker function for $G$ and write $W_0$ for its restriction to $G_0$.  Then it follows from the definitions of $\alpha$, $\iota_0$, and $\iota$ that
\[W(\iota(g_1,g_2)) = \chi_{\alpha_\mu}(\det(g_1))W_0(\iota_0(g_1,g_2)).\]
In particular, the GCD of the local integrals (\ref{eqn:mat4}) is $P_{\mathrm{BF}}'(X)$ for $X = \alpha p^{-s'}$.

Now assume that $\pi_p$ is unramified.  We computed in Section \ref{subsubsec:splitcase} that for $\Phi_p$ unramified and $W$ the normalized spherical Whittaker function, we have
\[I(\Phi_p,W,\nu,s') = L_p^\mathrm{L}(\pi_p,\mathrm{Std},s)=\prod_{i=1}^3 (1-\alpha_i X)^{-1}(1-\widehat{\alpha_i}X)^{-1},\]
where $X=\alpha p^{-s'},$ $\alpha_i$ are the Satake parameters of $\pi_0$, and $\widehat{\alpha_i}$ is the product of the $\alpha_j$ with $j\ne i$.  By definition of the GCD together with our comparison of (\ref{eqn:mat4}) with (\ref{eqn:mat3}), we have
\[\prod_{i=1}^3 (1-\alpha_i X)(1-\widehat{\alpha_i}X)|P_{\mathrm{BF}}'(X).\]
Using the aforementioned divisibility $P_{\mathrm{BF}}'(X) | P_{\mathrm{BF}}(X)$ and (\ref{eqn:matcomp}), we deduce that $L_p^\mathrm{L}(\pi_p,\mathrm{Std},s)$ is the GCD.
\end{proof}

\section{Regulator computation} \label{sec:regulatorcomp}

We assemble the results of the preceding sections to obtain a formula for the regulator of the explicit classes we constructed in the cohomology of Picard modular surfaces.  After defining the Whittaker model, we prove statements regarding algebraicity and non-vanishing.  These yield a coarse formula in the sense that it involves an indeterminacy up to $\ol{\mathbf{Q}}^\times$.  We then obtain a more explicit regulator formula by using exact local computations carried out in Section \ref{sec:diff} and in the literature.

\subsection{The Whittaker period} \label{subsec:whittakerperiod}

Let $\pi$ be as in Section \ref{sec:diff}.  Let $V_\pi$ denote the underlying space of automorphic forms in $\pi$.  Recall from Definition \ref{defin:whit} that we have set
\[\Lambda(\varphi) = \int_{U(\Q)\backslash U(\A)}{\chi^{-1}(u)\varphi(u)\,du}.\]
Also recall from Section \ref{subsec:diff} that $V_{\ol{\mathbf{Q}}}$ is our notation for the canonical model for $\pi_f$ over $\ol{\mathbf{Q}}$ coming from interior coherent cohomology.  We again write $V_f = V_{\ol{\mathbf{Q}}} \otimes_{\ol{\mathbf{Q}}} \mathbf{C}$ and $V_\pi$ for the underlying space of automorphic forms.  By Proposition \ref{prop:whitrational}, there exists $\Lambda_f: V_f \rightarrow \mathbf{C}$ sending $V_{\ol{\mathbf{Q}}}$ to $\ol{\mathbf{Q}}$, which is unique up to $\ol{\mathbf{Q}}^\times$.

Recall from (\ref{eqn:alpha}) the canonical isomorphism
\[\alpha: V_{\ol{\mathbf{Q}}} \otimes_{\ol{\mathbf{Q}}} \mathbf{C} \stackrel{\sim}{\rightarrow} \Hom_{K_\infty}(\mathfrak{p}^+ \otimes \mathfrak{p}^-,V_\pi).\]
We recall that in Definition \ref{defin:v0}, we have selected a vector $v_2 \in \mathfrak{p}^+ \otimes \mathfrak{p}^-$ such that the projection $v_0$ of $v_2$ to $V_\tau$ has nonzero image under the Whittaker functional.  We recall that $\alpha$ factors through this projection, so below, we regard the target of $\alpha$ as an element of $\Hom_{K_\infty}(V_\tau,V_\pi)$.

\begin{prop}
If $\Lambda(\alpha(\varphi_f)(v_0))$ is nonzero for $\varphi_f \in V_{\ol{\mathbf{Q}}}$, the coset of its value in $\mathbf{C}^\times/ \ol{\mathbf{Q}}^\times$ is independent of the choice of $\varphi_f$.
\end{prop}
\begin{proof}
The functional on $V_f$ defined by $\Lambda(\alpha(\cdot)(v_0))$ is a nonzero Whittaker functional.  There exists a nonzero Whittaker functional $\Lambda_f: V_f \rightarrow \mathbf{C}$ sending $V_{\ol{\mathbf{Q}}}$ to $\ol{\mathbf{Q}}$ by Proposition \ref{prop:whitrational}.  Therefore, $\Lambda(\alpha(\cdot)(v_0))$ is a multiple of $\Lambda_f$, say by $\beta\in \mathbf{C}^\times$.  It follows that $\Lambda(\alpha(\varphi_f)(v_0))$ must be equal to the product of $\beta$ with an element of $\ol{\mathbf{Q}}$ for any $\varphi_f \in V_{\ol{\mathbf{Q}}}$.
\end{proof}

We may therefore make the following definition.
\begin{defin}
The Whittaker period $W(\pi) \in \mathbf{C}^\times/\ol{\mathbf{Q}}^\times$ is defined to be the class $\Lambda(\alpha(\varphi_f)(v_0))$ for any $\varphi_f \in V_{\ol{\mathbf{Q}}}$ satisfying $\Lambda(\alpha(\varphi_f)(v_0))$.  We also write $W(\pi,\varphi_f) = \Lambda(\alpha(\varphi_f)(v_0)) \in \mathbf{C}^\times$ for the particular value associated to a choice of $\varphi_f$.
\end{defin}

\subsection{Coarse regulator formula}

We say that $\omega_{\varphi_f}$, defined as in Section \ref{subsec:diff}, is an algebraic $(1,1)$-form on $S_{K_f}$ if $\varphi_f \in V_{\ol{\mathbf{Q}}}$.

Given $\varphi_f$, write $\varphi$ for $\alpha(\varphi_f)(v_0)$.  Recall from Definition \ref{defin:whit} that whenever $\varphi_f$ is a pure tensor in a tensor product decomposition of $\pi$, we define $W_{\varphi,v}:G(\mathbf{Q}_v) \rightarrow \mathbf{C}$ by $W_{\varphi,v}(1) = \Lambda(g\varphi)/\Lambda(\varphi)$ for each place $v$ of $\mathbf{Q}$ (with $\Lambda$ as in Section \ref{subsec:whittakerperiod}) so that $W_{\varphi,v}(1)=1$.  Then
\[\Lambda(g\varphi)=W(\pi,\varphi_f)\prod_v W_{\varphi,v}(g).\]
Let $S$ be the set of finite places $p$ where $\pi$ is ramified (if $p$ is split or inert) or $\pi$ does not have a vector fixed by $G(\mathbf{Z}_p)$ (if $p$ is ramified).  In Section \ref{sec:diff}, we have proved that when $\varphi_f$ is a pure tensor,
\[I(\Phi,\varphi,\nu,s) = W(\pi,\varphi_f) L_{\infty}(\pi,\mathrm{Std},s) L^S(\pi \times \nu_1,\mathrm{Std},s) \prod_{p \in S} I_p(\Phi_p,W_{\varphi,p},\nu_p,s).\]
Here, we set $L_{\infty}(\pi,\mathrm{Std},s) = 8iW_{0,0}(8\sqrt{2}\pi D^{-\frac{3}{4}})^{-1} \pi^4 D^{\frac{3s-3}{2}}\Gamma_{\C}(s)\Gamma_{\C}(s+1)\Gamma_{\C}(s+1)$.

For the remainder of this paper, we will use the following definition of local $L$-factors, which in principle may not coincide with those defined in Definition \ref{defin:stdl}.  This definition provides concrete $L$-factors at all places.

\begin{definition} \label{defin:locallfactors} For any finite place $p$, define $L_{p}^\mathrm{G}(\pi_p \times \nu_{1,p},\mathrm{Std},s)$ to be the greatest common divisor of the zeta integrals $I_p(\Phi_p,W,\nu_p,s)$ as $\Phi_p$ and $W$ vary, which exists by the aforementioned work of Baruch \cite{baruch}.  It follows from Proposition \ref{prop:locrationality} and Lemma \ref{trivLoc} that $L_p^\mathrm{G}(\pi_p \times \nu_{1,p},\mathrm{Std},s) = 1/Q(q^{-s})$, where $Q \in \Qbar[X]$ has constant term $1$.  We define the standard $L$-function
\[L(\pi \times \nu_1,\mathrm{Std},s) = \prod_{p < \infty}{L_{p}^\mathrm{G}(\pi_p \times \nu_{1,p},\mathrm{Std},s)}\]
and completed $L$-function $\widehat{L}(\pi \times \nu_1, \mathrm{Std},s) = L_{\infty}(\pi_{\infty},\mathrm{Std},s)L(\pi \times \nu_1,\mathrm{Std},s)$.  (Also recall from Section \ref{subsec:arch} that we have assumed $\nu_{\infty}$ is trivial at infinity.)  If $S$ is a set of finite places, we will write $L^S(\pi \times \nu_1, \mathrm{Std},s)$ and $\widehat{L}^S(\pi \times \nu_1,\mathrm{Std},s)$ for the product of $L$-factors for finite places not in $S$ and the product of $L$-factors at all places not in $S$, respectively.
\end{definition}

\begin{remark} Theorem \ref{thm:unramgcd} can now be rephrased as saying that $L_p^{\mathrm{L}}(\pi_p \times \nu_{1,p},\mathrm{Std},s)=L_{p}^\mathrm{G}(\pi_p \times \nu_{1,p},\mathrm{Std},s)$ when $\pi_p$ and $\nu_{1,p}$ are unramified.\end{remark}

\begin{remark} Koseki and Oda \cite{ko} compute the GCD (as a product of $\Gamma$-factors, and ignoring any nonvanishing entire factors) at the Archimedean place and obtain $L_{\infty}(\pi_{\infty},\mathrm{Std},s)$.\end{remark}

\begin{remark}
Recall that while $L_p^G(\pi_p \times \nu_{1,p},\mathrm{Std},s)$ is defined for every $\pi_p$, we have only defined $L_p^\mathrm{L}(\pi_p \times \nu_{1,p},\mathrm{Std},s)$ when $\pi_p$ has a vector fixed by $G(\mathbf{Z}_p)$ if $p$ is inert or ramified in $E$.
\end{remark}

We have the following easy consequence of the Rankin-Selberg integral.
\begin{lemma}[{\cite{gelbartPS}}] \label{lem:vanishingat0} Suppose either that $\nu_1^2 \cdot \omega_\pi|_{Z(\mathbf{A})}$ is non-trivial or that $\pi(\nu_1 \circ \mu)$ does not have a period over $[H]=H(\Q)Z(\A)\backslash H(\A)$.  Then $L(\pi \times \nu_1,\mathrm{Std},s)$ vanishes to order at least $1$ at $s=0$.  \end{lemma}
\begin{proof} Since each local $L$-factor is a GCD at places of ramification of $\pi$ or $\nu_1$ and one can choose unramified data elsewhere, there are choices of finitely many pairs $(\Phi_i,\varphi_i)$ such that
\[\sum_{i}{I(\Phi_i,\varphi_i,\nu,s)} = \widehat{L}(\pi \times \nu_1,\mathrm{Std},s).\]
If $\nu_1^2 \cdot \omega_\pi|_{Z(\mathbf{A})}$ is trivial, then by Lemma \ref{lem:adelicclassical} and our central character condition, the Eisenstein series $E(\Phi,g,\nu,s)$ has a simple pole at $s=0$ with a residue proportional to $\Phi(0)\nu_1$.  Hence the residue at $s=0$ of the left-hand side is proportional to a period of $\pi(\nu_1 \circ \mu)$ over $[H]$, and thus vanishes by assumption.  If $\nu_1^2 \cdot \omega_\pi|_{Z(\mathbf{A})}$ is not trivial, $E(\Phi,g,\nu,s)$ has no pole at $s=0$.  It follows that $\widehat{L}(\pi,\mathrm{Std} \times \nu_1,s)$ is holomorphic at $s=0$ in either case.  Since $L_{\infty}(\pi,\mathrm{Std},s)$ has a simple pole at $s=0$ coming from $\Gamma_{\C}(s)$, $L(\pi \times \nu_1,\mathrm{Std},s)$ vanishes to order at least $1$ at $s=0$ to compensate. \end{proof}

\begin{definition} If $\nu_1^2 \cdot \omega_\pi|_{Z(\mathbf{A})}$ is non-trivial or $\pi(\nu_1 \circ \mu)$ has no period over $[H]$, we define $L^{(1)}(\pi \times \nu_1,\mathrm{Std})$ to be coefficient of $s$ in the Taylor expansion of $L(\pi\otimes \nu_1,\mathrm{Std},s)$ at $s=0$.\end{definition}

In our first main theorem, we show that for any suitably chosen data in the global integral, we may interpret its special value as a regulator pairing with an element of the higher Chow group.  We will rely on the main results of Sections \ref{sec:higherchow}, \ref{sec:hecke}, and \ref{sec:klf}.

We have defined cycles $C_{W,K_{W,f}} \subseteq S_{K_f}$ in Section \ref{sec:higherchow}, where both these varieties were initially defined over their reflex field $E$.  Via the projection $H \rightarrow \GL_2$ of groups over $\mathbf{Q}$, we have a morphism of the corresponding Shimura varieties that is defined over $E$.

Recall from Corollary \ref{coro:allphi} that we have units $u(\Phi,\nu_1)$ on the Shimura varieties associated to $\GL_{2/\mathbf{Q}}$ that are defined over $\mathbf{Q}(\nu_1)$.  By pullback, we obtain units on $C_{K_f}$ that are initially defined over $E(\nu_1)$.  For the purposes of discussion of the regulator map, we restricted scalars from $E$ to $\mathbf{Q}$ on $S_{K_f}$ and $C_{K_f}$.  Then these units are defined over $\mathbf{Q}(\nu_1)$.

In the following, we note that since the image of the regulator pairing is a $\mathbf{C}$-vector space, we may extend its definition to $\mathrm{Ch}^2(\ol{S}_{K_f/F},1) \otimes_\mathbf{Q} \ol{\mathbf{Q}}$ by linearity, where $\ol{\mathbf{Q}}$ is identified with the algebraic numbers in our fixed copy of $\mathbf{C}$ and $F\subseteq \mathbf{C}$ is a number field with a chosen embedding in $\mathbf{C}$.

\begin{thm}[Geometrization of the global integral] \label{thm:geom}
Assume that either $\nu_1^2 \cdot \omega_\pi|_{Z(\mathbf{A})}$ is non-trivial or $\pi(\nu_1 \circ \mu)$ has no period over $[H]$.  Let $\Phi$ be a $\mathbf{Q}$-valued factorizable Schwartz-Bruhat function $\Phi$ such that if $\nu_1^2 \cdot \omega_\pi|_{Z(\mathbf{A})}$ is trivial, then $\Phi(0)=0$.  For any $\varphi_f \in V_{\ol{\mathbf{Q}}}$, there exists an open compact $K_f \subseteq G(\mathbf{A}_f)$ and an element $[\eta] \in \mathrm{Ch}^2(\ol{S}_{K_f/\mathbf{Q}(\nu_1)},1) \otimes_\mathbf{Q} \ol{\mathbf{Q}}$ such that
\begin{equation}\label{eqn:reginterp}\langle \mathrm{Reg}([\eta]),\omega_{\varphi_f} \rangle_{K_f} = I(\Phi,\varphi,\nu,0).\end{equation}
\end{thm}

\begin{proof}  By Corollary \ref{coro:allphi} and the preceding remarks on pulling back units from $\GL_2$ to $H$, for some open compact $K_{H,f} \subseteq H(\mathbf{A}_f)$, there exists a unit $u(\Phi,\nu_1) \in \mathcal{O}(C_{K_{H,f}/\mathbf{Q}(\nu_1)})^\times \otimes \mathbf{Q}(\nu_1,\omega_\pi)$ such that $E(g_1,\Phi,\nu,s) = \log|u(\Phi,\nu_1)|+O(s)$.

Fix $K_f'$ small enough so that it fixes $\varphi_f$ and $K_{H,f} \supseteq K_f' \cap H(\mathbf{A}_f)$.  For any $K_f\subseteq G(\mathbf{A}_f)$, we write $C_{K_f}$ for $C_{K_f \cap H(\mathbf{A}_f)}$ in what follows.  Let $p$ be a finite prime such that all the data is unramified.  Note also that $\pi_p$ is generic and thus is not the trivial representation.  Using Lemma \ref{lem:heckenonzero}, we may select a $\mathbf{Q}$-valued Hecke operator $T_\xi$ such that $T_\xi' = T_\xi-\mu(\xi)$ acts by a nonzero scalar $\alpha_\xi\in \ol{\mathbf{Q}}^\times$ on $\varphi_f$.  By Lemma \ref{lem:heckeduality}, we have an identity
\begin{equation}\label{eqn:applyhecke} \langle (C_{K_f'},u(\Phi,\nu_1)) \cdot T_\xi' , \omega_{\varphi_f} \rangle_{K_f^{\prime R}} = \langle (C_{K_f'},u(\Phi,\nu_1)), T_\xi' \cdot \omega_{\varphi_f} \rangle_{K_f'} = \alpha_\xi\langle (C_{K_f'},u(\Phi,\nu_1)), \omega_{\varphi_f} \rangle_{K_f'},\end{equation}
with notation as used there.  Here, we have written $\omega_{\varphi_f}$ for the differential form on $S_{K_f'}$ as well as its pullback to $S_{K_f^{\prime R}}$.

We finally apply Proposition \ref{prop:cycles} to $(C_{K_f'},u(\Phi,\nu_1)) \cdot T_\xi$.  We take the closure of $(C_{K_f'},u(\Phi,\nu_1)) \cdot T_\xi$ in $\ol{S}_{K_f^{\prime R}}$ in order to obtain a formal sum of cycles and functions there.  Since $u(\Phi,\nu_1)$ has degree 0 on each irreducible curve in the formal sum, the sum meets condition (2).(a).  The annihilation property of $T_\xi'$ from Lemma \ref{lem:annihilate} together with the compatibility (\ref{eqn:compat}) of the Hecke actions on the cusps and cycles shows that the formal sum meets condition (2).(b).  It follows from the proposition that $(C_{K_f'},u(\Phi,\nu_1)) \cdot T_\xi$ can be extended to an element $[\eta]$ of the higher Chow group at level $K_f \subseteq K_f^{\prime R}$ by pulling back and adding cycles supported along the boundary.  These cycles do not contribute to the regulator integral, so $\langle [\eta],\omega_{\varphi_f} \rangle_{K_f} = \langle (C_{K_f'},u(\Phi,\nu_1)) \cdot T_\xi' , \omega_{\varphi_f} \rangle_{K_f^{\prime R}}.$  Multiplying $[\eta]$ by $\alpha_\xi^{-1}$, we obtain an element of $\mathrm{Ch}^2(\ol{S}_{K_f/\mathbf{Q}(\nu_1)},1) \otimes_\mathbf{Q} \ol{\mathbf{Q}}$ satisfying (\ref{eqn:reginterp}).
\end{proof}

\begin{remark}
The field of definition $\mathbf{Q}(\nu_1)$ of the unit can be changed to $\mathbf{Q}$ by first twisting $\pi$ by $\nu_1$ as in Remark \ref{remark:nonu1}.  This only changes $\omega_{\varphi_f}$ and $u(\Phi,\nu_1)$ by multiplying them by the respective scalars $\nu_1(\mu(k))$ and $\nu_1(\mu(k))^{-1}$ on each connected component of $C_{K_{H,f}}$.
\end{remark}

We prove the second main theorem, which concerns algebraicity of the global integral under a hypothesis that eliminates its pole.

\begin{theorem}[Algebraicity] \label{thm:rationality} Assume that either $\nu_1^2 \cdot \omega_\pi|_{Z(\mathbf{A})}$ is non-trivial or $\pi \times \nu_1$ has no period over $[H]$.  Let $\Phi$ be an $\ol{\mathbf{Q}}$-valued Schwartz-Bruhat function $\Phi$.  For all $\varphi_f \in V_{\ol{\mathbf{Q}}}$, we have
\begin{equation} \label{eqn:rationality} I(\Phi,\varphi,\nu,0) \in \Qbar W(\pi)W_{0,0}(8\sqrt{2}\pi D^{-\frac{3}{4}})^{-1} \pi^4 L^{(1)}(\pi \times \nu_1,\mathrm{Std}),\end{equation}
\end{theorem}

\begin{proof}  Since every local $L$-factor is defined using a GCD, there exists a finite set of pairs of factorizable data $(\Phi_i,\varphi_i)$ such that
\[\sum_i I(\Phi_i,\varphi_i,\nu,s) = \widehat{L}(\pi \times \nu_1,\mathrm{Std},s).\]
By Lemma \ref{lem:vanishingat0} and our hypotheses, the left-hand side is holomorphic at $s=0$, so the right-hand side is as well.

Now let $\varphi = \alpha(\varphi_f)(v_0)$ for a factorizable $\varphi_f \in V_{\ol{\mathbf{Q}}}$.  For any factorizable $\Phi$ that takes algebraic values, we have
\[I(\Phi,\varphi,\nu,s) = \widehat{L}(\pi \times \nu_1,\mathrm{Std},s) \prod_{p \in S} \frac{I_p(\Phi_p,W_p,\nu,s)}{L_p(\pi \times \nu_1,\mathrm{Std},s)},\]
where $S$ consists of the finite set of places $p$ where either $E$, $\Phi$, or $\pi$ is ramified.  We write the local $L$-factor as $\frac{1}{Q_p(p^{-s})}$ where $Q_p \in \ol{\mathbf{Q}}[X]$ has $Q_p(0)=1$.  Then $I_p =\frac{P_p(p^{-s})}{Q_p(p^{-s})}$ by the GCD property, so we may write
\[I(\Phi,\varphi,\nu,s) = \widehat{L}(\pi \times \nu_1,\mathrm{Std},s) \prod_{p \in S} P_p(p^{-s}).\]
Since $\widehat{L}(\pi \times \nu_1,\mathrm{Std},s)$ is holomorphic at $s=0$, $P_p(1)$ is algebraic, and the $\Gamma$-function has algebraic residues, we obtain (\ref{eqn:rationality}) in the case of factorizable data.  The general case follows immediately.
\end{proof}

The next main theorem addresses non-vanishing of the regulator pairing.

\begin{theorem}[Nonvanishing] \label{coarseThm}  Assume that the twisted base change of $\pi$ to $\mathrm{Res}_\mathbf{Q}^E \GL_3 \times \mathbf{G}_m$ remains cuspidal.  We break into two cases.
\begin{enumerate}
	\item If $\nu_1^2 \cdot \omega_\pi|_{Z(\mathbf{A})}$ is trivial, assume that $\pi$ is not $(H,\nu_1^{-1})$-distinguished for least one finite place $\ell$.  Then there exists $\varphi_f \in V_{\ol{\mathbf{Q}}}$ and a factorizable $\mathbf{Q}$-valued Schwartz-Bruhat function $\Phi$ such that $\Phi(0)=0$ and $I(\Phi,\varphi,\nu,0) \neq 0$.
	\item If $\nu_1^2 \cdot \omega_\pi|_{Z(\mathbf{A})}$ is non-trivial, then there exists $\varphi_f \in V_{\ol{\mathbf{Q}}}$ and a factorizable $\mathbf{Q}$-valued Schwartz-Bruhat function $\Phi$ such that $I(\Phi,\varphi,\nu,0) \neq 0$.
\end{enumerate}
\end{theorem}

\begin{proof}
Write $\widehat{\nu} = (\nu_2,\nu_1)$.  We have the identity $I(\Phi,\varphi,\nu,s) = I(\widehat{\Phi},\varphi,\widehat{\nu},1-s)$ from the functional equation (\ref{EisFE}) for the Eisenstein series.

We first pick a set $S$ of finite primes containing all those where $E$ or $\pi_p$ is ramified as well as the place $\ell$ if $\nu_1^2 \cdot \omega_\pi|_{Z(\mathbf{A})}$ is trivial.  At finite primes not in $S$, we select $\Phi$ to be unramified, so that the local integral computes $L_p(\pi \times \nu_1,\mathrm{Std},s)$ at those places.  (Note that $\widehat{\Phi}$ is also unramified in this case.)  Suppose we are given $\varphi_f \in V_{\ol{\mathbf{Q}}}$ that is spherical away from $S$ and is a pure tensor in a tensor product decomposition $V_{\ol{\mathbf{Q}}} = \otimes_{v<\infty} V_{\ol{\mathbf{Q}},v}$ defined over $\ol{\mathbf{Q}}$.  If $\Phi$ is also factorizable,
\begin{equation} \label{eqn:nonvanishing} I(\Phi,\varphi,\nu,s) = I(\widehat{\Phi},\varphi,\widehat{\nu},1-s) = \widehat{L}^S(\pi \times \nu_2,\mathrm{Std},1-s)\prod_{p \in S} I_p(\widehat{\Phi}_p,\varphi,\widehat{\nu}_p,1-s).\end{equation}
The quantity $\widehat{L}^S(\pi \times \nu_2^{-1},\mathrm{Std},1-s)$ is nonzero at $s=0$ by Theorem \ref{thm:pnt} and our hypothesis on the cuspidality of the twisted base change.  At places $p\in S$ other than $\ell$, we use Lemma \ref{trivLoc} to replace $\varphi_f$ and $\widehat{\Phi}$ with data trivializing the local integrals $I_p(\widehat{\Phi}_p,\varphi,\widehat{\nu}_p,1-s)$.  The argument in Lemma \ref{trivLoc} replaces $\varphi_f$ with $\pi(\eta)\varphi_f$ for $\eta \in C_c^\infty(G(\mathbf{Q}_p),\mathbf{Q})$ and uses a function $\widehat{\Phi}$ that takes values in $\mathbf{Q}$.  So the new choice of $\varphi_f$, say $\varphi_f'$, lies in $V_{\ol{\mathbf{Q}}}$ and the Fourier transform $\Phi$ of $\widehat{\Phi}$ is $\ol{\mathbf{Q}}$-valued.  Write $\varphi'=\alpha(\varphi_f')(v_0)$.

If $\nu_1^2 \cdot \omega_\pi|_{Z(\mathbf{A})}$ is non-trivial, it follows that $I(\Phi,\varphi',\nu,0) \ne 0$.  By linearity in $\Phi$, there must also exist one that is $\mathbf{Q}$-valued.  (Note that our argument in this case also shows that $\widehat{L}^S(\pi \times \nu_1,\mathrm{Std},1-s)$ is holomorphic at $s=0$ since the left-hand side of (\ref{eqn:nonvanishing}) is holomorphic for all $\widehat{\Phi}$ and $\varphi$.)

If $\nu_1^2 \cdot \omega_\pi|_{Z(\mathbf{A})}$ is trivial, we are reduced to picking $\varphi_f \in V_{\ol{\mathbf{Q}}}$ and a $\ol{\mathbf{Q}}$-valued $\Phi_\ell$ so that $\Phi_\ell(0)=0$ and $I_\ell(\widehat{\Phi}_\ell,\varphi,\widehat{\nu}_\ell,1)$ is nontrivial.  Note that by if $\pi$ is not $(H,\nu_1^{-1})$-distinguished, the period of $\pi(\nu_1 \circ \mu)$ over $[H]$ vanishes and $I(\Phi,\varphi,\nu,s)$ is holomorphic.  By our selection of data at $p \in S \setminus \set{\ell}$ and Theorem \ref{thm:pnt}, we deduce that $I_\ell(\widehat{\Phi}_\ell,\varphi,\widehat{\nu}_\ell,1-s)$ cannot have a pole at $s=0$ for any $\varphi$ or $\widehat{\Phi}_\ell$.

Assume for contradiction that there is no choice of a $\varphi_{\ell} \in V_{\ol{\mathbf{Q}},\ell}$ and a $\mathbf{Q}$-valued $\Phi_\ell$ so that $\Phi_\ell(0)=0$ and $I_\ell(\widehat{\Phi}_\ell,\varphi,\widehat{\nu}_\ell,1-s)|_{s=0} \ne 0$.  By considering $\mathbf{C}$-linear combinations of the data, we deduce that the hypotheses of the $s=1$ case of Lemma \ref{distLem} hold.  It follows that $\pi$ is $(H,\nu_2^{-1})$-distinguished at $\ell$.  Since $\nu_2 = \nu_1^{-1}\omega_\pi^{-1}|_{Z(\mathbf{Q}_\ell)}$ by Hypothesis \ref{hypo:central} and the triviality of $\nu_1^2 \cdot \omega_\pi|_{Z(\mathbf{A})}$, we have $\nu_2^{-1}=\nu_1^{-1}$, contradicting our hypothesis on $\pi_\ell$.

It follows that there exist $\varphi_{\ell}$ and $\Phi_\ell$ so that $\Phi_\ell(0)=0$ and $I_\ell(\widehat{\Phi}_\ell,\varphi,\nu_\ell,1-s)|_{s=0} \ne 0$, which implies that $I(\Phi,\varphi,\nu,s)|_{s=0}$ is nonzero.  By linearity, we may again assume $\Phi$ takes values in $\mathbf{Q}$.
\end{proof}

\begin{remark} \label{remark:endoscopyfunctional}
The local hypothesis in Theorem \ref{coarseThm} that $\pi$ has a finite place that is not $(H,\nu_1^{-1})$-distinguished is exactly equivalent to saying that the twisted representation $\pi \times \nu_1$ does not belong to an endoscopic $L$-packet by \cite[Theorem D]{grs4}.  (Note that in the notation there, one always has $\langle \rho,\pi \rangle = 1$ if $\pi$ is generic and endoscopic.)
\end{remark}

We may now obtain a class number formula and the non-vanishing of the higher Chow group.  The following corollary is simply the concatenation of the three main theorems above.

\begin{coro}[Class number formula, coarse form] \label{coro:chowreg} Under the assumptions of Theorem \ref{coarseThm}, there exists $\varphi_f \in V_{\ol{\mathbf{Q}}}$, $K_f$ sufficiently small, and a class $[\eta] \in \mathrm{Ch}^2(\ol{S}_{K_f/\mathbf{Q}(\nu_1)},1)$ for which
\[\langle \mathrm{Reg}([\eta]),\omega_{\varphi_f} \rangle_{K_f}  \in \Qbar^\times W(\pi)W_{0,0}(8\sqrt{2}\pi D^{-\frac{3}{4}})^{-1} \pi^4 L^{(1)}(\pi \times \nu_1,\mathrm{Std}).\]
In particular, $\langle \mathrm{Reg}([\eta]),\omega_{\varphi_f} \rangle_{K_f} \ne 0$.
\end{coro}

\begin{proof}
It is clear from Theorems \ref{thm:geom}, \ref{thm:rationality}, and \ref{coarseThm} that there exists $\varphi_f$, $K_f$, and a class $[\eta] \in \mathrm{Ch}^2(\ol{S}_{K_f/\mathbf{Q}(\nu_1)},1) \times_\mathbf{Q} \ol{\mathbf{Q}}$ with the desired property.  By linearity, there must also be $[\eta'] \in \mathrm{Ch}^2(\ol{S}_{K_f},1)$ that pairs nonvanishingly with $\omega_{\varphi_f}$ as well.
\end{proof}

\subsection{Finer computation} \label{subsec:finer}

We obtain a finer regulator formula by specifying vectors at each finite place.  We work separately at split, inert, and ramified places.  Our computations in this section will include stronger hypotheses on $\pi$ than in the preceding section, as we would like to leverage certain exact computations carried out in Section \ref{subsec:ramifiedlocal} and in the literature.

\subsubsection{Local integrals at split places}  \label{subsubsec:finesplit}

Suppose $p <\infty$ splits in the quadratic extension $E$ so that $\GU(J)$ becomes isomorphic to $\GL_3 \times \GL_1$ over $\Q_p$.  Since everything is local, and to separate out the similitude, we write $\pi_0 \otimes \lambda$ for the representation of $\GL_3 \times \GL_1$ given by $\pi_p$.  We impose a hypothesis on $\pi_p$ in this section: we require that $\omega_{\pi_p}$ and $\nu_1$ are both unramified.

Define $K_n \subseteq \GL_3(\Z_p)$ to be the subgroup of elements whose final row is congruent to $(0,0,1)$ modulo $p^n$.  Denote by $n = n(\pi)$ the conductor of $\pi$, so that the space of fixed vectors under $K_n$ in the space of $\pi$ has dimension $1$.  (See \cite{jpss2} and \cite{jacquet} or \cite{matringe3} for the notion of conductor and its properties.)  We write $\varphi_n$ for a generator of this space, which we call a new vector.

Since we have computed the local $L$-factor in the unramified case, we assume $n \ge 1$.  Define the Schwartz-Bruhat function $\Phi_n$ to be $(p^{2n}-p^{2n-2})\cdot \mathrm{char}(p^n\mathbf{Z}_p \oplus (1+p^n\mathbf{Z}_p))$, where $\mathrm{char}(\cdot)$ denotes the characteristic function.  Define $W_{\varphi_n}$ as in Definition \ref{defin:whit}.

Miyauchi has proved an explicit formula \cite[Theorem 4.1]{my5} for the value of $W_{\varphi_n}$ along elements of the torus.  This formula does not completely determine $W_{\varphi_n}$ since $BK_n$ is smaller than $G$.  However, if one selects the Schwartz-Bruhat function $\Phi_n$, the Iwasawa decomposition reduces the integral to the torus.  Miyauchi and Yamauchi \cite[Theorem 5.1]{my} applied this strategy to compute the Bump-Friedberg local integral explicitly up to a constant, which is the same integral studied by Matringe \cite{matringe2}.

We will use Remark \ref{remark:nonu1} to twist $\pi$ so that $\nu_1$ is trivial, which does not affect our assumption that $\omega_{\pi_p}$ is unramified.  Our local integrals differ slightly from the Bump-Friedberg construction, as we explained during the proof of Theorem \ref{thm:unramgcd}.  However, by substituting $\lambda(p) p^{-s}$ in place of $p^{-s}$ in the formulas, which is formally justified just as in the proof of Theorem \ref{thm:unramgcd}, we may obtain the following calculations of our local integrals from \cite[Theorem 5.1]{my} and a computation of the constant that we provide below.

\begin{thm} \label{thm:finesplit}
Suppose that $n(\pi) \ge 1$ and $\omega_{\pi_p}$ and $\nu_1$ are both unramified.  Write $\alpha_\mu = \lambda(p)$ and $\alpha_{\nu_1}=\nu_1(p)$. We have the following identifications of local integrals when $n(\pi)\ge 1$.
\begin{enumerate}
\item If the twisted Godement-Jacquet $L$-function $L(\pi_0,\mathrm{Std},s) = (1-\alpha_1 p^{-s})^{-1}(1-\alpha_2 p^{-s})^{-1}$ with the $\alpha_i \neq 0$, then
\[I(\Phi_n,\varphi_n,\nu,s) = (1-\alpha_\mu\alpha_{\nu_1}\alpha_1 p^{-s})^{-1}(1-\alpha_\mu\alpha_{\nu_1}\alpha_2 p^{-s})^{-1} (1-\alpha_\mu \alpha_{\nu_1}\alpha_1 \alpha_2 p^{-s})^{-1}.\]
\item If the Godement-Jacquet $L$-function $L(\pi_0,\mathrm{Std},s) = (1-\alpha_1 p^{-s})^{-1}$ with $\alpha_1 \neq 0$, then
\[I(\Phi_n,\varphi_n,\nu,s) = (1-\alpha_\mu\alpha_{\nu_1}\alpha_1 p^{-s})^{-1}.\]
\item Finally, if the Godement-Jacquet $L$-function $L(\pi_0,\mathrm{Std},s) = 1$, then $I(\Phi_n,\varphi_n,\nu,s) = 1$.
\end{enumerate}
\end{thm}

\begin{proof}
As the constant is omitted from \cite[Theorem 5.1]{my}, we briefly explain the calculation.  Write $T_H$ for the torus of $H$ and let $K_H = H(\mathbf{Z}_p)$.  Continue to use the subscript 1 to denote the projection $H \rightarrow \GL_2$.  The local integral
\[I_p(\Phi_{n},\varphi,\nu,s) = \int_{U_2\backslash H}{\Phi_n((0,1)g_1)|\det(g_1)|^{s}W_{\varphi_n}(g)\,dg}.\]
is, by the Iwasawa decomposition, equal to
\[I_p(\Phi_{n},\varphi,\nu,s) = \int_{T_H} \int_{K_H}{\delta_{B_2}^{-1}(\det(t_1))\Phi_n((0,1)t_1k_1)|\det(t_1)|^{s}W_{\varphi_n}(tk)\,dt\,dk}.\]
Write $t = (\diag(\tau_1,\tau_3),\tau_2) \in \GL_2(\mathbf{Q}_p) \times \mathbf{Q}_p^\times$ and $k_1 = \mm{a}{b}{c}{d}$.  Then $\Phi_n((0,1)t_1k_1) = \Phi_n((\tau_3c,\tau_3d))$.  Since the elements $c,d \in \mathbf{Z}_p$ are coprime, we must have $\tau_3 \in \mathbf{Z}_p$, from which we obtain $\tau_3,d \in \mathbf{Z}_p^\times$.  We translate $t$ and the bottom row of $k_1$ by inverse elements of $T_H \cap K_H$ so that $tk$ is the same while $\tau_3=1$ and $d \in 1+p^n\mathbf{Z}_p$.  Then $c \in p^n\mathbf{Z}_p$.  Since the measure of $K_n \cap H$ is $(p^{2n}-p^{2n-2})^{-1}$, we are left with the integral
\[\int_{T_H^0}{\delta_{B_2}^{-1}(\tau_1)|\det(\tau_1)|^{s}W_{\varphi_n}(t)\,dt},\]
where $T_H^0$ is the subgroup with $\tau_3=1$.  This reduces immediately to a sum over $t = (\diag(p^a,1),p^b)$ with $a \ge b \ge 0$ due to the vanishing properties of the Whittaker model.  Using the method from Theorem \ref{thm:unramgcd} to accomodate the twist from $\lambda$, this sum is then computed in \cite[Theorem 5.1]{my} (letting $s_1=s_2=s$ there) and yields the expressions above.
\end{proof}

In case (3) of Theorem \ref{thm:finesplit}, it follows immediately from the argument of Theorem \ref{thm:unramgcd} and the calculations of Matringe \cite[Theorem 3.3]{matringe2} that
\[I(\Phi_n,\varphi_n,\nu,s)=L_p^{\mathrm{L}}(\pi_p \times \nu_{1,p},\mathrm{Std},s)=L_{p}(\pi_p \times \nu_{1,p},\mathrm{Std},s)=1.\]
By contrast, if we compare the degree of the $L$-functions appearing in cases (1) and (2) with Matringe's computations (which show that the two-variable integral has the full Langlands $L$-factor as its GCD), we may reasonably conjecture that under our hypothesis that $\omega_{\pi_p}$ and $\nu_{1,p}$ are unramified that $I(\Phi_n,\varphi_n,\nu,s)$ is the product of $L_{p}(\pi_p \times \nu_{1,p},\mathrm{Std},s)$ with a factor of the form $(1-\alpha p^{-s})$.  We would like to make this factor precise by computing the true Langlands $L$-factors using the Bernstein-Zelevinsky classification.  Note that we may assume $\pi_p$ is tempered by Remark \ref{remark:temperedness}.

Suppose that the Godement-Jacquet $L$-function $L(\pi_0,\mathrm{Std},s) = (1-\alpha_1 p^{-s})^{-1}$ with $\alpha_1 \neq 0$.
\begin{enumerate}
  \item If $\pi_0$ is discrete series, then it must be a Steinberg representation induced from an unramified character, so one has that $L_{p}^\mathrm{L}(\pi_p \times \nu_{1,p},\mathrm{Std},s) = (1-\alpha_\mu\alpha_{\nu_1}\alpha_1 p^{-s})^{-1}(1-p\alpha_\mu\alpha_{\nu_1}\alpha_1^2 p^{-s})^{-1}$.  Note that $p\alpha_1$ is unitary.
  \item If $\pi_0$ is not discrete series, it must correspond to the sum of an unramified character with either a Steinberg whose character is ramified (and quadratic), a supercuspidal with unramified central character, or the sum of two ramified characters whose product is unramified.  In either case the central character of the Steinberg, supercuspidal, or product of two ramified character is unramified, say with Satake parameter $\beta$.  In this case, $L_{p}^\mathrm{L}(\pi_p \times \nu_{1,p},\mathrm{Std},s) = (1-\alpha_\mu\alpha_{\nu_1}\alpha_1 p^{-s})^{-1}(1-p\alpha_\mu\alpha_{\nu_1}\beta p^{-s})^{-1}$.
\end{enumerate}

Suppose that the Godement-Jacquet $L$-function $L(\pi_0,\mathrm{Std},s) = (1-\alpha_1 p^{-s})^{-1}(1-\alpha_2 p^{-s})^{-1}$ with $\alpha_1,\alpha_2 \neq 0$.  Then $\pi_0$ is the sum of an unramified character (say with Satake parameter $\alpha_1$) with a Steinberg induced from an unramified character and
\begin{align*}
  L_{p}^\mathrm{L}(\pi_p \times \nu_{1,p},\mathrm{Std},s) =& (1-\alpha_\mu\alpha_{\nu_1}\alpha_1 p^{-s})^{-1}(1-\alpha_\mu\alpha_{\nu_1}\alpha_2 p^{-s})^{-1}(1-\alpha_\mu\alpha_{\nu_1}\alpha_1\alpha_2 p^{-s})^{-1}\\&\times (1-p^{\frac{1}{2}}\alpha_\mu\alpha_{\nu_1}\alpha_2^2 p^{-s})^{-1}.\end{align*}
In this case, $p^{\frac{1}{2}}\alpha_2$ is unitary.

\subsubsection{New vectors and local integrals at inert places} \label{subsubsec:fineinert} Miyauchi \cite{my1,my2,my3,my4} has developed a theory of new vectors for the group $\U(J)$ when the residual characteristic is not $2$ and computed the standard $L$-factors of all generic representations, where these factors are defined as the GCD of the local Gelbart--Piatetski-Shapiro integral.  Since this integral is equivalent to ours, we obtain formulas for $L_p(\pi_p \times \nu_{1,p},\mathrm{Std},s)$ in all cases.

Miyauchi works with the anti-diagonal Hermitian form
\[J'' = \(\begin{array}{ccc} & & 1 \\ & 1 &\\ 1 & &\end{array}\),\]
which defines a group isomorphic to ours; we use Miyauchi's form in this discussion.  Write $E$ for the local unramified extension of $\mathbf{Q}_p$ and define the subgroup
\[K_n = \U(J'') \cap \(\begin{array}{ccc} \mathfrak{O}_E & \mathfrak{O}_E & \mathfrak{p}_E^{-n} \\ \mathfrak{p}_E^n & 1+\mathfrak{p}_E^n & \mathfrak{O}_E\\ \mathfrak{p}_E^n & \mathfrak{p}_E^n & \mathfrak{O}_E \end{array}\).\]
The conductor of $\pi$ is the smallest $n$ such that $V^{K_n}\ne 0$; Miyauchi shows that $V^{K_n}$ is then one-dimensional.  The newform, which is unique up to scaling, is a non-trivial vector in this space.  Define $\Phi_n$ to be the characteristic function of $p^n\mathbf{Z}_p \oplus \Z_p$, and let $\varphi_n$ denote a new vector.  By incorporating the twist $\nu_1$ into $\pi$ as discussed at the beginning of this section, we may write Miyauchi's result in the following way.
\begin{theorem}[{\cite[Theorem 3.4]{my4}}] \label{thm:fineinert} If $p \ne 2$, we have $I(\Phi_n,\varphi_n,\nu,s) = L_p^\mathrm{G}(\pi \times \nu_{1,p},\mathrm{Std},s)$. \end{theorem}

\subsubsection{Fine regulator formula}

We make the following assumptions on the generic cuspidal automorphic representation $\pi$.
\begin{enumerate}
  \item The representation $\pi_\infty$ is a discrete series with Blattner parameter $(1,-1)$.
  \item The base change of $\pi$ to $\mathrm{Res}_\mathbf{Q}^E(\GL_3 \times \mathbf{G}_m)$ is cuspidal.
  \item The central character $\omega_\pi$ and twist $\nu_1$ may be ramified only at inert primes.  Both are trivial at infinity.
  \item Either $\nu_1^2\omega_\pi|_{Z(\mathbf{A})}$ is non-trivial or $\pi(\nu_1 \circ \mu)$ has no period over $[H]$.
  \item For every rational prime $p$ that ramifies in $E$, $\pi_p$ has a non-zero vector fixed by $G(\mathbf{Z}_p)$.
  \item Either
  \begin{enumerate}
    \item the prime $p=2$ splits in $E$ or
    \item $p=2$ is inert in $E$ and $\pi_p$ and $\nu_{1,p}$ are both unramified.
  \end{enumerate}
\end{enumerate}
Recall that we have computed $L_p^{\mathrm{L}}(\pi_p \times \nu_{1,p},\mathrm{Std},s)$ explicitly in terms of our local integral when $p$ is ramified (in (\ref{eqn:ramlocaltwist})) or in case (1) or (2) when $p$ is split (in Section \ref{subsubsec:finesplit}).  On the other hand, Miyauchi has given a formula, Theorem \ref{thm:fineinert}, for $L_p^{\mathrm{G}}(\pi_p \times \nu_{1,p},\mathrm{Std},s)$ when $p$ is split in terms of our local integral.  This motivates making the following hybrid definition of $L$-factors for the purposes of stating the fine regulator formula.
\begin{defin}\label{defin:hybridl}
We define
\[L_{p}(\pi_p \times \nu_{1,p},\mathrm{Std},s)=\begin{cases} L_p^\mathrm{L}(\pi_p \times \nu_{1,p},\mathrm{Std},s) & p \textrm{ split or ramified} \\L_p^\mathrm{G}(\pi_p \times \nu_{1,p},\mathrm{Std},s)&$p$\textrm{ inert}. \end{cases}\]
\end{defin}
We accordingly adjust the definitions of $L(\pi\times\nu_1, \mathrm{Std},s)$, $\widehat{L}(\pi\times\nu_1, \mathrm{Std},s)$, and $L^{(1)}(\pi\times\nu_1, \mathrm{Std})$ in the following discussion.

Let $\Sigma$ be the set of split primes whose local representation is of type either (1) or (2) of Theorem \ref{thm:finesplit}.  At these places, write $\alpha_p\in \mathbf{C}^\times$ for the value such that $I(\Phi_n,\varphi_n,\nu,s)=(1-\alpha_pp^{-s})L_p^\mathrm{L}(\pi_p \times \nu_{1,p},\mathrm{Std},s)$.  These values were computed at the end of Section \ref{subsubsec:finesplit}.

In this situation, we may prove the following formula.
\begin{thm}[Class number formula, fine form] \label{thm:finecnf}
Let $\varphi_f \in V_{\ol{\mathbf{Q}}}$ be a new vector at every inert or split place and $G(\mathbf{Z}_p)$-stable at every ramified place.  Let $\Phi_p$ be the characteristic function of $\mathbf{Z}_p \oplus \mathbf{Z}_p$ at places where either $p$ is ramified or $\pi_p$ and $\nu_{1,p}$ are unramified.  At the remaining places, write $n_p=n_p(\pi)$ for the conductor of $\pi$ and set $\Phi_p = \Phi_{n_p}$.  Then we have
\begin{equation} \label{eqn:finereg} I(\Phi,\varphi_f,\nu,0) = W(\varphi_f,\pi)8iW_{0,0}(8\sqrt{2}\pi D^{-\frac{3}{4}})^{-1} \pi^4 D^{\frac{-3}{2}}\prod_{p \in \Sigma}(1-\alpha_pp^{-s}) L^{(1)}(\pi\times\nu_1, \mathrm{Std}).\end{equation}
If $\nu_1^2\omega_\pi|_{Z(\mathbf{A})}$ is non-trivial or there is at least one split place where $\pi_p$ is ramified, then $I(\Phi,\varphi_f,\nu,0)$ may be interpreted as a regulator pairing.  The right-hand side is non-zero if $\pi$ is unramified at every place $p$ that is inert in $E$ and if, additionally, no $\alpha_p=1$.
\end{thm}

\begin{proof}
Consider the factorization of the global integral $I(\Phi,\varphi_f,\nu,s)$ in the range of convergence, and apply the computations of Section \ref{subsec:ramifiedlocal}, Section \ref{subsubsec:finesplit}, and Section \ref{subsubsec:fineinert} to obtain
\begin{align} \label{eqn:fineregproof} I(\Phi,\varphi_f,\nu,s) =& W(\varphi_f,\pi)8iW_{0,0}(8\sqrt{2}\pi D^{-\frac{3}{4}})^{-1} \pi^4 D^{\frac{3s-3}{2}}\Gamma_{\C}(s)\Gamma_{\C}(s+1)\Gamma_{\C}(s+1)\\\nonumber &\times \prod_{p \in \Sigma}(1-\alpha_pp^{-s})L(\pi\times\nu_1, \mathrm{Std},s)\end{align}
for $\mathrm{Re}(s) \gg 0$.  The function $\Gamma_\mathbf{C}(s)$ has a simple pole with residue 1 at $s=0$. Since the left-hand side is holomorphic at 0 by our hypotheses on $\pi$, it follows that $L(\pi\times\nu_1, \mathrm{Std},s)$ vanishes to order at least 1, and all other terms are holomorphic.  This yields (\ref{eqn:finereg}).

The statement concerning the regulator pairing is Theorem \ref{thm:geom} once we point out that $\Phi_n(0)$ is 0 at split places with nonzero conductor.

Assume that for every $p$ inert in $E$, $\pi$ is unramified at $p$.  By Theorem \ref{thm:unramgcd}, Definition \ref{defin:hybridl}, and Theorem \ref{thm:morel}, the $L$-function $L(\pi\times\nu_1, \mathrm{Std},s)$ is the same as $L(\tau,\mathrm{Std},s)$ for a cuspidal automorphic representation $\tau$ on $\mathrm{Res}_\mathbf{Q}^E(\GL_3 \times \mathbf{G}_m)$ whose central character is unitary.   In particular, it follows easily from Theorem \ref{thm:pnt} and the functional equation of $L(\tau,\mathrm{Std},s)$ that the order of vanishing at $s=0$ of $L(\tau,\mathrm{Std},s)$ is exactly 1.  If, additionally, each $\alpha_p \ne 1$, then the nonvanishing of the right-hand side of (\ref{eqn:finereg}) follows.
\end{proof}

\bibliography{integralRepnBib}
\bibliographystyle{abbrv}

\end{document}